%% file: AffineSchurnn.tex
\documentclass[11pt, reqno]{amsart}
\usepackage{latexsym ,exscale, enumerate,amscd,xcolor}
\usepackage{
  fullpage,
  amsmath,
  amsfonts,
  amssymb,
  times,
  color,
  hyphenat,
  pinlabel,
}

\usepackage{pstricks}
\usepackage[tiling]{pst-fill}
\usepackage[all]{xy}
\SelectTips{cm}{}

\usepackage{graphicx}

\DeclareMathOperator{\He}{\mathcal{H}}

\input{KLdiags.tex}

\input{KLdefs.tex}
\input{defs.tex}

\title{A diagrammatic categorification of the affine $q$-Schur algebra 
$\hat{\SD}(n,n)$ for $n\geq 3$}
\author{Marco Mackaay and Anne-Laure Thiel}

\thanks{The two authors were supported by the FCT - Funda\c c\~{a}o para a 
Ci\^{e}ncia e a Tecnologia, through project number PTDC/MAT/101503/2008, 
New Geometry and Topology.}
\date{}

\begin{document}
\begin{abstract}
This paper is a follow-up to~\cite{MTh}. 
In that paper we categorified the affine $q$-schur algebra 
$\hat{\SD}(n,r)$ for $2<r<n$, using a quotient of Khovanov and Lauda's 
categorification of ${\mathbf U}_q(\hat{\mathfrak sl}_n)$~\cite{KL1,KL2,KL3}. 
In this paper we categorify $\hat{\SD}(n,n)$ for $n\geq 3$, using an 
extension of the aforementioned quotient.
\end{abstract}
\maketitle
\paragraph*{Acknowledgements}
We thank Jie Du and Qiang Fu for helpful exchanges of emails on the affine 
quantum Schur algebras.  

\tableofcontents

\section{Introduction}
The affine $q$-Schur algebras were defined by Green~\cite{Gr} 
for any $n,r\geq 3$, and are quotients of 
${\mathbf U}_q(\hat{\mathfrak{sl}}_n)$ and 
${\mathbf U}_q(\hat{\mathfrak{gl}}_n)$ if $r<n$. In~\cite{MTh} 
we defined a quotient of 
Khovanov and Lauda's categorification $\mathcal{U}(\hat{\mathfrak{sl}}_n)$, 
denoted $\hat{\mathcal S}(n,r)$, and showed that the Grothendieck group of 
its Karoubi envelope (idempotent completion) was exactly isomorphic 
to $\hat{\SD}(n,r)$ for $2<r<n$. In order to establish the isomorphism, we 
used Doty and Green's~\cite{DG} idempotented presentation of 
$\hat{\SD}(n,r)$ for $2<r<n$.

The case addressed in this paper is slightly more complicated, because 
$\hat{\SD}(n,n)$ is not a quotient of 
${\mathbf U}_q(\hat{\mathfrak{sl}}_n)$ or ${\mathbf U}_q(\hat{\mathfrak{gl}}_n)$ 
but of the strictly larger algebra $\hat{\mathbf U}_q(\hat{\mathfrak{gl}}_n)$, 
called the {\em extended} affine general linear quantum algebra and 
also due to Green~\cite{Gr}. Therefore, we have to extend the 
Khovanov-Lauda calculus of the corresponding 
quotient of $\mathcal{U}(\hat{\mathfrak{sl}}_n)$ by adding certain 
generating $1$ and $2$-morphisms and relations. We denote that extended 
$2$--category by $\hat{\mathcal S}(n,n)$ and show that 
the Grothendieck group of its Karoubi envelope 
is isomorphic to $\hat{\SD}(n,n)$ for $n\geq 3$. For that isomorphism we 
use Deng, Du and Fu's~\cite{DDF} presentation of $\hat{\SD}(n,n)$, 
which extends Doty and Green's.  
   
A little warning should be made. 
The results in this paper are not sufficient to categorify 
$\hat{\mathbf U}_q(\hat{\mathfrak{gl}}_n)$ diagrammatically, because 
that would require a categorification of $\hat{\SD}(n,r)$ for $2<n<r$ too. 
However, no Drinfeld-Jimbo type presentation of $\hat{\SD}(n,r)$ is known 
in that case, so even on the decategorified level there is an open question 
that would need to be solved first. For more information on this problem, 
see Question 4.3.2 in~\cite{Gr} and Chapter 5 in~\cite{DDF}. 

There is another technical detail that we should explain beforehand. 
In~\cite{MTh} we introduced a new degree-two variable $y$ 
and a $y$-deformation of the relations in 
Khovanov and Lauda's ${\mathcal U}(\hat{\mathfrak{sl}}_n)$, 
denoted ${\mathcal U}(\hat{\mathfrak{sl}}_n)_{[y]}$. 
The corresponding Schur quotients were denoted $\hat{\mathcal S}(n,r)_{[y]}$. 
This $y$-deformation was introduced in order to establish 
a precise relation between $\hat{\mathcal S}(n,r)_{[y]}$ and an extension of 
the affine singular Soergel bimodules built from 
Soergel's reflection faithful representation of the affine Weyl group, 
which were defined and studied by Williamson~\cite{Wi}. 
However, we also proved that the ideals generated by $y$ are virtually 
nilpotent, so that the Grothendieck groups of 
${\mathcal U}(\hat{\mathfrak{sl}}_n)_{[y]}$ and $\hat{\mathcal S}(n,r)_{[y]}$ 
are isomorphic to those of ${\mathcal U}(\hat{\mathfrak{sl}}_n)$ and 
$\hat{\mathcal S}(n,r)$. Furthermore, for $y=0$ the $2$-representations 
in~\cite{MTh} give $2$-functors from $\mathcal{U}(\hat{\mathfrak{sl}}_n)$ 
to certain extensions of the affine Soergel bimodules built from 
the geometric representation of the affine Weyl group, which is not 
reflection faithful but still has some nice properties (for more information 
on this topic, see Section 3.1 in~\cite{EW} and the results in~\cite{Li1}). 
In order to keep the calculations simple in this paper, we 
put $y=0$ here. It would not be hard to give the $y$-deformed relations in the 
definition of $\hat{\mathcal S}(n,n)$, which would give a 
$2$-category $\hat{\mathcal S}(n,n)_{[y]}$, but some of 
the subsequent calculations would be much harder in the $y$-deformed 
setting, e.g. the ones in the proof of 
Proposition~\ref{prop:polidigonreswithdots}. 

In general, it would be interesting to know more about the relation 
between $\hat{\mathcal S}(n,r)$, for $n\geq r$, and its $y$-deformation 
and the $2$-category of affine singular Soergel bimodules. 

Knowing more about this relation might also help to establish a 
connection with the work by Lusztig~\cite{LuAff} and 
Ginzburg and Vasserot~\cite{GV} on perverse sheaves and affine quantum 
$\mathfrak{gl}_n$. 

Finally, we thank Ben Webster for sharing with us his ideas on 
a potentially interesting connection between our results in this paper and 
his work in~\cite{Web2}, the precise formulation of which remains to be 
worked out. 

\section{affine quantum algebras}
In this section, we first recall the definition of the extended affine quantum 
general linear algebra $\hat{\mathbf U}_q(\hat{\mathfrak{gl}}_n)$ and 
its subalgebras ${\mathbf U}_q(\hat{\mathfrak{gl}}_n)$ and 
${\mathbf U}_q(\hat{\mathfrak{sl}}_n)$. After that, we recall the definition of 
the affine quantum Schur algebras $\hat{\SD}(n,r)$, due to Green~\cite{Gr}. 
Furthermore, we recall an idempotented presentation of the affine 
quantum Schur algebras, due to Doty and Green~\cite{DG} for 
$n>r$ and to Deng, Du and Fu~\cite{DDF} for $n=r$. 

\subsection{The (extended) affine quantum general and special linear algebras}
For the rest of this paper, let $n\geq 3$.

Since in this paper we are only interested in the affine quantum general and special linear algebras at level $0$, 
i.e. the $q$-analogue of the loop algebras without central extension, 
we can work with the normal $\mathfrak{gl}_n$-weight lattice, which is isomorphic to $\bZ^n$. Let 
$\epsilon_i=(0,\ldots,1,\ldots,0)\in \bZ^n$, with $1$ being on the $i$th 
coordinate, and $\alpha_i=\epsilon_i-\epsilon_{i+1}\in\bZ^{n}$, for 
$i=1,\ldots,n$, where the subscripts have to be understood modulo $n$, e.g. $\alpha_n=\epsilon_n-\epsilon_1=(-1,0,\ldots,0,1)$. 
We also define the Euclidean inner product on $\bZ^n$ by  
$\left< \epsilon_i,\epsilon_j\right>=\delta_{i,j}$. 
   
\begin{defn}\cite{Gr} The {\em extended quantum general linear algebra} 
$\hat{\mathbf U}_q(\hat{\mathfrak{gl}}_n)$ is 
the associative unital $\bQ(q)$-algebra generated by $R^{\pm 1}$, $K_i^{\pm 1}$ and $E_{\pm i}$, for $i=1,\ldots, n$, subject to the relations
\begin{gather}
K_iK_j=K_jK_i\quad K_iK_i^{-1}=K_i^{-1}K_i=1
\\
E_iE_{-j} - E_{-j}E_i = \delta_{i,j}\dfrac{K_iK_{i+1}^{-1}-K_i^{-1}K_{i+1}}{q-q^{-1}}
\\
K_iE_{\pm j}=q^{\pm \left< \epsilon_i,\alpha_j \right>}E_{\pm j}K_i
\\
E_{\pm i}^2E_{\pm j}-(q+q^{-1})E_{\pm i}E_{\pm j}E_{\pm i}+E_{\pm j}E_{\pm i}^2=0
\qquad\text{if}\quad |i-j|=1\mod n
\\
E_{\pm i}E_{\pm j}-E_{\pm j}E_{\pm i}=0\qquad\text{else}\\
RR^{-1}=R^{-1}R=1\\
RX_iR^{-1}=X_{i+1}\qquad\text{for}\;X_i\in\{E_{\pm i},K_i^{-1}\}.
\end{gather} 
In all equations, the subscripts have to be read modulo $n$.
\end{defn}

\begin{defn} 
\label{defn:qsln}
The {\em affine quantum general linear algebra} 
${\mathbf U}_q(\hat{\mathfrak{gl}}_n)\subseteq \hat{\mathbf U}_q(\hat{\mathfrak{gl}}_n)$ is 
the unital $\bQ(q)$-subalgebra generated by $E_{\pm i}$ and $K_i^{\pm 1}$, for $i=1,\ldots,n$. 

The {\em affine quantum special linear algebra} 
${\mathbf U}_q(\hat{\mathfrak{sl}}_n)\subseteq {\mathbf U}_q(\hat{\mathfrak{gl}}_n)$
is the unital $\bQ(q)$-subalgebra generated by 
$E_{\pm i}$ and $K_iK^{-1}_{i+1}$, for $i=1,\ldots, n$.
\end{defn}

\begin{rem}
A little warning about the notation is needed here. 
Our notation follows that of~\cite{DG,Gr}, which differs from that of
~\cite{DDF}. What we call ${\mathbf U}_q(\hat{\mathfrak{gl}}_n)$, 
Deng, Du and Fu call $U_{\Delta}(n)$. In Remark 5.3.2~\cite{DDF} 
they define $\hat{\mathbf U}$, which is equal to our 
$\hat{\mathbf U}_q(\hat{\mathfrak{gl}}_n)$. Finally, 
their ${\mathbf U}(\hat{\mathfrak{gl}}_n)$ is the {\em quantum loop algebra} 
(see their Definition 2.3.1), which contains 
$U_{\Delta}(n)$, i.e. our ${\mathbf U}_q(\hat{\mathfrak{gl}}_n)$, as a proper 
subalgebra. In their notation, $\hat{\mathbf U}$ 
is not a subalgebra of ${\mathbf U}(\hat{\mathfrak{gl}}_n)$, 
because $R\in\hat{\mathbf U}$ would have to be equal to 
an infinite linear combination of generators of the latter. 
\end{rem}

We will also need the bialgebra structure on 
$\hat{\mathbf U}_q(\hat{\mathfrak{gl}}_n)$.
\begin{defn}\cite{Gr} $\hat{\mathbf U}_q(\hat{\mathfrak{gl}}_n)$ is a bialgebra with 
counit 
$\epsilon\colon \hat{\mathbf U}_q(\hat{\mathfrak{gl}}_n)\to \Q(q)$
defined by 
$$\epsilon(E_{\pm i})=0,\qquad\epsilon(R^{\pm 1})=\epsilon(K_i^{\pm 1})=1$$
and coproduct 
$\Delta\colon \hat{\mathbf U}_q(\hat{\mathfrak{gl}}_n)\to 
\hat{\mathbf U}_q(\hat{\mathfrak{gl}}_n)\otimes \hat{\mathbf U}_q(\hat{\mathfrak{gl}}_n),$ 
defined by 
\begin{eqnarray}
\Delta(1)&=&1\otimes 1\\
\Delta(E_i)&=&E_i\otimes K_iK_{i+1}^{-1}+1\otimes E_i\\
\Delta(E_{-i})&=&K_i^{-1}K_{i+1}\otimes E_{-i}+E_{-i}\otimes 1\\
\Delta(K_i^{\pm 1})&=&K_i^{\pm 1}\otimes K_i^{\pm 1}\\
\Delta(R^{\pm 1})&=&R^{\pm 1}\otimes R^{\pm 1}.
\end{eqnarray}
\end{defn} 
As a matter of fact, $\hat{\mathbf U}_q(\hat{\mathfrak{gl}}_n)$ is 
even a Hopf algebra, but we do not need the antipode in this paper. Note 
that $\Delta$ and $\epsilon$ can be restricted to 
${\mathbf U}_q(\hat{\mathfrak{gl}}_n)$ and 
${\mathbf U}_q(\hat{\mathfrak{sl}}_n)$, which are bialgebras too. 

At level 0, we can also work with the ${\mathbf U}_q(\mathfrak{sl}_n)$-weight 
lattice, which is 
isomorphic to $\bZ^{n-1}$. Suppose that $V$ is a 
${\mathbf U}_q(\hat{\mathfrak{gl}}_n)$-weight representation with 
weights $\lambda=(\lambda_1,\ldots,\lambda_n)\in\bZ^n$, i.e. 
$$V\cong \bigoplus_{\lambda}V_{\lambda}$$ 
and $K_i$ acts as multiplication by 
$q^{\lambda_i}$ on $V_{\lambda}$. Then $V$ is also a 
${\mathbf U}_q(\hat{\mathfrak{sl}}_n)$-weight representation with weights 
$\overline{\lambda}=(\overline{\lambda}_1,\ldots,\overline{\lambda}_{n-1})\in
\bZ^{n-1}$ such that 
$\overline{\lambda}_j=\lambda_j-\lambda_{j+1}$ for $j=1,\ldots,n-1$. 
Conversely, given a ${\mathbf U}_q(\hat{\mathfrak{sl}}_n)$-weight 
representation with weights $\mu=(\mu_1,\ldots,\mu_{n-1})$, there is not a 
unique choice of ${\mathbf U}_q(\hat{\mathfrak{gl}}_n)$-action on $V$. We can 
fix this by choosing the action of $K_1\cdots K_n$. In terms of weights, this 
corresponds to the observation that, for any $r\in\bZ$ the equations 
\begin{align}
\label{eq:sl-gl-wts1}
\lambda_i-\lambda_{i+1}&=\mu_i\\
\label{eq:sl-gl-wts2}
\qquad \sum_{i=1}^{n}\lambda_i&=r
\end{align}  
determine $\lambda=(\lambda_1,\ldots,\lambda_n)$ uniquely, 
if there exists a solution to~\eqref{eq:sl-gl-wts1} and~\eqref{eq:sl-gl-wts2} 
at all. To fix notation, we 
define the map $\phi_{n,r}\colon \bZ^{n-1}\to \bZ^{n}\cup \{*\}$ by 
\begin{equation}
\label{eq:psi}
\phi_{n,r}(\mu)=\lambda 
\end{equation}
if~\eqref{eq:sl-gl-wts1} and \eqref{eq:sl-gl-wts2} have a solution, and 
put $\phi_{n,r}(\mu)=*$ otherwise. This map already appeared in~\cite{MTh} 
and~\cite{MSVschur}.  

As far as weight representations are concerned, we can restrict our attention 
to the Beilinson-Lusztig-MacPherson~\cite{BLM} idempotented version of these 
quantum groups, denoted $\Uglaffext$, $\Uglaff$ and $\Uaff$ respectively. 
To understand their definition, recall that $K_i$ acts as $q^{\lambda_i}$ on the 
$\lambda$-weight space of any weight representation. 
For each $\lambda\in\bZ^n$ adjoin an idempotent $1_{\lambda}$ to 
$\hat{\mathbf U}_q(\hat{\mathfrak{gl}}_n)$ and add 
the relations
\begin{align*}
1_{\lambda}1_{\mu} &= \delta_{\lambda,\mu}1_{\lambda}   
\\
E_{\pm i}1_{\lambda} &= 1_{\lambda\pm\alpha_i}E_{\pm i}
\\
K_i1_{\lambda} &= q^{\lambda_i}1_{\lambda}
\\
R1_{(\lambda_1,\ldots,\lambda_n)}&= 1_{(\lambda_n,\lambda_1,\ldots,\lambda_{n-1})}R.
\end{align*}
\begin{defn} 
\label{defn:Uglndot}
The {\em idempotented extended affine quantum general linear algebra} is defined by 
$$\Uglaffext=\bigoplus_{\lambda,\mu\in\bZ^n}1_{\lambda}\hat{\mathbf U}_q(\hat{\mathfrak{gl}}_n)1_{\mu}.$$
\end{defn}
\noindent Of course one defines $\Uglaff\subset\Uglaffext$ as the 
idempotented subalgebra generated by $1_{\lambda}$ and $E_{\pm i}1_{\lambda}$, 
for $i=1,\ldots, n$ and $\lambda\in\mathbb{Z}^n$. Similarly 
for $\hat{\mathbf U}_q(\mathfrak{sl}_n)$, adjoin an idempotent $1_{\lambda}$ 
for each $\lambda\in\bZ^{n-1}$ and add the relations
\begin{align*}
1_{\lambda}1_{\mu} &= \delta_{\lambda,\mu}1_{\lambda}   
\\
E_{\pm i}1_{\lambda} &= 1_{\lambda\pm\alpha_i}E_{\pm i}
\\
K_iK^{-1}_{i+1}1_{\lambda} &= q^{\lambda_i}1_{\lambda}.
\end{align*}
\begin{defn} The {\em idempotented quantum special linear algebra} is defined by 
$$\Uaff=\bigoplus_{\lambda,\mu\in\bZ^{n-1}}1_{\lambda}{\mathbf U}_q(\hat{\mathfrak{sl}}_n)1_{\mu}.$$
\end{defn}

Just to fix notation for future use. 
\begin{notat} For $\ii=(\mu_1 i_1,\ldots,\mu_{m}i_{m})$, with 
$\mu_j=\pm$, define $$E_{\ii}:=E_{\mu_1 i_1}\cdots E_{\mu_{m} i_{m}}$$ 
and define $\ii_{\Lambda}\in\bZ^n$ to be the $n$-tuple such that 
$$E_{\ii}1_{\lambda}=1_{\lambda + \ii_{\Lambda}}E_{\ii}.$$
Following Khovanov and Lauda~\cite{KL1,KL2,KL3}, 
we call $\ii$ a {\em signed sequence} 
and denote the set of signed sequences by $\mathrm{SSeq}$. 
\end{notat}

\subsection{The affine $q$-Schur algebra}

As we did in~\cite{MTh}, we first copy some facts about 
the action of $\hat{\mathbf U}_q(\hat{\mathfrak{gl}}_n)$ on tensor space 
from~\cite{DG, Gr}. After that we define the quotient 
$\hat{\SD}(n,r)$, for $n\geq r$, and give a presentation of that algebra. 
Note that the case $n=r$ was not considered in~\cite{MTh}.   
 
\subsubsection{Tensor space}
Let $V$ be the $\mathbb{Q}(q)$-vector space freely generated by 
$\{e_t\mid t\in\mathbb{Z}\}$. 
\begin{defn}\cite{Gr}
The following defines an action of $\hat{\mathbf U}_q(\hat{\mathfrak{gl}}_n)$ on $V$
\begin{eqnarray}
E_ie_{t+1}=e_t&\text{if}\; i\equiv t\mod n\\
E_ie_{t+1}=0&\text{if}\; i\not\equiv t\mod n\\
E_{-i}e_t=e_{t+1}&\text{if}\; i\equiv t\mod n\\
E_{-i}e_t=0&\text{if}\; i\not\equiv t\mod n\\
K_i^{\pm 1}e_t=q^{\pm 1}e_t&\text{if}\; i\equiv t\mod n\\
K_i^{\pm 1}e_t=e_t&\text{if}\; i\not\equiv t\mod n\\
R^{\pm 1}e_t=e_{t\pm 1}&\text{for all}\; t\in\mathbb{Z}.
\end{eqnarray}
\end{defn} 
\noindent Note that $V$ is clearly a weight-representation 
of $\hat{{\mathbf U}}_q(\hat{\mathfrak{gl}}_n)$, with 
$e_t$ having weight equal to $\epsilon_i$, for $i\equiv t \mod n$. 
Therefore $V$ is also a representation of $\Uglaffext$. Let 
$r\in\mathbb{N}_{>0}$ be arbitrary but fixed. 
As usual, one extends the above action to $V^{\otimes r}$, 
using the coproduct in $\hat{\mathbf U}_q(\hat{\mathfrak{gl}}_n)$. Again, 
this is a weight-representation, and therefore also a 
representation of $\Uglaffext$. There is also a right 
action of the extended affine Hecke algebra $\hat{\He}_{\hat{A}_{r-1}}$ on $V^{\otimes r}$, 
whose precise definition is not relevant here, 
which commutes with the left action of $\hat{\mathbf U}_q(\hat{\mathfrak{gl}}_n)$. 

\begin{defn}\cite{Gr}
The {\em affine $q$-Schur algebra} $\hat{\SD}(n,r)$ is by definition the centralizing algebra $$\mbox{End}_{\hat{\He}_{\hat{A}_{r-1}}}(V^{\otimes r}).$$ 
\end{defn}
\noindent It turns out that the image of 
the representation 
$\psi_{n,r}\colon \hat{\mathbf U}_q(\hat{\mathfrak{gl}}_n)\to 
\mbox{End}(V^{\otimes r})$ 
is isomorphic to $\hat{\SD}(n,r)$. If $n>r$, then we can even restrict to 
${\mathbf U}_q(\hat{\mathfrak{sl}}_n)\subset 
\hat{\mathbf U}_q(\hat{\mathfrak{gl}}_n)$, i.e.  
$$\psi_{n,r}({\mathbf U}_q(\hat{\mathfrak{sl}}_n))\cong \hat{\SD}(n,r).$$
If $n=r$, this is no longer true, as we will show below.   
 
\subsubsection{Presentation of $\hat{\SD}(n,r)$ for $n>r$}
In this subsection, let $n>r$. As already mentioned, the map 
$$\psi_{n,r}\colon \Uglaff\to \mbox{End}(V^{\otimes r})\to \hat{\SD}(n,r)$$ 
is surjective. This observation gives rise to the following presentation of 
$\hat{\SD}(n,r)$. The proof can be found 
in~\cite{DG} (Theorem 2.6.1).
\begin{thm}\cite{DG}
\label{thm:presentationnd}
For $n>r$, the $\bQ(q)$-algebra $\hat{\SD}(n,r)$ is isomorphic to 
the associative unital $\bQ(q)$-algebra generated by $1_{\lambda}$ and $E_{\pm i}$, for $\lambda\in\Lambda(n,r)$ and $i=1,\ldots, n$,  
subject to the relations
\begin{gather}
\label{rel1}
1_{\lambda}1_{\mu} = \delta_{\lambda,\mu}1_{\lambda} \\
\label{rel2}
E_{\pm i}1_{\lambda} = 1_{\lambda\pm\alpha_i}E_{\pm i}\\
\label{rel3}
(E_iE_{-j} - E_{-j}E_i)1_{\lambda} = \delta_{i,j}[\lambda_i-\lambda_{i+1}]1_{\lambda}
\\
\label{rel4}
(E_{\pm i}^2E_{\pm j}-(q+q^{-1})E_{\pm i}E_{\pm j}E_{\pm i}+E_{\pm j}E_{\pm i}^2)1_{\lambda}=0
\qquad\text{if}\quad |i-j|=1\mod n, 
\\
\label{rel5}
(E_{\pm i}E_{\pm j}-E_{\pm j}E_{\pm i})1_{\lambda}=0\qquad\text{else}.
\end{gather} 

In all equations the subscripts $i,j$ have to be read modulo $n$, and 
the equations hold for any $\lambda\in\Lambda(n,r)$. If $\lambda\pm\alpha_i
\not\in\Lambda(n,r)$, the corresponding idempotent is zero by convention.   
\end{thm}

We can restrict $\psi_{n,r}$ even further and obtain a surjection 
$\psi_{n,r}\colon \Uaff\to \hat{\SD}(n,r)$, which can be given 
explicitly on the generators. 
For any $\lambda\in\mathbb{Z}^{n-1}$, we have 
$$
\psi_{n,r}(E_{\pm i}1_{\lambda})=E_{\pm i}1_{\phi_{n,r}(\lambda)},
$$
where $\phi_{n,r}\colon \mathbb{Z}^{n-1}\to \Lambda(n,r)\cup \{*\}$ is the map 
defined in~\eqref{eq:psi}. By convention, we put $1_{*}=0$. 

\subsubsection{Presentation of $\hat{\SD}(n,n)$}

A Drinfeld-Jimbo type presentation of $\hat{\SD}(n,n)$ is harder to get, 
because 
$$\psi_{n,n}({\mathbf U}_q(\hat{\mathfrak{sl}}_n))=
\psi_{n,n}({\mathbf U}_q(\hat{\mathfrak{gl}}_n))$$
is a proper subalgebra of $\hat{\SD}(n,n)$. 
Therefore Green~\cite{Gr} introduced $\hat{\mathbf U}_q(\hat{\mathfrak{gl}}_n))
$, which contains the new invertible element $R$, and proved 
that $\hat{\SD}(n,n)$ is a quotient of this extended algebra. 
As vector spaces, we get the following $\mathbb{Q}(q)$-linear isomorphism:
$$\hat{\SD}(n,n)\cong \psi_{n,n}({\mathbf U}_q(\hat{\mathfrak{sl}}_n))\oplus
\bigoplus_{t\ne 0}\mathbb{Q}[R^t,R^{-t}].$$ 
However, this is not an algebra isomorphism. In Theorem 5.3.5 
in~\cite{DDF} Deng, Du and Fu show which relations need to be added in 
order to get a presentation of the algebra $\hat{\SD}(n,n)$. Let us 
first recall a slightly different presentation obtained by 
adding two new elements, $E_{\pm\delta}$, instead of $R^{\pm 1}$. 
This presentation, also due to Deng, Du and Fu~\cite{DDF}, turns out to be 
easier to categorify. As in~\cite{MTh}, we write $1_n:=1_{(1^n)}$. Recall 
that the {\em divided powers} are defined by 
$$E_{\pm i}^{(a)}:=\dfrac{E_{\pm i}^{a}}{[a]!},\quad\text{for}\;i=1,\ldots,n.$$ 
\begin{thm}\cite{DDF}
\label{thm:Eextrarels}
The $\mathbb{Q}(q)$-algebra $\hat{\SD}(n,n)$ is generated by $E_{\pm\delta}$, $E_{\pm i}$ and 
$1_{\lambda}$, for $i=1,\ldots,n$ and $\lambda\in\Lambda(n,n)$, subject to the relations 
\ref{rel1} through \ref{rel5} together with  
\begin{enumerate}[i)]
\item $E_{\pm\delta}1_{\lambda}=1_{\lambda}E_{\pm\delta}=0$ for all $\lambda\ne (1^n)$;
\item $E_{\pm\delta}1_n=1_nE_{\pm\delta}$;
\item $E_{+\delta}E_{-\delta}1_n=E_{-\delta}E_{+\delta}1_n=1_n$;
\item $E_iE_{+\delta}1_n=E_i^{(2)}E_{i-1}\ldots E_1E_n\cdots E_{i+1}1_n$;
\item $1_nE_{+\delta}E_i=1_nE_{i-1}\ldots E_1E_n\cdots E_{i+1}E_i^{(2)}$;
\item $E_{-i}E_{+\delta}1_n=E_{i-1}\cdots E_1E_n\cdots E_{i+1}1_n$;
\item $1_nE_{+\delta}E_{-i}=1_nE_{i-1}\cdots E_1E_n\cdots E_{i+1}$;
\item $E_{-i}E_{-\delta}1_n=E_{-i}^{(2)}E_{-(i+1)}\cdots E_{-n}E_{-1}\cdots 
E_{-(i-1)}1_n$;
\item $1_nE_{-\delta}E_{-i}=1_nE_{-(i+1)}\cdots E_{-n}E_{-1}\cdots 
E_{-(i-1)}E_{-i}^{(2)}$;
\item $E_iE_{-\delta}1_n=E_{-(i+1)}\cdots E_{-n}E_{-1}\cdots E_{-(i-1)}1_n$;
\item $1_nE_{-\delta}E_i=1_nE_{-(i+1)}\cdots E_{-n}E_{-1}\cdots E_{-(i-1)}$, 
\end{enumerate}
for any $i=1,\ldots,n$.
\end{thm}
To see that Theorem~\ref{thm:Eextrarels} really gives a presentation 
of $\hat{\SD}(n,n)$, recall that Deng, Du and Fu give the following definition 
in (5.3.1.1) and (5.3.1.2) in~\cite{DDF} (They use the notation 
$\rho$ where we use $R$): 
\begin{defn} Define 
$$R^{-1}:=E_{+\delta}1_n+\sum_{i=1}^n \sum_{\substack{(a_1,\ldots,a_n)\in\Lambda(n,n)\\a_i=0}}E_{i-1}^{(a_{i-1})}\cdots E_1^{(a_1)}E_{n}^{(a_n)}\cdots E_{i+1}^{(a_{i+1})}
1_{(a_n,a_1,\ldots,a_{n-1})}$$
and 
$$R:=E_{-\delta}1_n+\sum_{i=1}^n \sum_{\substack{(a_1,\ldots,a_n)\in\Lambda(n,n)\\a_i=0}}
E_{-(i-1)}^{(a_{i-1})}\cdots E_{-1}^{(a_1)}E_{-n}^{(a_n)}\cdots E_{-(i+1)}^{(a_{i+1})}
1_{(a_1,\ldots,a_n)}.$$
\end{defn}
\noindent Then note that 
$$E_{i-1}^{(a_{i-1})}\cdots E_1^{(a_1)}E_{n}^{(a_n)}\cdots E_{i+1}^{(a_{i+1})}
1_{(a_n,a_1,\ldots,a_{n-1})}=1_{(a_1,\ldots,a_n)}E_{i-1}^{(a_{i-1})}\cdots E_1^{(a_1)}E_{n}^{(a_n)}\cdots E_{i+1}^{(a_{i+1})}$$
and 
$$E_{i-1}^{(a_{i-1})}\cdots E_1^{(a_1)}E_{n}^{(a_n)}\cdots E_{i+1}^{(a_{i+1})}1_{\lambda}=0$$
for all $\lambda\ne (a_n,a_1,\ldots,a_{n-1})$. Likewise, we have
$$
E_{-(i-1)}^{(a_{i-1})}\cdots E_{-1}^{(a_1)}E_{-n}^{(a_n)}\cdots E_{-(i+1)}^{(a_{i+1})}
1_{(a_1,\ldots,a_n)}=1_{(a_n,a_1,\ldots,a_{n-1})}
E_{-(i-1)}^{(a_{i-1})}\cdots E_{-1}^{(a_1)}E_{-n}^{(a_n)}\cdots E_{-(i+1)}^{(a_{i+1})}
$$
and 
$$E_{-(i-1)}^{(a_{i-1})}\cdots E_{-1}^{(a_1)}E_{-n}^{(a_n)}\cdots E_{-(i+1)}^{(a_{i+1})}1_{\lambda}=0$$
for all $\lambda\ne(a_1,\ldots,a_n)$. These remarks show 
that Proposition 5.3.3 and Corollary 5.3.4 in~\cite{DDF} 
imply that the presentation of $\hat{\SD}(n,n)$ in Theorem 5.3.5 in that paper, 
is equivalent to the one we have given in Theorem~\ref{thm:Eextrarels}. 
In particular, the relations in Theorem~\ref{thm:Eextrarels} imply 
the following relations, which are exactly the ones in 
Theorem 5.3.5~\cite{DDF}: 
\begin{cor} In $\hat{\SD}(n,n)$, we have 
$$RR^{-1}=R^{-1}R=1,\quad RE_{\pm i}R^{-1}=E_{\pm(i+1)},\quad R1_{\lambda}R^{-1}=
1_{(\lambda_n,\lambda_1\ldots,\lambda_{n-1})}.$$
As usual, we read the indices modulo $n$. 
\end{cor}
Therefore, the surjective algebra homomorphism
$$\psi_{n,n}\colon \Uglaffext\to \hat{\SD}(n,n)$$
can be defined as
$$
\psi_{n,n}(1_{\lambda})=
\begin{cases}
1_{\lambda}&\quad\text{if}\;\lambda\in\Lambda(n,n)\\
0&\quad\text{else}
\end{cases}
$$
and
$$
\psi_{n,n}(E_{\pm i}1_{\lambda})=E_{\pm i}\psi_{n,n}(1_{\lambda}),\quad \psi_{n,n}(R^{\pm 1}1_{\lambda})=R^{\pm 1}\psi_{n,n}(1_{\lambda}).
$$

In Lemma 3.2 and Corollary 5.6 in~\cite{DD} Deng and Du also show that 
there exists an embedding 
$$\iota_n\colon \hat{\SD}(n,n)\to \hat{\SD}(n+1,n),$$   
which gives an isomorphism of algebras 
$$\hat{\SD}(n,n)\cong \bigoplus_{\lambda,\mu\in\Lambda(n,n)}1_{(\lambda,0)}
\hat{\SD}(n+1,n)1_{(\mu,0)}.$$
At that point of their paper they use a different presentation of 
the affine $q$--Schur algebras, but by Proposition 7.1~\cite{DD} it is not 
hard to work out the image under $\iota_n$ of the 
generators of $\hat{\SD}(n,n)$ in Theorem~\ref{thm:Eextrarels}. 
Note that we have multiplied their images of $E_{+n}$ and $E_{-n}$ by $-1$, 
which is more convenient for categorification and does not invalidate 
their results. 
\begin{prop}\cite{DD}
\label{prop:iota}
The $\mathbb{Q}(q)$-linear algebra homomorphism 
$$\iota_{n}\colon \hat{\SD}(n,n)\to \hat{\SD}(n+1,n)$$ 
defined by
\begin{eqnarray*}
1_{\lambda}&\mapsto& 1_{(\lambda,0)}\\
E_{\pm i}1_{\lambda}&\mapsto& E_{\pm i}1_{(\lambda,0)}\\
E_n1_{\lambda}&\mapsto& E_nE_{n+1}1_{(\lambda,0)}\\
E_{-n}1_{\lambda}&\mapsto& E_{-(n+1)}E_{-n}1_{(\lambda,0)}\\
E_{+\delta}1_n&\mapsto& E_nE_{n-1}\cdots E_1E_{n+1}1_{(1^n,0)}\\
E_{-\delta}1_n&\mapsto& E_{-(n+1)}E_{-1}\cdots E_{-n}1_{(1^n,0)}
\end{eqnarray*}  
for any $1\leq i\leq n-1$ and $\lambda\in\Lambda(n,n)$,  
is an embedding and gives an isomorphism of algebras
$$\hat{\SD}(n,n)\cong \bigoplus_{\lambda,\mu\in\Lambda(n,n)}1_{(\lambda,0)}\hat{\SD}(n+1,n)1_{(\mu,0)}.$$
\end{prop}

\section{A diagrammatic categorification of $\hat{\SD}(n,n)$}

\begin{defn}
\label{defn:Scat}
The $2$-category $\hat{\mathcal{S}}(n,n)$ is defined as the quotient of 
$\mathcal{U}(\hat{\mathfrak{gl}}_n)$ by the ideal generated by all diagrams 
with regions whose labels are not contained in $\Lambda(n,n)$, 
just as in~\cite{MTh} (taking $y=0$ in that paper), 
together with the generating $1$-morphisms
$${\mathbf 1}_n\mathcal{E}_{+\delta}{\mathbf 1}_n\{t\}\quad\text{and}\quad 
{\mathbf 1}_n\mathcal{E}_{-\delta}{\mathbf 1}_n\{t\},$$
for $t\in\mathbb{Z}$, and the following generating $2$-morphisms 
\[
\begin{array}{ccc}
  1_{\cal{E}_{+\delta} {\mathbf 1}_n\{t\}} &\quad  & 1_{\cal{E}_{-\delta} {\mathbf 1}_n\{t\}} \\ \\
   \xy
 (0,0)*{\black\xybox{(0,8);(0,-8); **\dir{-} ?(.5)*\dir{>}+(2.3,0)*{\scriptstyle{}};}};
 (-1,-11)*{\scriptstyle\delta};(-1,11)*{\scriptstyle \delta};
 (6,2)*{ (1^n)};
 (-8,2)*{ (1^n)};
 (-10,0)*{};(10,0)*{};
 \endxy
 & &
 \;\;   
   \xy
 (0,0)*{\black\xybox{(0,8);(0,-8); **\dir{-} ?(.5)*\dir{<}+(2.3,0)*{\scriptstyle{}};}};
 (-1,-11)*{\scriptstyle\delta};(-1,11)*{\scriptstyle\delta};
 (6,2)*{ (1^n)};
 (-8.5,2)*{ (1^n)};
 (-12,0)*{};(12,0)*{};
 \endxy
\\ \\
   \;\;\text{ {\rm deg} 0}\;\;
 & &\;\;\text{ {\rm deg} 0}\;\;
\end{array}
\]

\[
\begin{tabular}{|l|c|c|c|c|}
\hline
  {\bf Notation:} \xy (0,-5)*{};(0,7)*{}; \endxy
 & \text{$\Ucupr_{\delta}$} 
 & \text{$\Ucupl_{\delta}$} 
 & \text{$\Ucapl_{\delta}$} 
 & \text{$\Ucapr_{\delta}$} \\
 \hline
  {\bf 2-morphism:} & \xy
    (0,-3)*{\black\bbpef{\black \delta}};
    (8,-3)*{(1^n)};
    (-12,0)*{};(12,0)*{};
    \endxy
  & \xy
    (0,-3)*{\black\bbpfe{\black \delta}};
    (8,-3)*{(1^n)};
    (-12,0)*{};(12,0)*{};
    \endxy
  & \xy
    (0,0)*{\black\bbcef{\black \delta}};
    (8,3)*{(1^n)};
    (-12,0)*{};(12,0)*{};
    \endxy 
  & \xy
    (0,0)*{\black\bbcfe{\black \delta}};
    (8,3)*{(1^n)};
    (-12,0)*{};(12,0)*{};(8,8)*{};
    \endxy\\& & &  &\\ \hline
 {\bf Degree:} & \;\;\text{  $0$}\;\;
 & \;\;\text{ $0$}\;\;
 & \;\;\text{ $0$}\;\;
 & \;\;\text{  $0$}\;\;
 \\
 \hline
\end{tabular}
\]

\[
\begin{tabular}{|l|c|c|c|c|}
\hline
  \hbox{\bf Notation:} \xy (0,-5)*{};(0,7)*{}; \endxy
 & \text{$\DeltaEE_{\delta,{\blue i}}$} 
 & \text{$\DeltaFF_{\delta,{\blue i}}$}
 & \text{$\EEDelta^{\delta,{\blue i}}$}
 & \text{$\FFDelta^{\delta,{\blue i}}$}
\\ 
  \hline
  \hbox{\bf 2-morphism:} \xy (0,-5)*{};(0,7)*{};\endxy & 
\xy
(-10,17)*{};
(-12,14)*{\scriptstyle\blue i};
(-8,14)*{\scriptstyle\blue i-1};
(-2,14)*{\scriptstyle\blue 1};
(2,14)*{\scriptstyle\blue n};
(7,14)*{\scriptstyle\blue i+2};
(13,14)*{\scriptstyle\blue i+1};
(0,-15)*{\scriptstyle \delta};
(7,0)*{(1^n)};
(0,0)*{\figs{0.27}{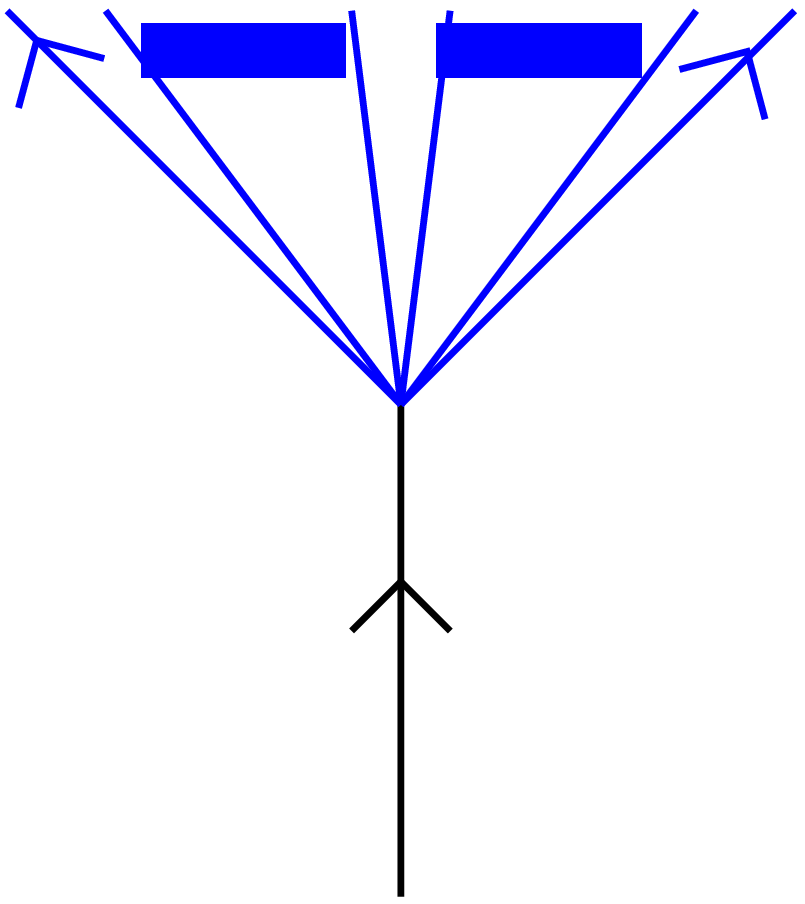}};
\endxy
  & 
\xy
(-10,17)*{};
(-12,14)*{\scriptstyle\blue i};
(-8,14)*{\scriptstyle\blue i+1};
(-2,14)*{\scriptstyle\blue n};
(2,14)*{\scriptstyle\blue 1};
(7,14)*{\scriptstyle\blue i-2};
(13,14)*{\scriptstyle\blue i-1};
(0,-15)*{\scriptstyle\delta};
(7,0)*{(1^n)};
(0,0)*{\figs{0.27}{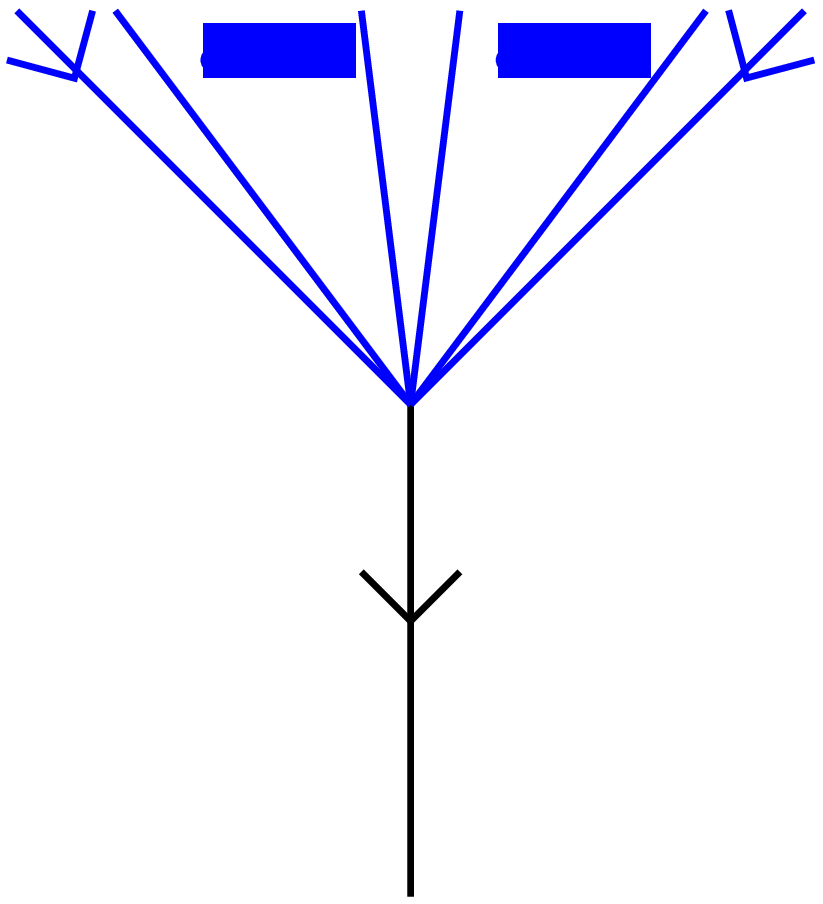}};
\endxy
&
\xy
(-10,17)*{};
(-12,-14)*{\scriptstyle\blue i};
(-8,-14)*{\scriptstyle\blue i-1};
(-2,-14)*{\scriptstyle\blue 1};
(2,-14)*{\scriptstyle\blue n};
(7,-14)*{\scriptstyle\blue i+2};
(13,-14)*{\scriptstyle\blue i+1};
(0,15)*{\scriptstyle\delta};
(7,0)*{(1^n)};
(0,0)*{\figs{0.27}{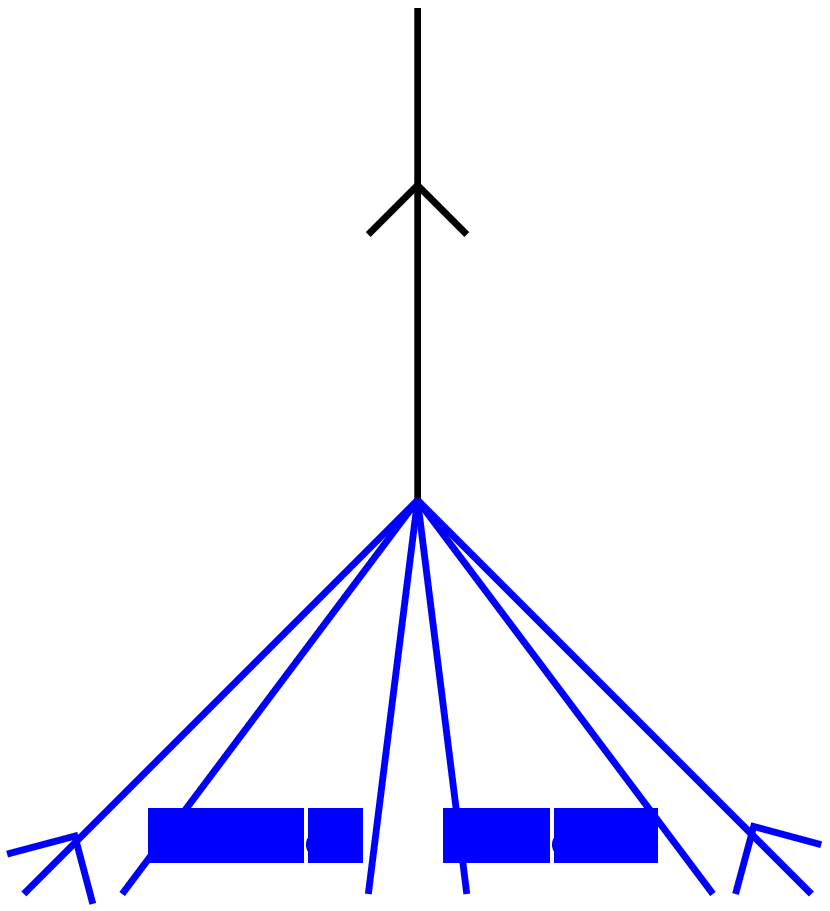}};
\endxy
&
\xy
(-10,17)*{};
(-12,-14)*{\scriptstyle\blue i};
(-8,-14)*{\scriptstyle\blue i+1};
(-2,-14)*{\scriptstyle\blue n};
(2,-14)*{\scriptstyle\blue 1};
(7,-14)*{\scriptstyle\blue i-2};
(13,-14)*{\scriptstyle\blue i-1};
(0,15)*{\scriptstyle\delta};
(7,0)*{(1^n)};
(0,0)*{\figs{0.27}{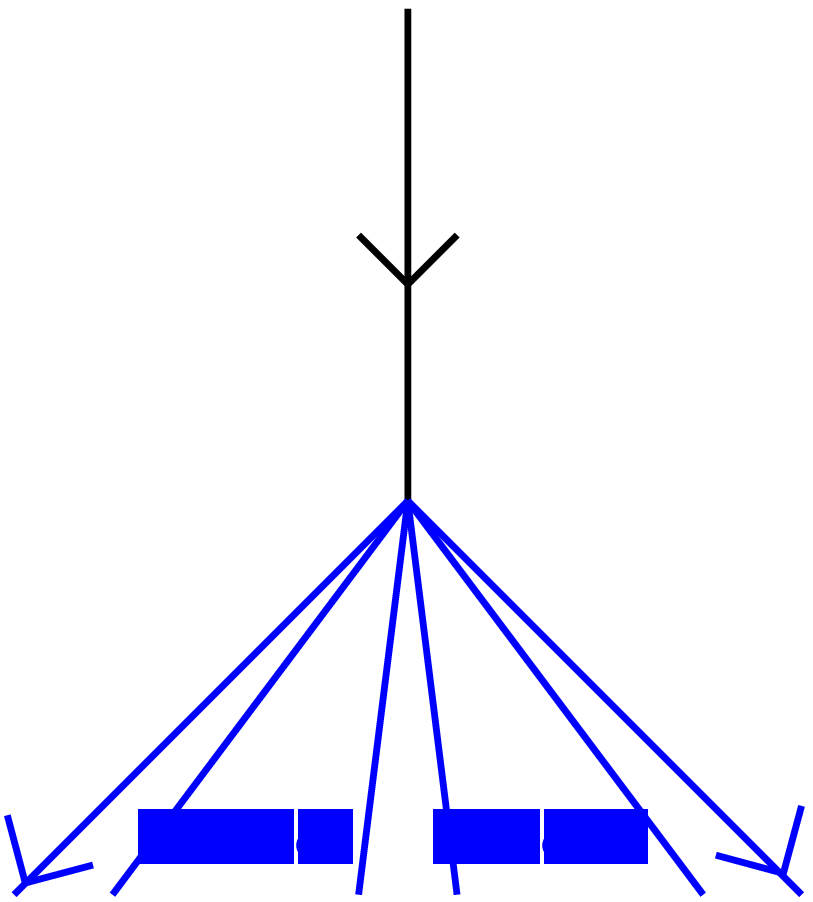}};
\endxy
   \\& & & & \\ \hline
 {\bf Degree:} 
& 
\;\;\text{$1$}\;\;
& 
\;\;\text{$1$}\;\;
& 
\;\;\text{$1$}\;\;
& 
\;\;\text{$1$}\;\;
\\
 \hline
\end{tabular}
\]
Relations:\\ \\
$\cal{E}_{+\delta}1_n$ and $\cal{E}_{-\delta}{\mathbf 1}_n$ are biadjoint inverses of 
eachother:
\begin{equation} 
\label{eq_biadjoint1}
\text{$
  \xy   0;/r.18pc/:
    (0,0)*{\black\xybox{
    (-8,0)*{}="1";
    (0,0)*{}="2";
    (8,0)*{}="3";
    (-8,-10);"1" **\dir{-};
    "1";"2" **\crv{(-8,8) & (0,8)} ?(0)*\dir{>} ?(1)*\dir{>};
    "2";"3" **\crv{(0,-8) & (8,-8)}?(1)*\dir{>};
    "3"; (8,10) **\dir{-};}};
    (12,-9)*{(1^n)};
    (-6,9)*{(1^n)};
    (-7,-13)*{\scriptstyle\delta};
    \endxy
    \; =
    \;
\xy   0;/r.18pc/:
    (0,0)*{\black\xybox{
    (-8,0)*{}="1";
    (0,0)*{}="2";
    (8,0)*{}="3";
    (0,-10);(0,10)**\dir{-} ?(.5)*\dir{>};}};
    (7,8)*{(1^n)};
    (-9,8)*{(1^n)};
    (0,-13)*{\scriptstyle\delta};
    \endxy
\qquad \quad  \xy   0;/r.18pc/:
    (0,0)*{\black\xybox{
    (-8,0)*{}="1";
    (0,0)*{}="2";
    (8,0)*{}="3";
    (-8,-10);"1" **\dir{-};
    "1";"2" **\crv{(-8,8) & (0,8)} ?(0)*\dir{<} ?(1)*\dir{<};
    "2";"3" **\crv{(0,-8) & (8,-8)}?(1)*\dir{<};
    "3"; (8,10) **\dir{-};}};
    (12,-9)*{(1^n)};
    (-6,9)*{(1^n)};
    (-7,-13)*{\scriptstyle\delta};
    \endxy
    \; =
    \;
\xy   0;/r.18pc/:
    (0,0)*{\black\xybox{
    (-8,0)*{}="1";
    (0,0)*{}="2";
    (8,0)*{}="3";
    (0,-10);(0,10)**\dir{-} ?(.5)*\dir{<};}};
    (9,8)*{(1^n)};
    (-6,8)*{(1^n)};
    (0,-13)*{\scriptstyle\delta};
    \endxy
$}
\end{equation}

\begin{equation}
\label{eq_biadjoint2}
\text{$
 \xy   0;/r.18pc/:
    (0,0)*{\black\xybox{
    (8,0)*{}="1";
    (0,0)*{}="2";
    (-8,0)*{}="3";
    (8,-10);"1" **\dir{-};
    "1";"2" **\crv{(8,8) & (0,8)} ?(0)*\dir{>} ?(1)*\dir{>};
    "2";"3" **\crv{(0,-8) & (-8,-8)}?(1)*\dir{>};
    "3"; (-8,10) **\dir{-};}};
    (12,9)*{(1^n)};
    (-5,-10)*{(1^n)};
    (8,-13)*{\scriptstyle\delta};
    \endxy
    \; =
    \;
      \xy 0;/r.18pc/:
    (0,0)*{\black\xybox{
    (8,0)*{}="1";
    (0,0)*{}="2";
    (-8,0)*{}="3";
    (0,-10);(0,10)**\dir{-} ?(.5)*\dir{>};}};
    (7,0)*{(1^n)};
    (-9,0)*{(1^n)};
    (0,-13)*{\scriptstyle\delta};
    \endxy
\qquad \quad \xy  0;/r.18pc/:
    (0,0)*{\black\xybox{
    (8,0)*{}="1";
    (0,0)*{}="2";
    (-8,0)*{}="3";
    (8,-10);"1" **\dir{-};
    "1";"2" **\crv{(8,8) & (0,8)} ?(0)*\dir{<} ?(1)*\dir{<};
    "2";"3" **\crv{(0,-8) & (-8,-8)}?(1)*\dir{<};
    "3"; (-8,10) **\dir{-};}};
    (14,9)*{(1^n)};
    (-5,-10)*{(1^n)};
    (8,-13)*{\scriptstyle\delta};
    \endxy
    \; =
    \;
\xy  0;/r.18pc/:
    (0,0)*{\black\xybox{
    (8,0)*{}="1";
    (0,0)*{}="2";
    (-8,0)*{}="3";
    (0,-10);(0,10)**\dir{-} ?(.5)*\dir{<};}};
    (9,0)*{(1^n)};
    (-7,0)*{(1^n)};
    (0,-13)*{\scriptstyle\delta};
    \endxy
$}
\end{equation}

\begin{equation} 
\label{deltabubbles}
 \xy
 (-12,0)*{\black\ncbub_{\delta}};
 (-4,5)*{(1^n)};
 \endxy
  = 
\xy
 (-12,0)*{\black\nccbub_{\delta}};
 (-4,5)*{(1^n)};
 \endxy = 1
\end{equation}

\begin{equation}
\label{deltainvert2}
\xy
    (0,3)*{\black\bbpef{\black \delta}};
    (8,0)*{(1^n)};
    (-12,0)*{};(12,0)*{};
    (0,-3)*{\black\bbcef{}};
    (-12,0)*{};(12,0)*{}; 
\endxy
=
\xy   
    (0,0)*{\black\xybox{
    (0,-6.5);(0,6.5)**\dir{-} ?(.5)*\dir{<};}};
    (0,-9.5)*{\scriptstyle\delta};
    (9,0)*{\black\xybox{
    (0,-6.5);(0,6.5)**\dir{-} ?(.5)*\dir{>};}};
    (15,0)*{(1^n)};
    (9,-9.5)*{\scriptstyle\delta};
    \endxy
\qquad
\xy
    (0,3)*{\black\bbpfe{\black \delta}};
    (8,0)*{(1^n)};
    (-12,0)*{};(12,0)*{};
    (0,-3)*{\black\bbcfe{}};
    (-12,0)*{};(12,0)*{};(8,8)*{};
    \endxy
=
\xy   
    (0,0)*{\black\xybox{
    (0,-6.5);(0,6.5)**\dir{-} ?(.5)*\dir{>};}};
    (0,-9.5)*{\scriptstyle\delta};
    (9,0)*{\black\xybox{
    (0,-6.5);(0,6.5)**\dir{-} ?(.5)*\dir{<};}};
    (15,0)*{(1^n)};
    (9,-9.5)*{\scriptstyle\delta};
    \endxy
\end{equation}

We impose full cyclicity w.r.t. $\DeltaEE_{\delta,\blue i}, \DeltaFF_{\delta,\blue i}, 
\EEDelta^{\delta,\blue i-1}$ and $\FFDelta^{\delta,\blue i+1}$, e.g. 
by using the adequate cups 
and caps we can rotate $\DeltaEE_{\delta,\blue i}$ to obtain $\FFDelta^{\delta,
\blue i+1}$.

Furthermore, we impose the relations:
\begin{equation}
\label{DeltaEEX}
\begin{array}{ccc}
\xy
(0,0)*{\figs{0.25}{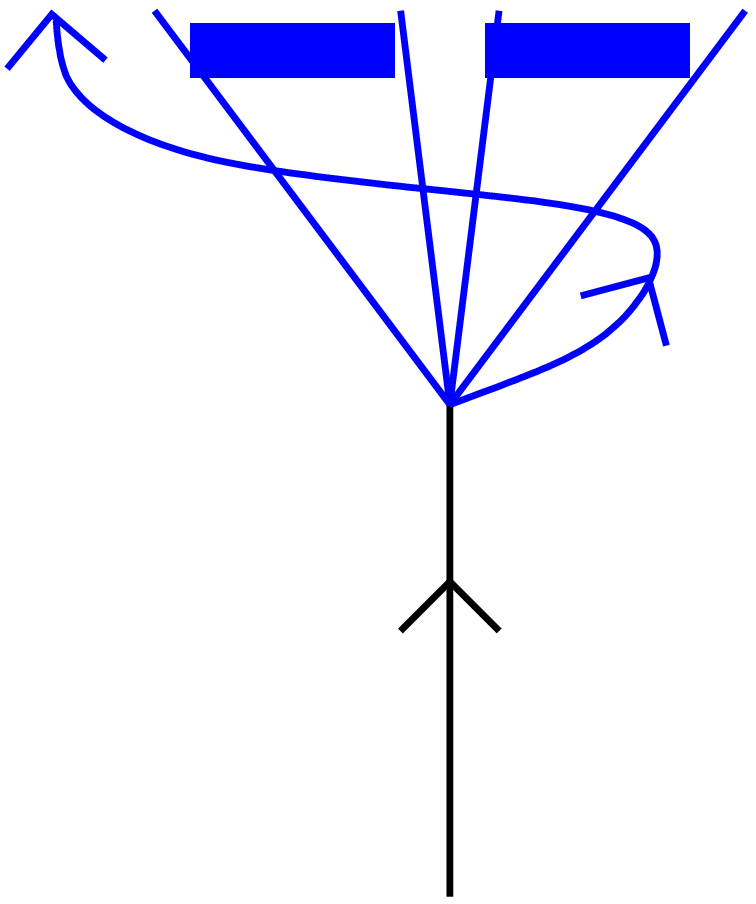}};
(12,0)*{(1^n)};
(-9,14)*{\scriptstyle\blue i+1};
(-5,14)*{\scriptstyle\blue i};
(0,14)*{\scriptstyle\blue 1};
(4,14)*{\scriptstyle\blue n};
(12,14)*{\scriptstyle\blue i+2};
(2,-15)*{\scriptstyle\delta};
\endxy
&=&
\xy
(-10,17)*{};
(-8,9)*{\blue\bullet};
(-12,14)*{\scriptstyle\blue i+1};
(-8,14)*{\scriptstyle\blue i};
(-2,14)*{\scriptstyle\blue 1};
(2,14)*{\scriptstyle\blue n};
(7,14)*{\scriptstyle\blue i+3};
(13,14)*{\scriptstyle\blue i+2};
(0,-15)*{\scriptstyle\delta};
(7,0)*{(1^n)};
(0,0)*{\figs{0.27}{DeltaEE}};
\endxy
-
\xy
(-10,17)*{};
(-6,9)*{\blue\bullet};
(-12,14)*{\scriptstyle\blue i+1};
(-8,14)*{\scriptstyle\blue i};
(-2,14)*{\scriptstyle\blue 1};
(2,14)*{\scriptstyle\blue n};
(7,14)*{\scriptstyle\blue i+3};
(13,14)*{\scriptstyle\blue i+2};
(0,-15)*{\scriptstyle\delta};
(7,0)*{(1^n)};
(0,0)*{\figs{0.27}{DeltaEE}};
\endxy
\end{array}
\end{equation}

\begin{equation}
\label{DeltaEERX}
\begin{array}{ccc}
\xy
(0,0)*{\figs{0.25}{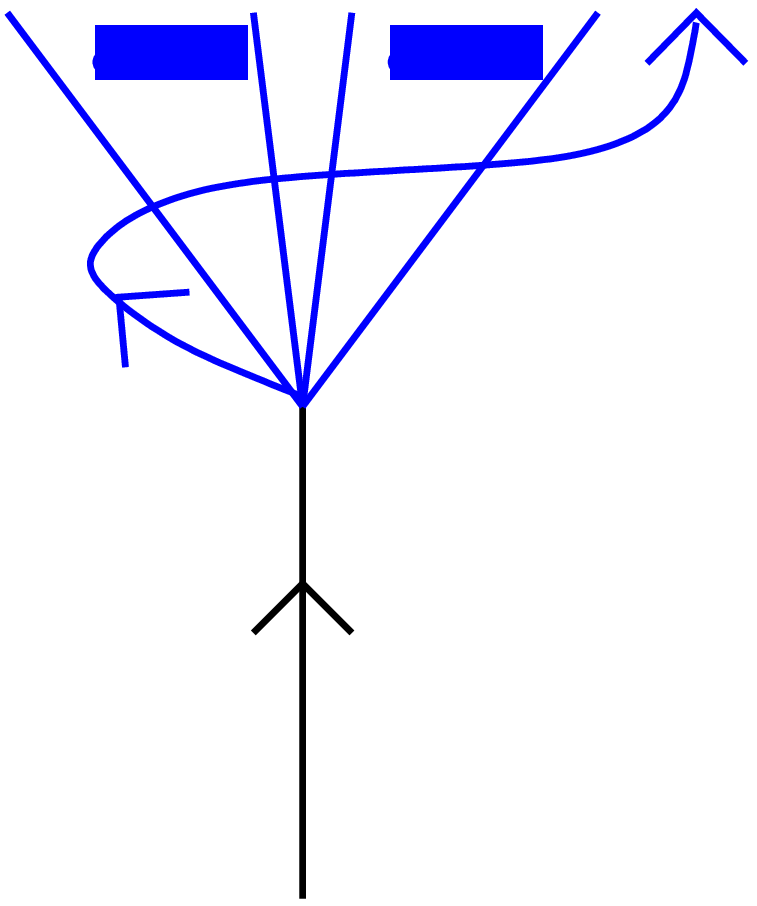}};
(12,0)*{(1^n)};
(-11,14)*{\scriptstyle\blue i-1};
(-4,14)*{\scriptstyle\blue 1};
(0,14)*{\scriptstyle\blue n};
(5,14)*{\scriptstyle\blue i+1};
(9,14)*{\scriptstyle\blue i};
(-2,-15)*{\scriptstyle\delta};
\endxy
&=&
\xy
(-10,17)*{};
(5.5,9)*{\blue\bullet};
(-12,14)*{\scriptstyle\blue i-1};
(-6.5,14)*{\scriptstyle\blue i-2};
(-2,14)*{\scriptstyle\blue 1};
(2,14)*{\scriptstyle\blue n};
(7,14)*{\scriptstyle\blue i+1};
(11,14)*{\scriptstyle\blue i};
(0,-15)*{\scriptstyle\delta};
(7,0)*{(1^n)};
(0,0)*{\figs{0.27}{DeltaEE}};
\endxy
\ -
\xy
(-10,17)*{};
(8,9)*{\blue\bullet};
(-12,14)*{\scriptstyle\blue i-1};
(-6.5,14)*{\scriptstyle\blue i-2};
(-2,14)*{\scriptstyle\blue 1};
(2,14)*{\scriptstyle\blue n};
(7,14)*{\scriptstyle\blue i+1};
(11,14)*{\scriptstyle\blue i};
(0,-15)*{\scriptstyle\delta};
(7,0)*{(1^n)};
(0,0)*{\figs{0.27}{DeltaEE}};
\endxy
\end{array}
\end{equation}

\begin{equation}
\label{LXEEDelta}
\begin{array}{ccc}
\xy
(0,0)*{\figs{0.25}{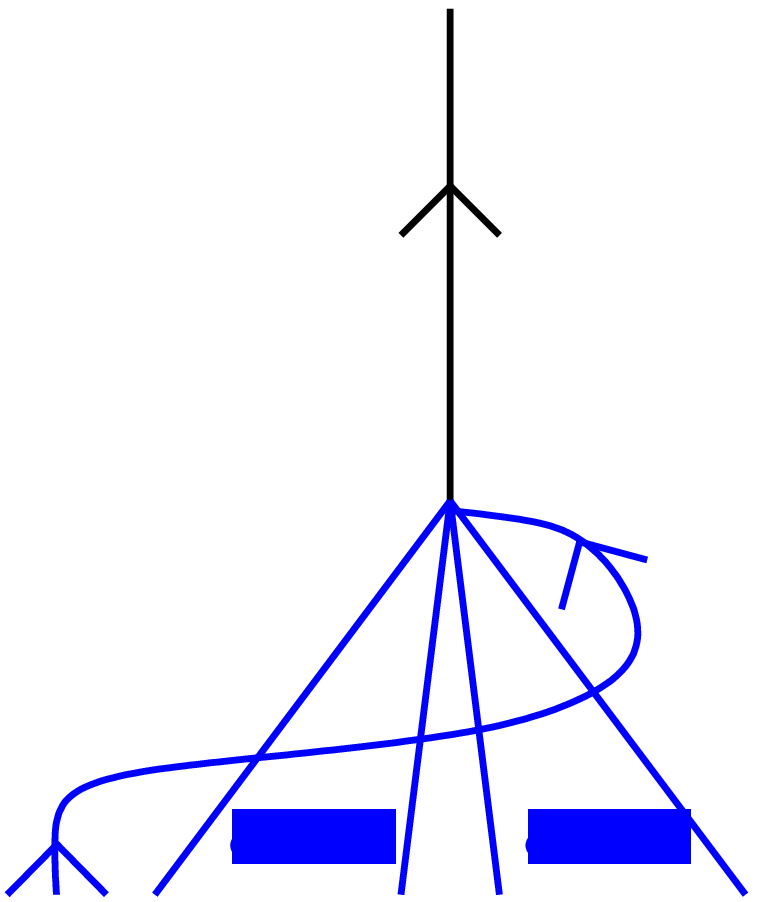}};
(12,0)*{(1^n)};
(-9,-14)*{\scriptstyle\blue i+1};
(-5,-14)*{\scriptstyle\blue i};
(0,-14)*{\scriptstyle\blue 1};
(4,-14)*{\scriptstyle\blue n};
(12,-14)*{\scriptstyle\blue i+2};
(2,15)*{\scriptstyle\delta};
\endxy
&=&
\xy
(-10,17)*{};
(-8,-9)*{\blue\bullet};
(-12,-14)*{\scriptstyle\blue i+1};
(-8,-14)*{\scriptstyle\blue i};
(-2,-14)*{\scriptstyle\blue 1};
(2,-14)*{\scriptstyle\blue n};
(7,-14)*{\scriptstyle\blue i+3};
(13,-14)*{\scriptstyle\blue i+2};
(0,15)*{\scriptstyle\delta};
(7,0)*{(1^n)};
(0,0)*{\figs{0.27}{EEDelta}};
\endxy
-
\xy
(-10,17)*{};
(-6,-9)*{\blue\bullet};
(-12,-14)*{\scriptstyle\blue i+1};
(-8,-14)*{\scriptstyle\blue i};
(-2,-14)*{\scriptstyle\blue 1};
(2,-14)*{\scriptstyle\blue n};
(7,-14)*{\scriptstyle\blue i+3};
(13,-14)*{\scriptstyle\blue i+2};
(0,15)*{\scriptstyle\delta};
(7,0)*{(1^n)};
(0,0)*{\figs{0.27}{EEDelta}};
\endxy
\end{array}
\end{equation}

\begin{equation}
\label{RXEEDelta}
\begin{array}{ccc}
\xy
(0,0)*{\figs{0.25}{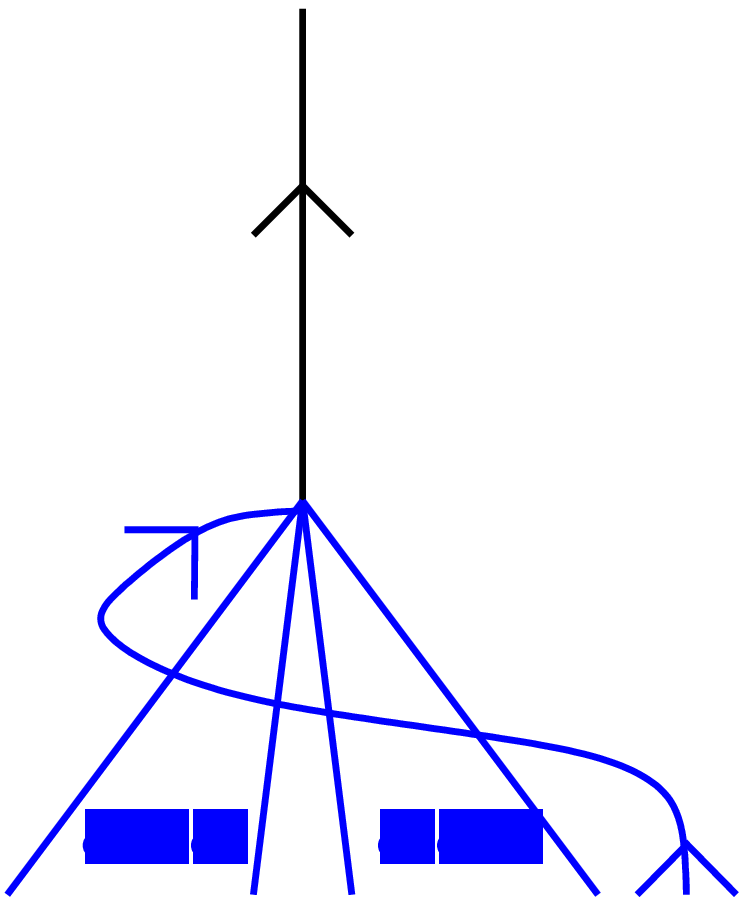}};
(12,0)*{(1^n)};
(-9.5,-14)*{\scriptstyle\blue i-1};
(-3,-14)*{\scriptstyle\blue 1};
(0,-14)*{\scriptstyle\blue n};
(5,-14)*{\scriptstyle\blue i+1};
(8.5,-14)*{\scriptstyle\blue i};
(-2,15)*{\scriptstyle\delta};
\endxy
&=&
\xy
(-10,17)*{};
(5.7,-9)*{\blue\bullet};
(-12,-14)*{\scriptstyle\blue i-1};
(-6.5,-14)*{\scriptstyle\blue i-2};
(-2,-14)*{\scriptstyle\blue 1};
(2,-14)*{\scriptstyle\blue n};
(7,-14)*{\scriptstyle\blue i+1};
(12,-14)*{\scriptstyle\blue i};
(0,15)*{\scriptstyle\delta};
(7,0)*{(1^n)};
(0,0)*{\figs{0.27}{EEDelta}};
\endxy
\ -
\xy
(-10,17)*{};
(7.8,-9)*{\blue\bullet};
(-12,-14)*{\scriptstyle\blue i-1};
(-6.5,-14)*{\scriptstyle\blue i-2};
(-2,-14)*{\scriptstyle\blue 1};
(2,-14)*{\scriptstyle\blue n};
(7,-14)*{\scriptstyle\blue i+1};
(12,-14)*{\scriptstyle\blue i};
(0,15)*{\scriptstyle\delta};
(7,0)*{(1^n)};
(0,0)*{\figs{0.27}{EEDelta}};
\endxy
\end{array}
\end{equation}

\begin{equation}
\label{DeltaFEEDeltaF}
\begin{array}{ccccccc}
\xy
(0,0)*{\figs{0.25}{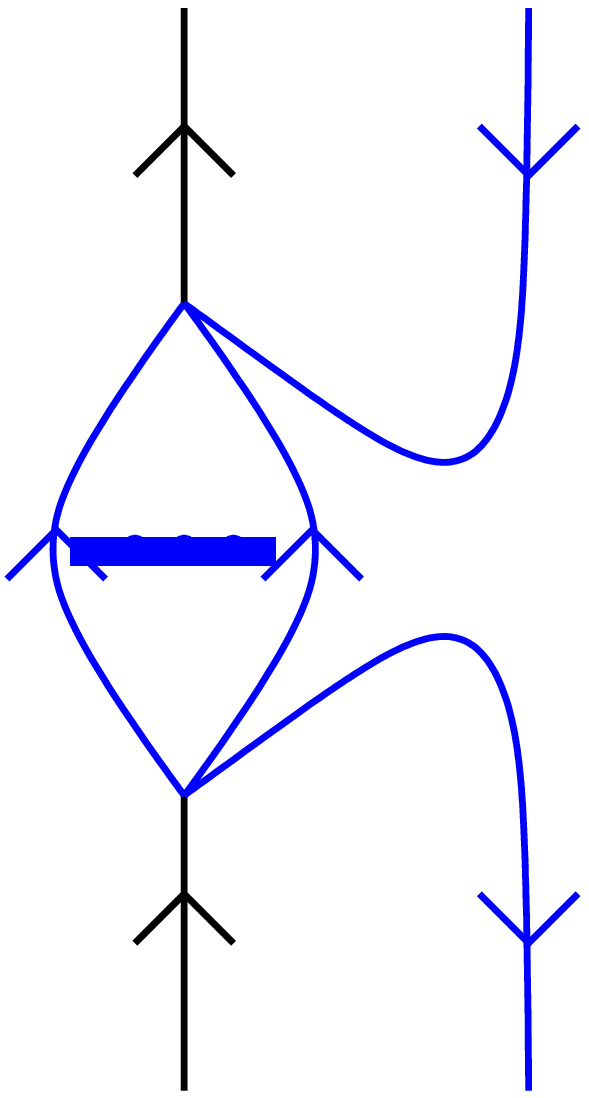}};
(-12,0)*{(1^n)};
(6,-16)*{\scriptstyle\blue i};
(-3,-16)*{\scriptstyle\delta};
\endxy
&=&
\xy
 (-10,0)*{\black\xybox{(0,14);(0,-14); **\dir{-} ?(.5)*\dir{>}+
(2.3,0)*{\scriptstyle{}};}};
(-11,-16)*{\scriptstyle\delta};
 (-18,0)*{(1^n)};
 (-10,0)*{};(10,0)*{};
 (0,0)*{\blue\xybox{(0,14);(0,-14); **\dir{-} ?(.5)*\dir{<}+
(2.3,0)*{\scriptstyle{}};}};
(-1,-16)*{\scriptstyle\blue i};
 (-12,0)*{};(12,0)*{};
\endxy
&{}\qquad{}&
\xy
(0,0)*{\figs{0.25}{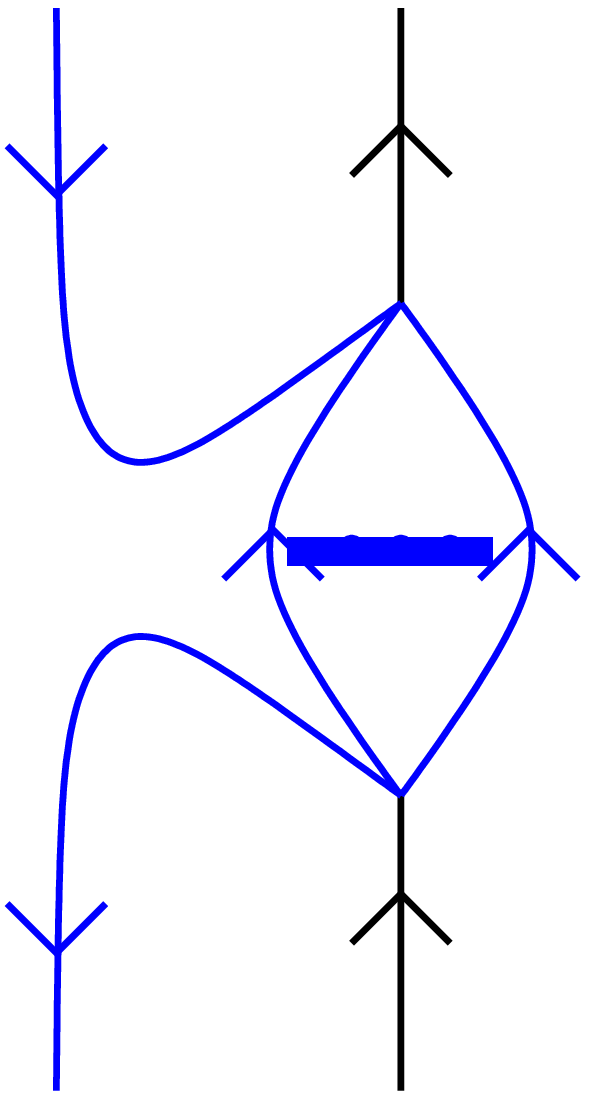}};
(12,0)*{(1^n)};
(-6,-16)*{\scriptstyle\blue i};
(3,-16)*{\scriptstyle \delta};
\endxy
&=&
\xy
 (0,0)*{\black\xybox{(0,14);(0,-14); **\dir{-} ?(.5)*\dir{>}+
(2.3,0)*{\scriptstyle{}};}};
(-1,-16)*{\scriptstyle\delta};
 (5,0)*{(1^n)};
 (-10,0)*{};(10,0)*{};
 (-10,0)*{\blue\xybox{(0,14);(0,-14); **\dir{-} ?(.5)*\dir{<}+
(2.3,0)*{\scriptstyle{}};}};
(-11,-16)*{\scriptstyle\blue i};
 (-12,0)*{};(12,0)*{};
\endxy
\end{array}
\end{equation}

\begin{equation}
\label{EEDeltaEE}
\begin{array}{ccccc}
\xy
(0,0)*{\figs{0.25}{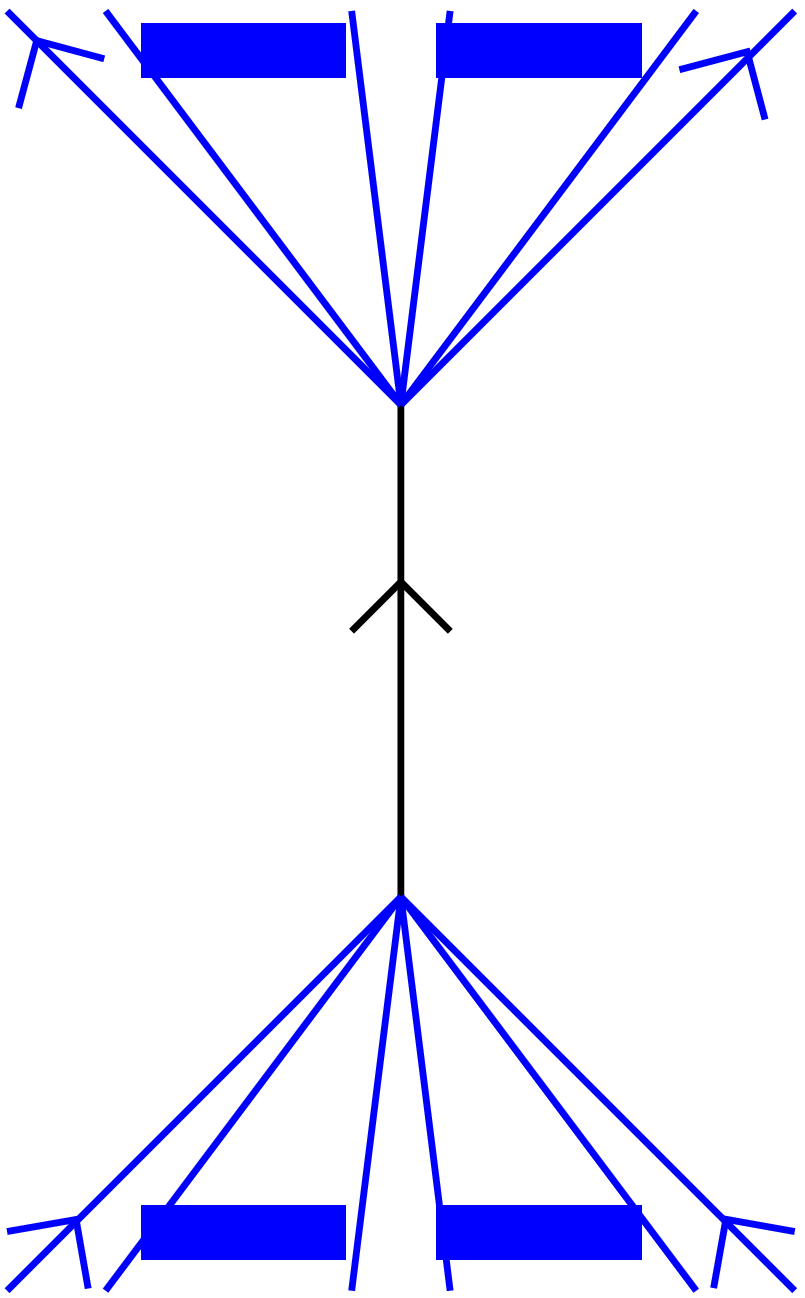}};
(-11,18)*{\scriptstyle\blue i};
(-7,18)*{\scriptstyle\blue i-1};
(-2,18)*{\scriptstyle\blue 1};
(2,18)*{\scriptstyle\blue n};
(7,18)*{\scriptstyle\blue i+2};
(13,18)*{\scriptstyle\blue i+1};
(-3,0)*{\scriptstyle\delta};
(7,0)*{(1^n)};
(-11,-18)*{\scriptstyle\blue i};
(-7,-18)*{\scriptstyle\blue i-1};
(-2,-18)*{\scriptstyle\blue 1};
(2,-18)*{\scriptstyle\blue n};
(7,-18)*{\scriptstyle\blue i+2};
(13,-18)*{\scriptstyle\blue i+1};
\endxy
&=&
\xy
 (-5,0)*{\blue\xybox{(0,14);(0,-14); **\dir{-} ?(.5)*\dir{>}+
(2.3,0)*{\scriptstyle{}};}};
 (0,0)*{\blue\xybox{(0,14);(0,-14); **\dir{-} ?(.5)*\dir{>}+
(2.3,0)*{\scriptstyle{}};}};
(3.5,0)*{\blue\cdots};
(8.5,0)*{\blue\xybox{(0,14);(0,-14); **\dir{-} ?(.5)*\dir{>}+
(2.3,0)*{\scriptstyle{}};}};
(13.5,0)*{\blue\xybox{(0,14);(0,-14); **\dir{-} ?(.5)*\dir{>}+
(2.3,0)*{\scriptstyle{}};}};
(17,0)*{\blue\cdots};
(22,0)*{\blue\xybox{(0,14);(0,-14); **\dir{-} ?(.5)*\dir{>}+
(2.3,0)*{\scriptstyle{}};}};
(26,0)*{\blue\xybox{(0,14);(0,-14); **\dir{-} ?(.5)*\dir{>}+
(2.3,0)*{\scriptstyle{}};}};
(-6.5,-18)*{\scriptstyle\blue i};
(-0.5,-18)*{\scriptstyle\blue i-1};
(7.5,-18)*{\scriptstyle\blue 1};
(12.5,-18)*{\scriptstyle\blue n};
(20,-18)*{\scriptstyle\blue i+2};
(26,-18)*{\scriptstyle\blue i+1};
(31,0)*{(1^n)};
(25,8)*{\blue\bullet};
\endxy
&-&
\xy
 (-5,0)*{\blue\xybox{(0,14);(0,-14); **\dir{-} ?(.5)*\dir{>}+
(2.3,0)*{\scriptstyle{}};}};
 (0,0)*{\blue\xybox{(0,14);(0,-14); **\dir{-} ?(.5)*\dir{>}+
(2.3,0)*{\scriptstyle{}};}};
(3.5,0)*{\blue\cdots};
(8.5,0)*{\blue\xybox{(0,14);(0,-14); **\dir{-} ?(.5)*\dir{>}+
(2.3,0)*{\scriptstyle{}};}};
(13.5,0)*{\blue\xybox{(0,14);(0,-14); **\dir{-} ?(.5)*\dir{>}+
(2.3,0)*{\scriptstyle{}};}};
(17,0)*{\blue\cdots};
(22,0)*{\blue\xybox{(0,14);(0,-14); **\dir{-} ?(.5)*\dir{>}+
(2.3,0)*{\scriptstyle{}};}};
(26,0)*{\blue\xybox{(0,14);(0,-14); **\dir{-} ?(.5)*\dir{>}+
(2.3,0)*{\scriptstyle{}};}};
(-6.5,-18)*{\scriptstyle\blue i};
(-0.5,-18)*{\scriptstyle\blue i-1};
(7.5,-18)*{\scriptstyle\blue 1};
(12.5,-18)*{\scriptstyle\blue n};
(20,-18)*{\scriptstyle\blue i+2};
(26,-18)*{\scriptstyle\blue i+1};
(31,0)*{(1^n)};
(-6,8)*{\blue\bullet};
\endxy
\end{array}
\end{equation}

\begin{equation}
\label{EEDeltaEEPerm}
\begin{array}{ccc}
\xy
(0,0)*{\figs{0.25}{EEDeltaEE}};
(-11,18)*{\scriptstyle\blue j};
(-7,18)*{\scriptstyle\blue j-1};
(-2,18)*{\scriptstyle\blue 1};
(2,18)*{\scriptstyle\blue n};
(7,18)*{\scriptstyle\blue j+2};
(13,18)*{\scriptstyle\blue j+1};
(-3,0)*{\scriptstyle\delta};
(7,0)*{(1^n)};
(-11,-18)*{\scriptstyle\blue i};
(-7,-18)*{\scriptstyle\blue i-1};
(-2,-18)*{\scriptstyle\blue 1};
(2,-18)*{\scriptstyle\blue n};
(7,-18)*{\scriptstyle\blue i+2};
(13,-18)*{\scriptstyle\blue i+1};
\endxy
&=&
\xy
(-5,0)*{\blue\xybox{(-5,-16);(26,16); **\dir{-} ?(.1)*\dir{<}+
(2.3,0)*{\scriptstyle{}};}};
(0,0)*{\blue\xybox{(-5,-16);(26,16); **\dir{-} ?(.1)*\dir{<}+
(2.3,0)*{\scriptstyle{}};}};
(10,0)*{\blue\xybox{(-5,-16);(26,16); **\dir{-} ?(.1)*\dir{<}+
(2.3,0)*{\scriptstyle{}};}};
(-5,0)*{\blue\xybox{(-5,16);(26,-16); **\dir{-} ?(.9)*\dir{>}+
(2.3,0)*{\scriptstyle{}};}};
(5,0)*{\blue\xybox{(-5,16);(26,-16); **\dir{-} ?(.9)*\dir{>}+
(2.3,0)*{\scriptstyle{}};}};
(10,0)*{\blue\xybox{(-5,16);(26,-16); **\dir{-} ?(.9)*\dir{>}+
(2.3,0)*{\scriptstyle{}};}};
(-7,-12)*{\blue\cdots};
(12,-12)*{\blue\cdots};
(-12,12)*{\blue\cdots};
(17,12)*{\blue\cdots};
(-21,-18)*{\scriptstyle\blue i};
(-16,-18)*{\scriptstyle\blue i-1};
(-6,-18)*{\scriptstyle\blue j+1};
(10,-18)*{\scriptstyle\blue j};
(20,-18)*{\scriptstyle\blue i+2};
(26,-18)*{\scriptstyle\blue i+1};
(-21,18)*{\scriptstyle\blue j};
(-11,18)*{\scriptstyle\blue i+2};
(-6,18)*{\scriptstyle\blue i+1};
(10,18)*{\scriptstyle\blue i};
(15,18)*{\scriptstyle\blue i-1};
(26,18)*{\scriptstyle\blue j+1};
(28,0)*{(1^n)};
\endxy
\end{array}
\end{equation}
\noindent Note that cyclicity implies the analogous relations with all 
orientations reversed. 
\end{defn}

Before giving the following lemma, we recall that the Karoubi envelope 
(or idempotent completion) of Khovanov and Lauda's 
2-categories, e.g. $\mathrm{Kar}
\,\mathcal{U}(\mathfrak{sl}_n)$ and 
$\mathrm{Kar}\,\mathcal{U}(\mathfrak{gl}_n)$, 
contain the categorified divided 
powers $\mathcal{E}_{\pm i}^{(a)}$, which satisfy 
$$
\mathcal{E}_{\pm i}^a=\left(\mathcal{E}_{\pm i}^{(a)}\right)^{\oplus [a]!}.
$$
In~\cite{KLMS} the $2$-morphisms in 
$\mathrm{Kar}\,\mathcal{U}(\mathfrak{sl}_2)$ between the divided powers 
were worked out explicitly. Using the fact 
that $\mathrm{Kar}\,\mathcal{U}(\mathfrak{sl}_2)$ can be embedded 
into $\mathrm{Kar}\,\mathcal{U}(\hat{\mathfrak{sl}}_n)$ for any choice of 
simple root, we can use the results in~\cite{KLMS}. We do not need much of 
that calculus in this paper, but we do have to recall 
the {\em splitters} (see definitions below Lemma 2.2.3 and 
see (2.63) in~\cite{KLMS}) 
$$
\xy
(0,0)*{\figs{0.25}{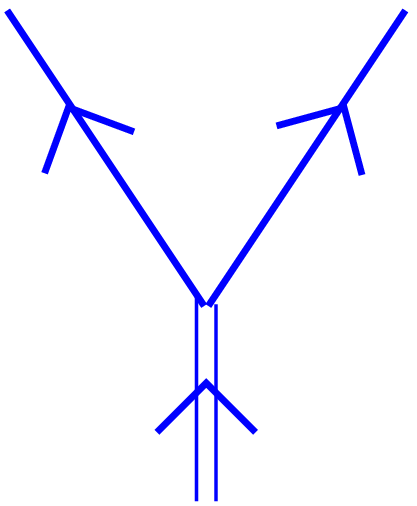}};
(17,0)*{\colon \mathcal{E}_{+i}^{(2)}\to \mathcal{E}_{+i}^2};
(0,-8)*{\blue\scriptstyle i};
(-5,8)*{\blue\scriptstyle i};
(5,8)*{\blue\scriptstyle i};
\endxy
\qquad\qquad
\xy
(0,0)*{\figs{0.25}{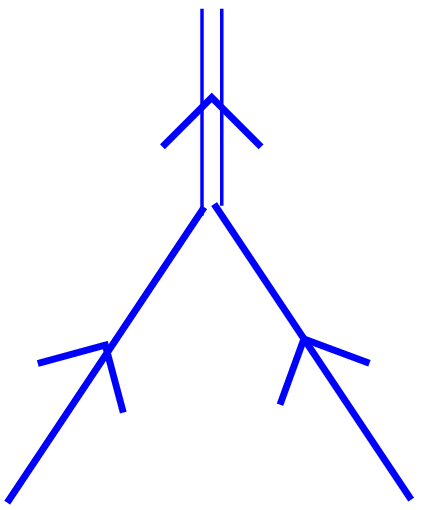}};
(17,0)*{\colon \mathcal{E}_{+i}^2\to \mathcal{E}_{+i}^{(2)}};
(0,8)*{\blue\scriptstyle i};
(-5,-8)*{\blue\scriptstyle i};
(5,-8)*{\blue\scriptstyle i};
\endxy
$$
and the relations (see (2.36), (2.64) and (2.65) in~\cite{KLMS}) 
$$
\xy
(0,0)*{\figs{0.35}{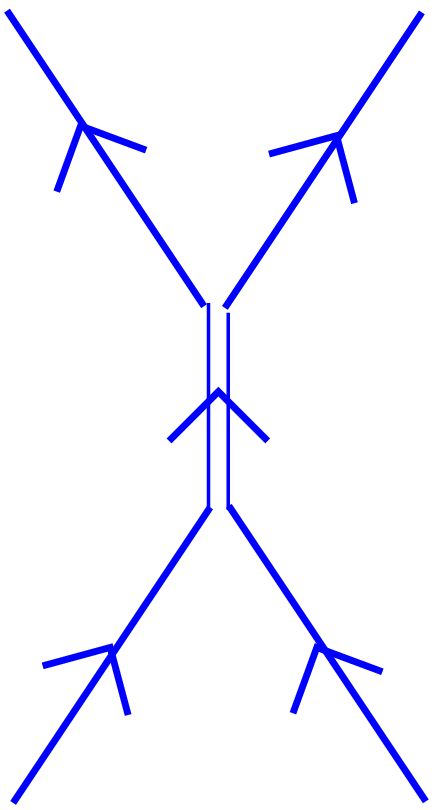}};
(-8,-16)*{\scriptstyle\blue i};
(8,-16)*{\scriptstyle\blue i};
(-8,16)*{\scriptstyle\blue i};
(8,16)*{\scriptstyle\blue i};
\endxy
\;=\;
\xy
(0,0)*{\figs{0.35}{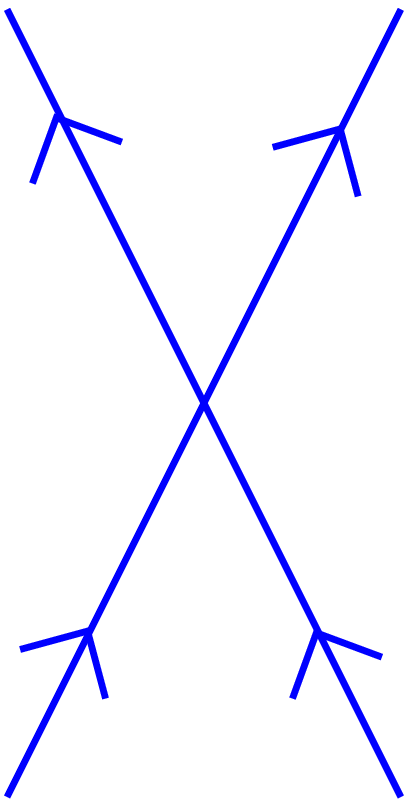}};
(-8,-16)*{\scriptstyle\blue i};
(8,-16)*{\scriptstyle\blue i};
(-8,16)*{\scriptstyle\blue i};
(8,16)*{\scriptstyle\blue i};
\endxy
\qquad\qquad
\xy
(0,0)*{\figs{0.32}{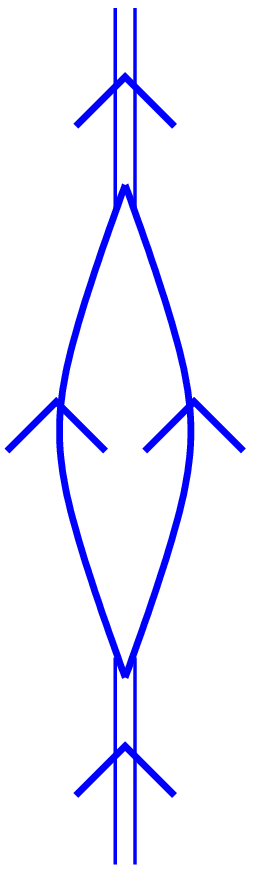}};
(0,-16)*{\scriptstyle\blue i};
\endxy
=0
\qquad\qquad
\xy
(0,0)*{\figs{0.32}{E2EEE2}};
(0,-16)*{\scriptstyle\blue i};
(-1.5,-3.5)*{\blue \bullet};
\endxy
=\;
\xy
(0,0)*{\blue\xybox{(0,14);(0,-14); **\dir2{-} ?(.5)*\dir{>}+
(2.3,0)*{\scriptstyle{}};}};
(-1,-16)*{\scriptstyle\blue i};
\endxy
\qquad\qquad
\xy
(0,0)*{\figs{0.32}{E2EEE2}};
(0,-16)*{\scriptstyle\blue i};
(1.5,-3.5)*{\blue \bullet};
\endxy
=
\;-\;\;
\xy
(0,0)*{\blue\xybox{(0,14);(0,-14); **\dir2{-} ?(.5)*\dir{>}+
(2.3,0)*{\scriptstyle{}};}};
(-1,-16)*{\scriptstyle\blue i};
\endxy
$$ 
for any $i=1,\ldots,n$. By cyclicity, we get similar splitters and 
relations for $\mathcal{E}_{-i}^{(2)}$, $i=1,\ldots,n$.

\begin{lem}
\begin{equation}
\label{EEFDeltaEE}
\begin{array}{ccccccc}
\xy
(0,0)*{\figs{0.25}{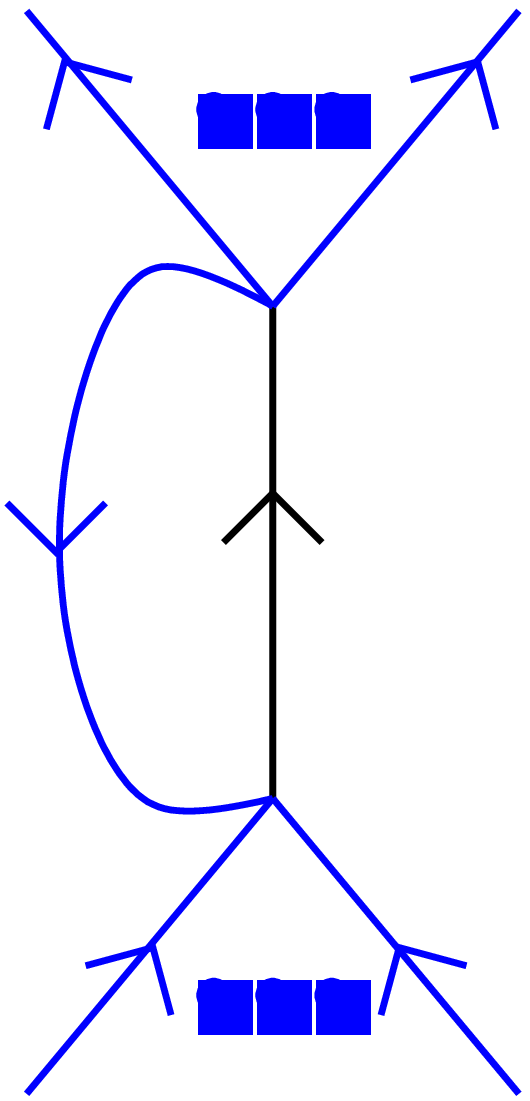}};
(10,0)*{(1^n)};
(-8,-16)*{\scriptstyle\blue i-1};
(8,-16)*{\scriptstyle\blue i+1};
(-9,0)*{\scriptstyle\blue i};
(3,0)*{\scriptstyle \delta};
\endxy
&=&
\xy
 (-5,0)*{\blue\xybox{(0,14);(0,-14); **\dir{-} ?(.5)*\dir{>}+
(2.3,0)*{\scriptstyle{}};}};
 (5,0)*{\blue\xybox{(0,14);(0,-14); **\dir{-} ?(.5)*\dir{>}+
(2.3,0)*{\scriptstyle{}};}};
(10,0)*{\blue\cdots};
 (20,0)*{\blue\xybox{(0,14);(0,-14); **\dir{-} ?(.5)*\dir{>}+
(2.3,0)*{\scriptstyle{}};}};
(-6,-16)*{\scriptstyle\blue i-1};
(4,-16)*{\scriptstyle\blue i-2};
(19,-16)*{\scriptstyle\blue i+1};
(25,0)*{(1^n)};
 (-10,0)*{};(10,0)*{};
\endxy
&{}\qquad{}&
\xy
(0,0)*{\figs{0.25}{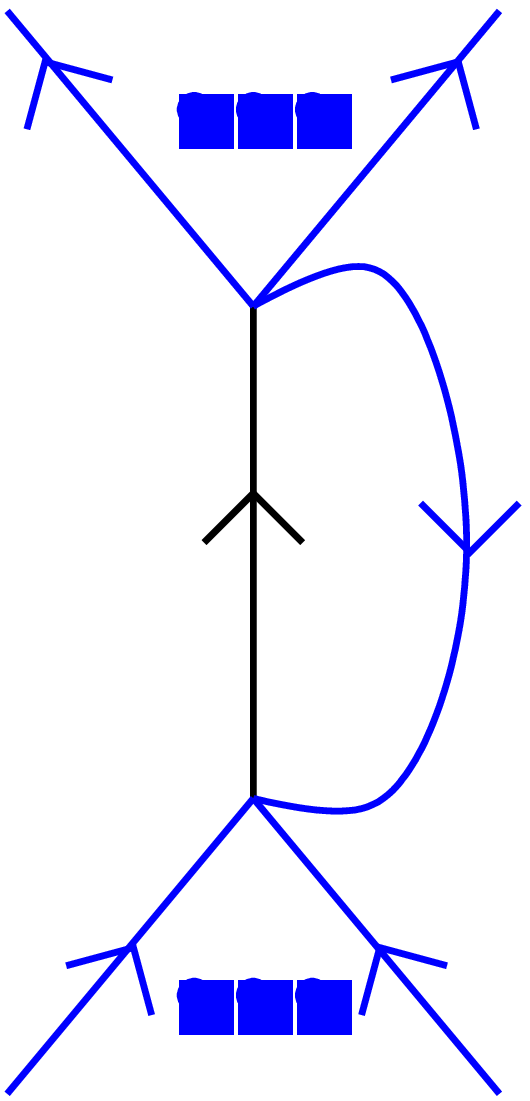}};
(-10,0)*{(1^n)};
(-9,-16)*{\scriptstyle\blue i-1};
(9,-16)*{\scriptstyle\blue i+1};
(-3,0)*{\scriptstyle \delta};
(8,0)*{\scriptstyle\blue i};
\endxy
&=&
\xy
 (-5,0)*{\blue\xybox{(0,14);(0,-14); **\dir{-} ?(.5)*\dir{>}+
(2.3,0)*{\scriptstyle{}};}};
 (5,0)*{\blue\xybox{(0,14);(0,-14); **\dir{-} ?(.5)*\dir{>}+
(2.3,0)*{\scriptstyle{}};}};
(10,0)*{\blue\cdots};
 (20,0)*{\blue\xybox{(0,14);(0,-14); **\dir{-} ?(.5)*\dir{>}+
(2.3,0)*{\scriptstyle{}};}};
(-6,-16)*{\scriptstyle\blue i-1};
(4,-16)*{\scriptstyle\blue i-2};
(19,-16)*{\scriptstyle\blue i+1};
(-13,0)*{(1^n)};
 (-10,0)*{};(10,0)*{};
\endxy
\end{array}
\end{equation}

\begin{equation}
\label{EDeltaEEEDelta}
\begin{array}{ccccccc}
\xy
(0,0)*{\figs{0.25}{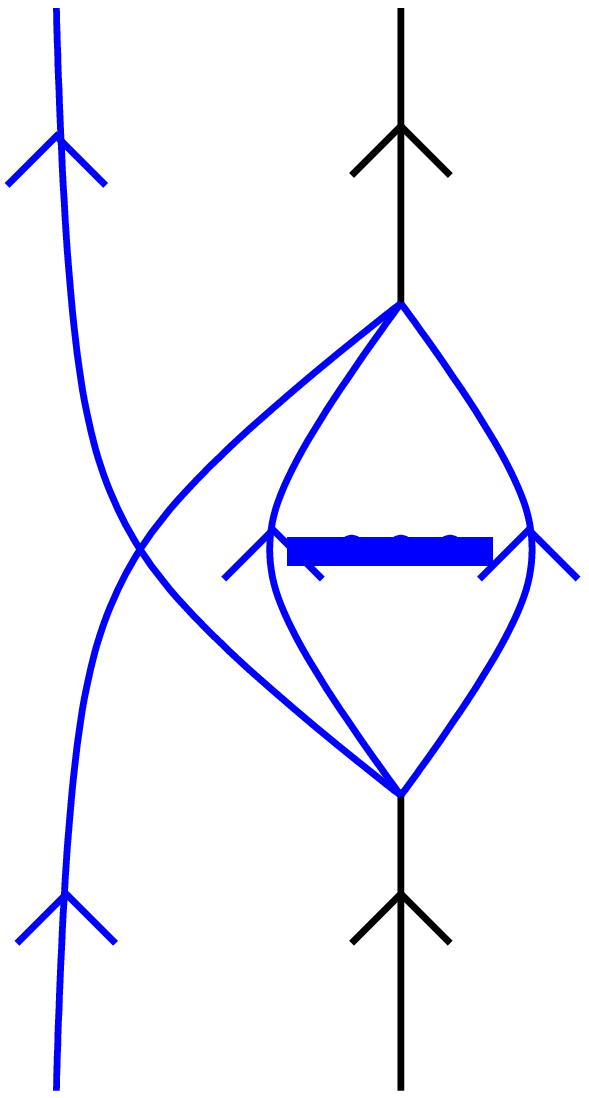}};
(12,0)*{(1^n)};
(-6,-16)*{\scriptstyle\blue i};
(3,-16)*{\scriptstyle \delta};
\endxy
&=&
\xy
 (0,0)*{\black\xybox{(0,14);(0,-14); **\dir{-} ?(.5)*\dir{>}+
(2.3,0)*{\scriptstyle{}};}};
(-1,-16)*{\scriptstyle \delta};
 (5,0)*{(1^n)};
 (-10,0)*{};(10,0)*{};
 (-10,0)*{\blue\xybox{(0,14);(0,-14); **\dir{-} ?(.5)*\dir{>}+
(2.3,0)*{\scriptstyle{}};}};
(-11,-16)*{\scriptstyle\blue i};
 (-12,0)*{};(12,0)*{};
\endxy
&{}\qquad{}&
\xy
(0,0)*{\figs{0.25}{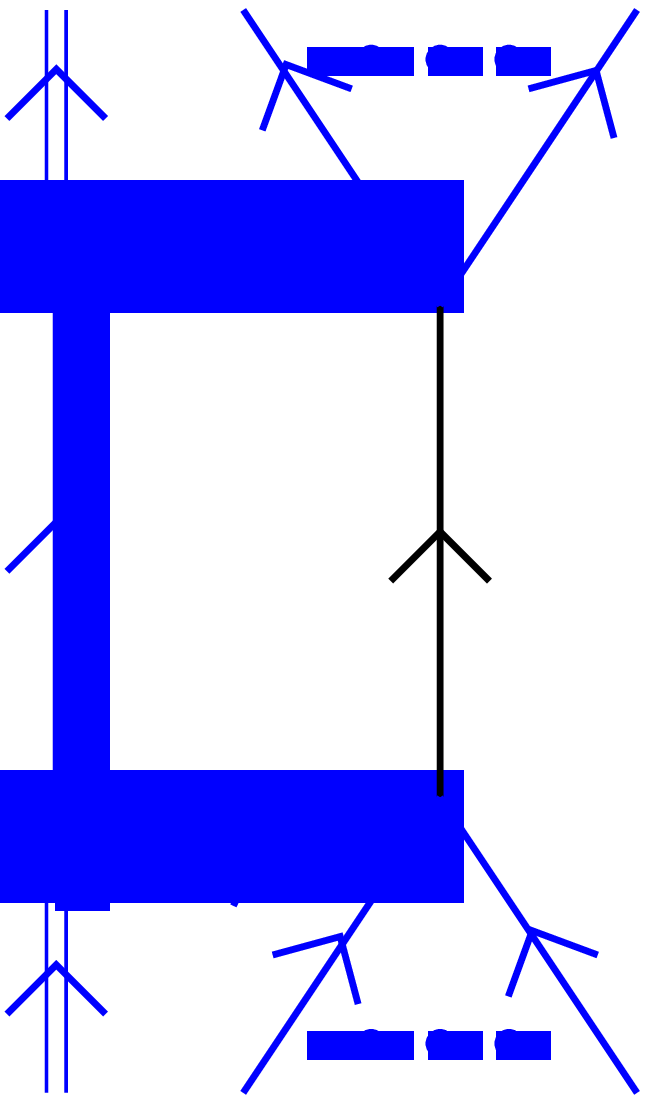}};
(12,0)*{(1^n)};
(0,-16)*{\scriptstyle\blue i-1};
(12,-16)*{\scriptstyle\blue i+1};
(-7,-16)*{\scriptstyle\blue i};
(6,0)*{\scriptstyle \delta};
\endxy
&=&
\xy
(-5,0)*{\blue\xybox{(0,14);(0,-14); **\dir2{-} ?(.5)*\dir{>}+
(2.3,0)*{\scriptstyle{}};}};
 (5,0)*{\blue\xybox{(0,14);(0,-14); **\dir{-} ?(.5)*\dir{>}+
(2.3,0)*{\scriptstyle{}};}};
(10,0)*{\blue\cdots};
 (20,0)*{\blue\xybox{(0,14);(0,-14); **\dir{-} ?(.5)*\dir{>}+
(2.3,0)*{\scriptstyle{}};}};
(-6,-16)*{\scriptstyle\blue i};
(4,-16)*{\scriptstyle\blue i-1};
(19,-16)*{\scriptstyle\blue i+1};
(25,0)*{(1^n)};
 (-10,0)*{};(10,0)*{};
\endxy
\end{array}
\end{equation}

\begin{equation}
\label{DeltaEEEDeltaE}
\begin{array}{ccccccc}
\xy
(0,0)*{\figs{0.25}{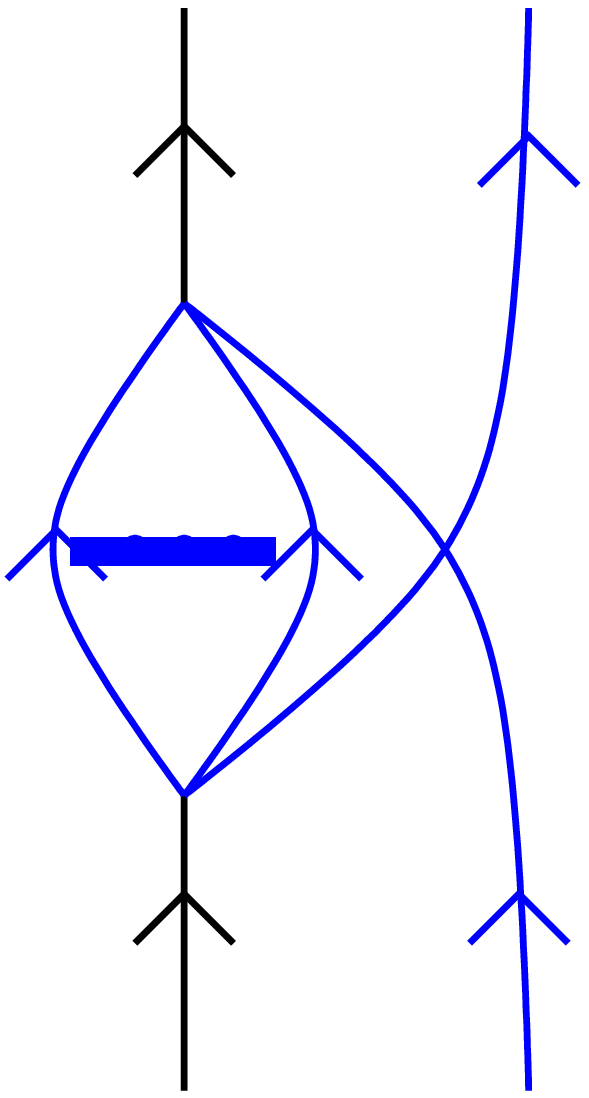}};
(-12,0)*{(1^n)};
(6,-16)*{\scriptstyle\blue i};
(-3,-16)*{\scriptstyle \delta};
\endxy
&=&
\xy
 (0,0)*{\blue\xybox{(0,14);(0,-14); **\dir{-} ?(.5)*\dir{>}+
(2.3,0)*{\scriptstyle{}};}};
(-1,-16)*{\scriptstyle i};
 (-18,0)*{(1^n)};
 (-10,0)*{};(10,0)*{};
 (-10,0)*{\black\xybox{(0,14);(0,-14); **\dir{-} ?(.5)*\dir{>}+
(2.3,0)*{\scriptstyle{}};}};
(-11,-16)*{\scriptstyle \delta};
 (-12,0)*{};(12,0)*{};
\endxy
&{}\qquad{}&
\xy
(0,0)*{\figs{0.25}{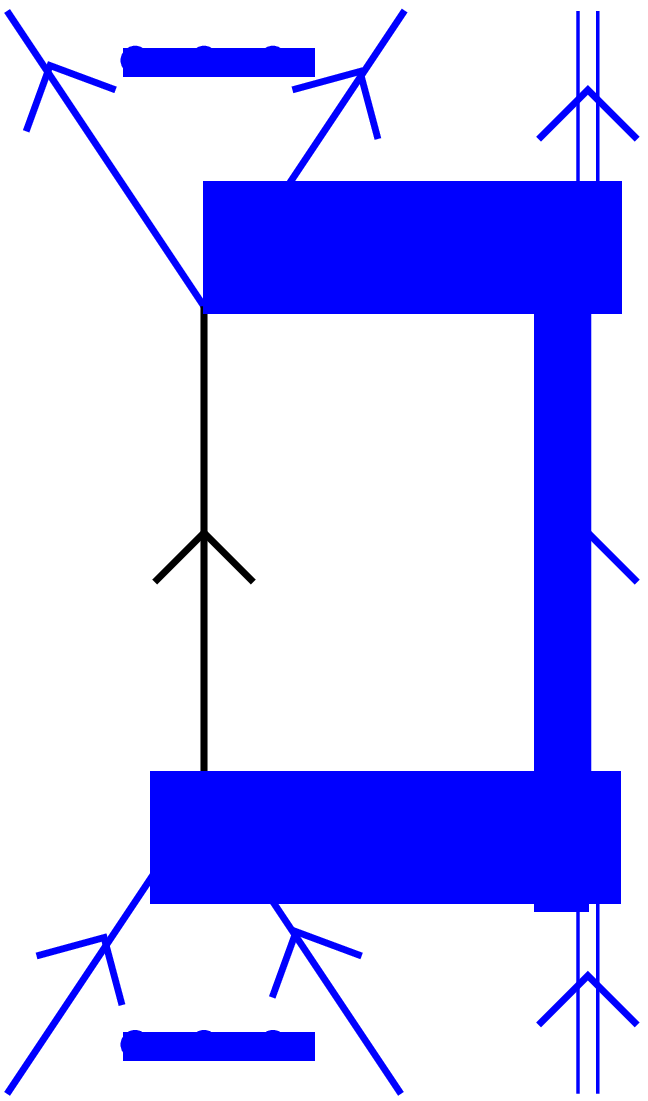}};
(-14,0)*{(1^n)};
(0,-16)*{\scriptstyle\blue i+1};
(7,-16)*{\scriptstyle\blue i};
(-11,-16)*{\scriptstyle\blue i-1};
(-6,0)*{\scriptstyle \delta};
\endxy
&=&
\xy
(-5,0)*{\blue\xybox{(0,14);(0,-14); **\dir{-} ?(.5)*\dir{>}+
(2.3,0)*{\scriptstyle{}};}};
 (11,0)*{\blue\xybox{(0,14);(0,-14); **\dir{-} ?(.5)*\dir{>}+
(2.3,0)*{\scriptstyle{}};}};
(2,0)*{\blue\cdots};
 (20,0)*{\blue\xybox{(0,14);(0,-14); **\dir2{-} ?(.5)*\dir{>}+
(2.3,0)*{\scriptstyle{}};}};
(-6,-16)*{\scriptstyle\blue i-1};
(10,-16)*{\scriptstyle\blue i+1};
(19,-16)*{\scriptstyle\blue i};
(-14,0)*{(1^n)};
 (-10,0)*{};(10,0)*{};
\endxy
\end{array}
\end{equation}
By cyclicity, we get the analogous relations with all orientations 
reversed. 
\end{lem}
\begin{proof}
The equations in~\eqref{EEFDeltaEE} follow directly from~\eqref{EEDeltaEE} and 
the bubble relations. Note that one of the terms we get by 
applying~\eqref{EEDeltaEE} has a bubble of degree $-2$, which is equal to 
zero, and the other term has a bubble of degree $0$ which is equal to 
$-1$ if it is counter-clockwise and $+1$ if it is clockwise.     
\vskip0.2cm
We only prove the equations in~\eqref{EDeltaEEEDelta}. The equations 
in~\eqref{DeltaEEEDeltaE} can be proved similarly. By the second 
relation in~\eqref{DeltaFEEDeltaF}, curl removal and 
the evaluation of degree zero bubbles, we get
$$
\xy
(0,0)*{\figs{0.25}{EDeltaEEEDelta}};
(12,0)*{(1^n)};
(-6,-16)*{\scriptstyle\blue i};
(3,-16)*{\scriptstyle \delta};
\endxy
=
\xy
  (0,-3)*{\blue\xybox{
  (-3,-8)*{};(3,8)*{} **\crv{(-3,-1) & (3,1)}?(1)*\dir{>};?(0)*\dir{>};
    (3,-8)*{};(-3,8)*{} **\crv{(3,-1) & (-3,1)}?(1)*\dir{>};
  (-3,-12)*{\bbsid};
  (-3,8)*{\bbsid};
  (3,8)*{}="t1";
  (9,8)*{}="t2";
  (3,-8)*{}="t1'";
  (9,-8)*{}="t2'";
   "t1";"t2" **\crv{(3,14) & (9, 14)};
   "t1'";"t2'" **\crv{(3,-14) & (9, -14)};
   "t2'";"t2" **\dir{-} ?(.5)*\dir{<};}};
   (9,0)*{}; (-5,-16)*{\scriptstyle\blue i};
(15,0)*{\black\xybox{(0,14);(0,-14); **\dir{-} ?(.5)*\dir{>}+
(2.3,0)*{\scriptstyle{}};}};
(14,-16)*{\scriptstyle \delta};
 (20,0)*{(1^n)};
 (-10,0)*{};(10,0)*{};
\endxy
\;\;=\;\;
\xy
 (0,0)*{\black\xybox{(0,14);(0,-14); **\dir{-} ?(.5)*\dir{>}+
(2.3,0)*{\scriptstyle{}};}};
(-1,-16)*{\scriptstyle \delta};
 (5,0)*{(1^n)};
 (-10,0)*{};(10,0)*{};
 (-10,0)*{\blue\xybox{(0,14);(0,-14); **\dir{-} ?(.5)*\dir{>}+
(2.3,0)*{\scriptstyle{}};}};
(-11,-16)*{\scriptstyle\blue i};
 (-12,0)*{};(12,0)*{};
\endxy
$$

By~\eqref{EEDeltaEE} and the relations in (2.64) 
in~\cite{KLMS}, we get
$$
\xy
(0,0)*{\figs{0.25}{E2DeltaEEE2Delta}};
(12,0)*{(1^n)};
(-2,-16)*{\scriptstyle\blue i-1};
(9,-16)*{\scriptstyle\blue i+1};
(-7,-16)*{\scriptstyle\blue i};
(6,0)*{\delta};
\endxy
\;\;=\;\;
\xy
(0,0)*{\figs{0.32}{E2EEE2}};
(0,-16)*{\scriptstyle\blue i};
(24,-3.5)*{\blue \bullet};
(10,0)*{\blue\xybox{(0,14);(0,-14); **\dir{-} ?(.5)*\dir{>}+
(2.3,0)*{\scriptstyle{}};}};
(15,0)*{\blue\cdots};
(25,0)*{\blue\xybox{(0,14);(0,-14); **\dir{-} ?(.5)*\dir{>}+
(2.3,0)*{\scriptstyle{}};}};
(10,-16)*{\scriptstyle\blue i-1};
(25,-16)*{\scriptstyle\blue i+1};
(30,0)*{(1^n)};
\endxy
\;\;-\;\;
\xy
(0,0)*{\figs{0.32}{E2EEE2}};
(0,-16)*{\scriptstyle\blue i};
(1.5,-3.5)*{\blue \bullet};
(10,0)*{\blue\xybox{(0,14);(0,-14); **\dir{-} ?(.5)*\dir{>}+
(2.3,0)*{\scriptstyle{}};}};
(15,0)*{\blue\cdots};
(25,0)*{\blue\xybox{(0,14);(0,-14); **\dir{-} ?(.5)*\dir{>}+
(2.3,0)*{\scriptstyle{}};}};
(10,-16)*{\scriptstyle\blue i-1};
(25,-16)*{\scriptstyle\blue i+1};
(30,0)*{(1^n)};
\endxy
\;\;=\;\;
\xy
(0,0)*{\blue\xybox{(0,14);(0,-14); **\dir2{-} ?(.5)*\dir{>}+
(2.3,0)*{\scriptstyle{}};}};
(5,0)*{\blue\xybox{(0,14);(0,-14); **\dir{-} ?(.5)*\dir{>}+
(2.3,0)*{\scriptstyle{}};}};
(10,0)*{\blue\cdots};
(15,0)*{\blue\xybox{(0,14);(0,-14); **\dir{-} ?(.5)*\dir{>}+
(2.3,0)*{\scriptstyle{}};}};
(-2,-16)*{\scriptstyle\blue i};
(5,-16)*{\scriptstyle\blue i-1};
(15,-16)*{\scriptstyle\blue i+1};
(20,0)*{(1^n)};
\endxy
$$

\end{proof}

\begin{lem}
\label{lem:polidigonreswithoutdots}
We have 
$$
\xy
(0,0)*{\black\xybox{(0,14);(0,-14); **\dir{-} ?(.5)*\dir{>}+
(2.3,0)*{\scriptstyle{}};}};
(-1,-17)*{\scriptstyle \delta};
(-10,0)*{\blue\ncbub_{i}};
(5,0)*{(1^n)};
\endxy
\;\;
=\;\;
\xy
(0,0)*{\figs{0.25}{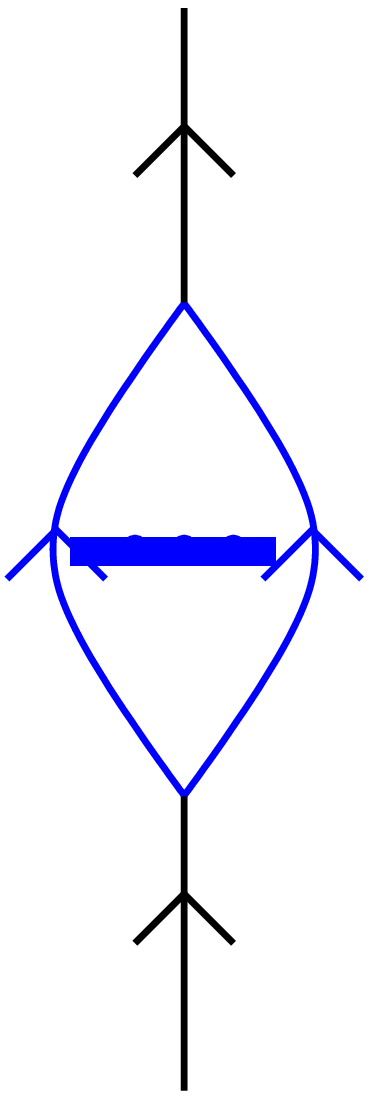}};
(10,0)*{(1^n)};
(-4,-4)*{\blue \scriptstyle i};
(6,-4)*{\blue \scriptstyle i+1};
\endxy
\;\;
=
\;\;
\xy
(0,0)*{\black\xybox{(0,14);(0,-14); **\dir{-} ?(.5)*\dir{>}+
(2.3,0)*{\scriptstyle{}};}};
(-1,-17)*{\scriptstyle \delta};
(10,0)*{\blue\ncbub_{i+1}};
(20,0)*{(1^n)};
\endxy
$$
\end{lem}
\begin{proof}
The first equality is a direct consequence of the first relation 
in~\eqref{DeltaFEEDeltaF}. 

The second is a consequence of the first relation in~\eqref{DeltaFEEDeltaF} 
and the fact that 
$$
\xy
(10,0)*{\blue\ncbub_{i+1}};
(20,0)*{(1^n)};
\endxy
\;\;
=
\;
\xy
(10,0)*{\blue\nccbub_{i+1}};
(20,0)*{(1^n)};
\endxy,
$$
which follows from the infinite Grassmannian relation for bubbles.
\end{proof}

In order to formulate the following results, define 
$$
\xy
(0,0)*{\blue\bbox{z_m}};
(10,0)*{(1^n)};
\endxy
:=
-\left(\xy
(0,0)*{{\blue \nccbub}+{\blue \nccbub}+\cdots+{\blue \nccbub}};
(-15,4)*{\blue\scriptstyle{i-1}};
(2,4)*{\blue\scriptstyle{i-2}};
(27,4)*{\blue\scriptstyle{m}};
\endxy
\right)
\xy
(0,0)*{(1^n)};
\endxy
$$
The sum of the bubbles is over the colors 
$i-1,i-2,\ldots,m$, if $1\leq m\leq i-1$, and over the 
colors $i-1,i-2,\ldots,1,n,n-1,\ldots m$, if $m\geq i+1$. 
These are exactly the colors of all the strands in the diagram on the 
left-hand side of Lemma~\ref{lem:polidigonreswithdots} between 
the strands $i-1$ and $m$.  
By definition we take ${\blue\bbox{z_i}}=0$ and use 
the convention that $0^0=1$. 

Similarly, we define
$$
\xy
(0,0)*{\blue\bbox{y_m}};
(10,0)*{(1^n)};
\endxy
:=
-\left(\xy
(0,0)*{{\blue \ncbub}+{\blue \ncbub}+\cdots+{\blue \ncbub}};
(-16,4)*{\blue\scriptstyle{m}};
(2,4)*{\blue\scriptstyle{m-1}};
(28,4)*{\blue\scriptstyle{i+2}};
\endxy
\ \right)
\xy
(0,0)*{(1^n)};
\endxy.
$$
The sum of the bubbles is over the colors 
$m,m-1,\ldots,i+2$, if $i+2\leq m\leq n$, and over the 
colors $m,m-1,\ldots,1,n,n-1,\ldots i+2$, if $m\leq i+1$. These are exactly 
the colors of all the strands in the diagram on the 
left-hand side of Lemma~\ref{lem:polidigonreswithdots} between 
the strands $m$ and $i+2$.  
By definition we take ${\blue\bbox{y_{i+1}}}=0$ and use 
the convention that $0^0=1$. 

Note that 
$$\blue\bbox{y_{i-1}}=\blue\bbox{z_{i+2}}$$ 
by the infinite Grassmannian relation. 

\begin{lem}
\label{lem:polidigonreswithdots}
For any $1\leq m\leq n$ and $s,t\in\mathbb{N}$, we have 
\begin{equation}
\label{eq:polidigonreswithdots}
\xy
(0,0)*{\figs{0.25}{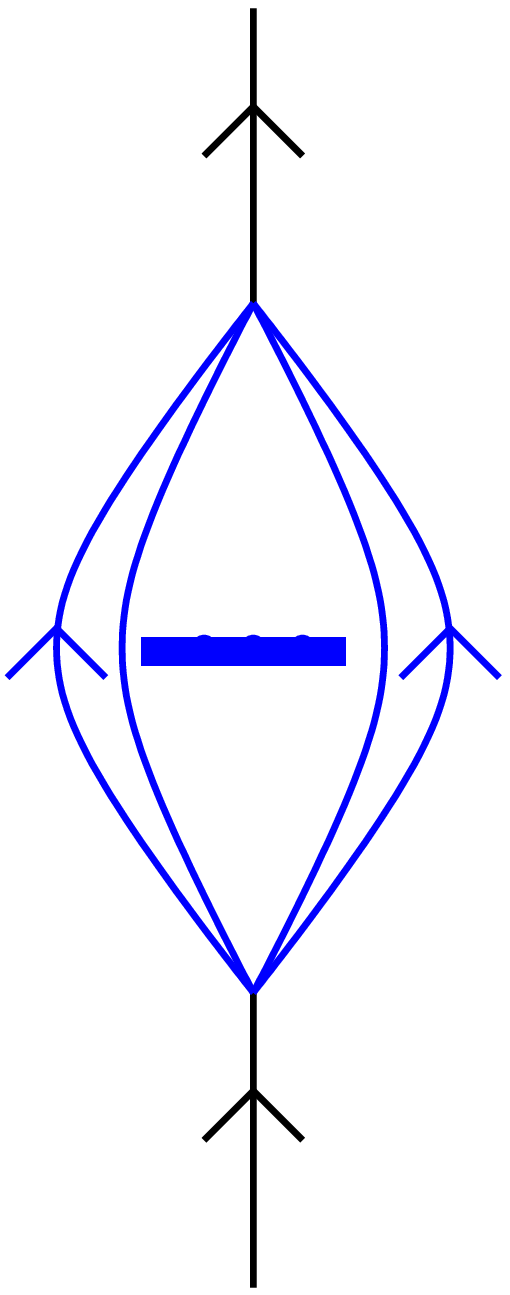}};
(12,0)*{(1^n)};
(-6,-4)*{\blue\scriptstyle{i}};
(0,-19)*{\scriptstyle \delta};
(1.9,4)*{\blue\bullet};
(-3.6,4)*{\blue\bullet};
(-5.5,5.5)*{\blue\scriptstyle{t}};
(0,5)*{\blue\scriptstyle{s}};
\endxy
=
\sum_{j=0}^{s}\binom{s}{j}
\xy
(0,0)*{\black\xybox{(0,14);(0,-14); **\dir{-} ?(.5)*\dir{>}+
(2.3,0)*{\scriptstyle{}};}};
(-1,-17)*{\scriptstyle \delta};
(-15,0)*{{\blue\cbub{s+t-j}{i}}\;{\blue\bbox{z_m^{j}}}};
(5,0)*{(1^n)};
\endxy.
\end{equation}
On the left-hand side of Equation~\eqref{eq:polidigonreswithdots}, the $t$ dots 
are on the $i$-th strand and the $s$ dots are on the $m$-th strand. Similarly, 
we have 
\begin{equation}
\label{eq:polidigonreswithdots2}
\xy
(0,0)*{\figs{0.25}{DeltaEEDeltaCompl}};
(12,0)*{(1^n)};
(-6,-4)*{\blue\scriptstyle{i}};
(0,-19)*{\scriptstyle \delta};
(3.6,4)*{\blue\bullet};
(-1.9,4)*{\blue\bullet};
(0,5)*{\blue\scriptstyle{t}};
(5.5,5)*{\blue\scriptstyle{s}};
\endxy
=
\sum_{j=0}^{t}\binom{t}{j}\;\;
\xy
(0,0)*{\black\xybox{(0,14);(0,-14); **\dir{-} ?(.5)*\dir{>}+
(2.3,0)*{\scriptstyle{}};}};
(-1,-17)*{\scriptstyle \delta};
(12,0)*{{\blue\ccbub{s+t-j}{i+1\ \ \ \ }}\;{\blue\bbox{y_m^{j}}}};
(30,0)*{(1^n)};
\endxy.
\end{equation}
On the left-hand side of Equation~\eqref{eq:polidigonreswithdots}, the $t$ dots 
are on the $m$-th strand and the $s$ dots are on the $i+1$-st strand. 
\end{lem}
\begin{proof}
We only prove the first equation. The second can be proved in a similar way. 
The proof is by induction w.r.t. $s$. For $s=0$ and any 
$1\leq m\leq n$ and $t\in\mathbb{N}$, the result follows 
from~\eqref{DeltaFEEDeltaF}. 

Suppose $s>0$, $t\in\mathbb{N}$ and $m\ne i+1$. The case $m=i$ 
follows from~\eqref{DeltaFEEDeltaF}, so we can assume that $m\ne i$. 
First note the following:
\begin{equation}
\label{eq:polidigoncurl}
0=
\xy
(0,0)*{\figs{0.25}{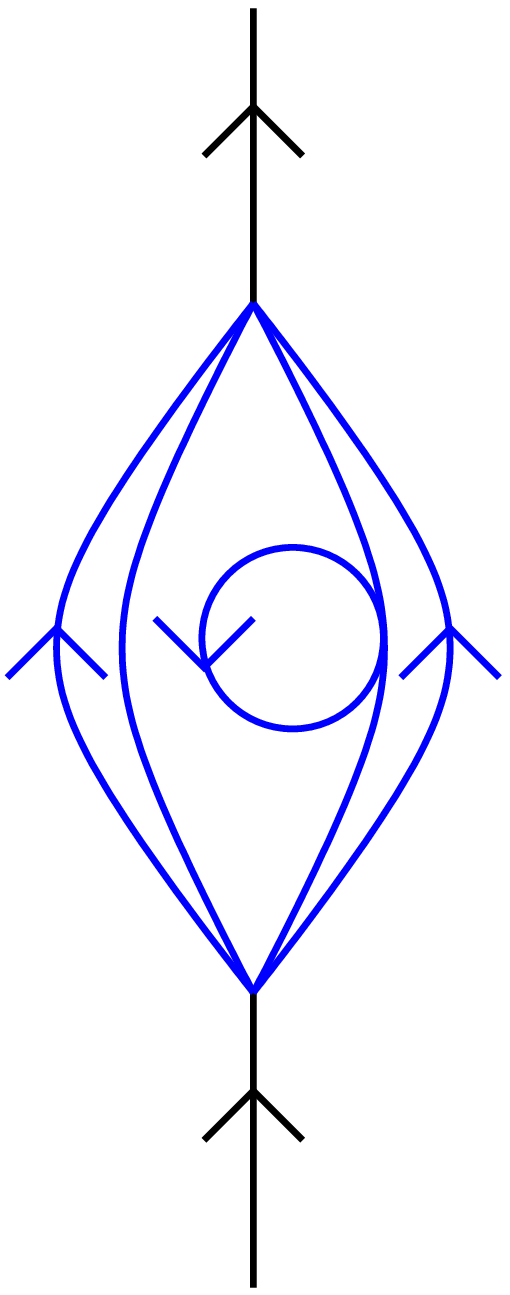}};
(12,0)*{(1^n)};
(-6,-4)*{\blue\scriptstyle{i}};
(0,4)*{\blue\scriptstyle{m}};
(0,-19)*{\scriptstyle \delta};
(-3.6,4)*{\blue\bullet};
(-5.5,5.5)*{\blue\scriptstyle{t}};
\endxy
=
\;\;-\;\;
\xy
(0,0)*{\figs{0.25}{DeltaEEDeltaCompl}};
(12,0)*{(1^n)};
(-6,-4)*{\blue\scriptstyle{i}};
(0,-19)*{\scriptstyle \delta};
(2.2,4)*{\blue\bullet};
(1,-3)*{\blue\scriptstyle{m}};
(-3.6,4)*{\blue\bullet};
(-5.5,5.5)*{\blue\scriptstyle{t}};
\endxy
\;\;
+
\;\;
\xy
(0,0)*{\figs{0.25}{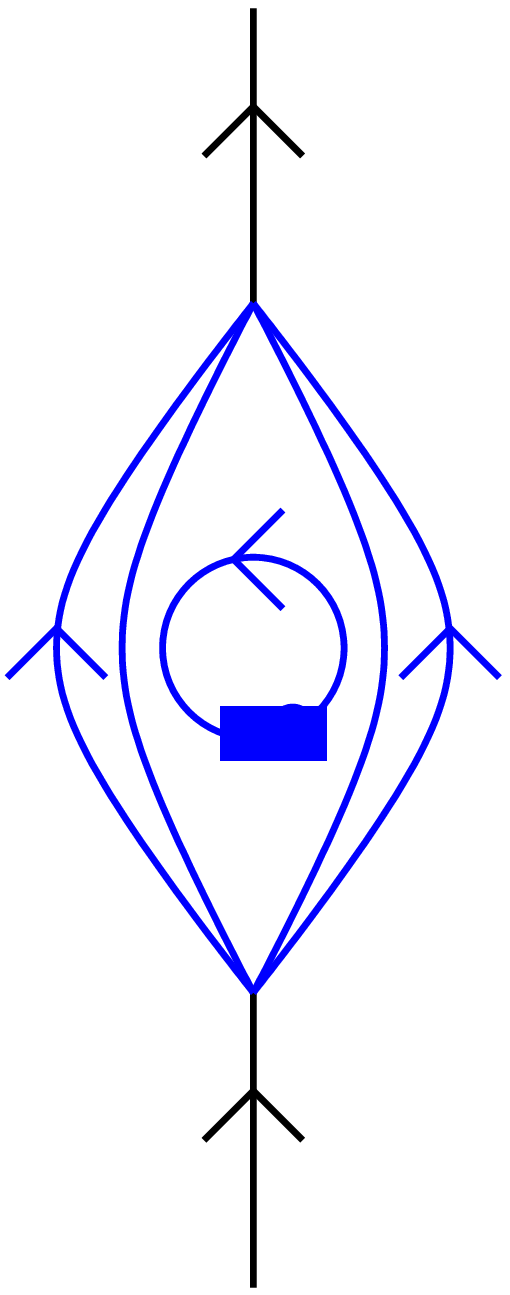}};
(12,0)*{(1^n)};
(-6,-4)*{\blue\scriptstyle{i}};
(0,0)*{\blue\scriptstyle{m}};
(0,-19)*{\scriptstyle \delta};
(0,-4)*{\blue\scriptstyle{-1}};
(-3.6,4)*{\blue\bullet};
(-5.5,5.5)*{\blue\scriptstyle{t}};
\endxy.
\end{equation}
The first equality holds, because the label of the region 
inside the curl does not belong to $\Lambda(n,n)$; its $m+1$-st 
entry equals $-1$. The second equality follows from resolving 
the curl. The minus sign is a consequence of our normalization of 
degree zero bubbles in~\cite{MTh}, because the label $\lambda$ 
of the region just outside the 
bubble satisfies $\lambda_{m+1}=0$. Note that the bubble in the second term has 
degree two, since $\lambda_m-\lambda_{m+1}=1$, for any $m\ne i,i+1$.

Equation~\eqref{eq:polidigoncurl} implies 
\begin{equation}
\label{eq:polidigondotred}
\xy
(0,0)*{\figs{0.25}{DeltaEEDeltaCompl}};
(12,0)*{(1^n)};
(-6,-4)*{\blue\scriptstyle{i}};
(0,-19)*{\scriptstyle \delta};
(1,-3)*{\blue\scriptstyle{m}};
(2.2,4)*{\blue\bullet};
(4,6)*{\blue\scriptstyle{s}};
(-3.6,4)*{\blue\bullet};
(-5.5,5.5)*{\blue\scriptstyle{t}};
\endxy
\;\;
=
\;\;
\xy
(0,0)*{\figs{0.25}{Polidigonbubble}};
(12,0)*{(1^n)};
(-6,-4)*{\blue\scriptstyle{i}};
(0,0)*{\blue\scriptstyle{m}};
(0,-19)*{\scriptstyle \delta};
(0,-4)*{\blue\scriptstyle{-1}};
(2.2,4)*{\blue\bullet};
(6,6)*{\blue\scriptstyle{s-1}};
(-3.6,4)*{\blue\bullet};
(-5.5,5.5)*{\blue\scriptstyle{t}};
\endxy.
\end{equation}  
Now slide the $m$-bubble to the left. Note that the strand directly to the left of the bubble has color $m+1$ (we keep 
considering the colors modulo $n$). Therefore, by the bubble slide 
relations and the degree zero bubble relations in~\cite{MTh}, we get 
\begin{equation}
\label{eq:bubblemove}
\xy
(0,0)*{\figs{0.25}{Polidigonbubble}};
(12,0)*{(1^n)};
(-6,-4)*{\blue\scriptstyle{i}};
(0,0)*{\blue\scriptstyle{m}};
(0,-19)*{\scriptstyle \delta};
(0,-4)*{\blue\scriptstyle{-1}};
(2.2,4)*{\blue\bullet};
(6,6)*{\blue\scriptstyle{s-1}};
(-3.6,4)*{\blue\bullet};
(-5.5,5.5)*{\blue\scriptstyle{t}};
\endxy
\;\;
=
\;\;
\xy
(0,0)*{\figs{0.25}{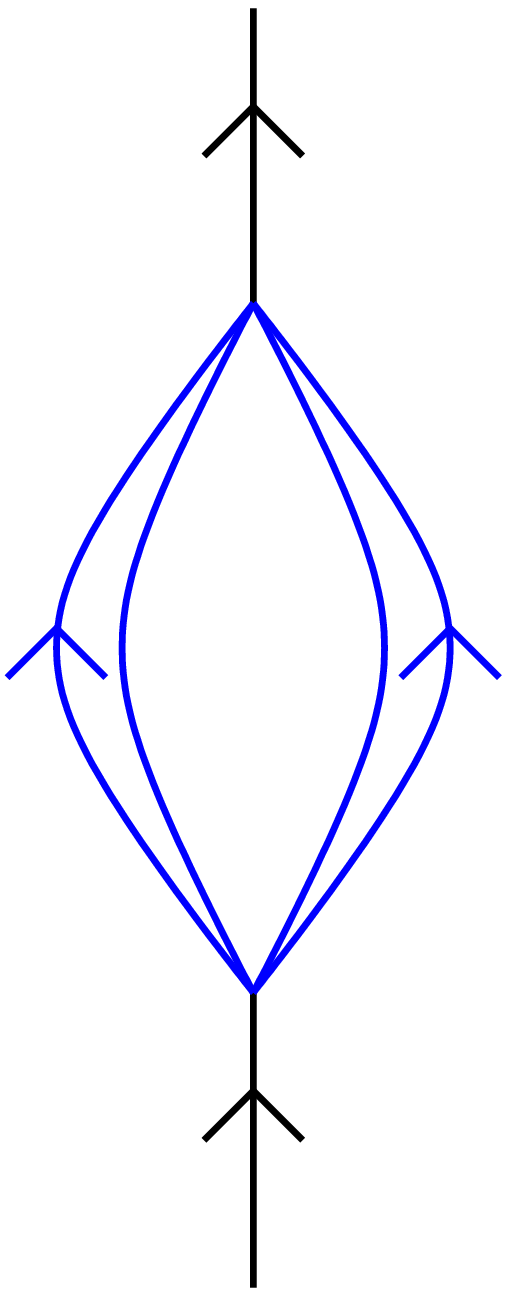}};
(12,0)*{(1^n)};
(-6,-4)*{\blue\scriptstyle{i}};
(0,-19)*{\scriptstyle \delta};
(-2.2,4)*{\blue\bullet};
(6,6)*{\blue\scriptstyle{s-1}};
(2.2,4)*{\blue\bullet};
(-3.6,4)*{\blue\bullet};
(-5.5,5.5)*{\blue\scriptstyle{t}};
\endxy
\;\;
-
\;\;
\xy
(0,0)*{\figs{0.25}{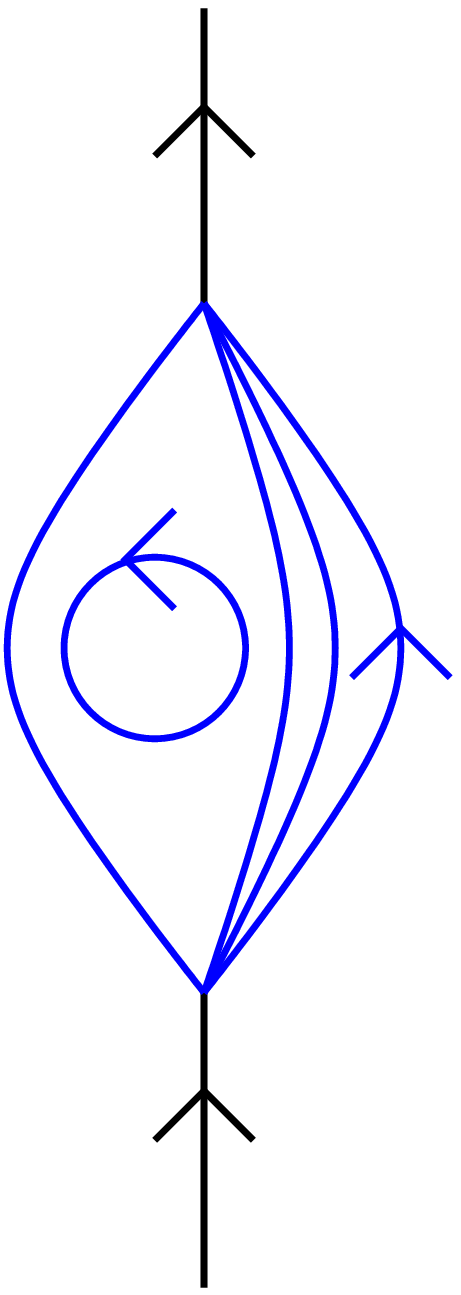}};
(12,0)*{(1^n)};
(-6,-4)*{\blue\scriptstyle{i}};
(0,-19)*{\scriptstyle \delta};
(6,6)*{\blue\scriptstyle{s-1}};
(1.8,4)*{\blue\bullet};
(-2,0)*{\blue\scriptstyle{m}};
(-3.6,4)*{\blue\bullet};
(-5.5,5.5)*{\blue\scriptstyle{t}};
\endxy
\end{equation}  
The new bubble, in the second diagram on the right-hand side of Equation~\eqref{eq:bubblemove}, still has color $m$ of course. But now it is in between 
the strands colored $m+2$ and $m+1$, reading from left to right. The label, 
$\lambda$, of the region between these two strands satisfies 
$\lambda_{m+1}=1$. Thus, by the degree zero bubble relations in~\cite{MTh}, 
the counter-clockwise degree zero $m$-bubble in that region is equal to 
one, which explains the 
positive sign of the first term on the right-hand side 
in~\eqref{eq:bubblemove}. Note that the label of the region 
containing the $m$-bubble in the second term satisfies 
$\lambda_m-\lambda_{m+1}=0$, so the dotless $m$-bubble has degree 2, 
as it should.  

Note that the $m$-bubble in the second term in~\eqref{eq:bubblemove} 
can be slid completely to the left-hand side. After that, 
we can use~\eqref{eq:polidigondotred} to eliminate the dot on the 
$m+1$th strand and slide the $m+1$-bubble completely to the 
left-hand side. Repeating this for 
all strands between $i-1$ and $m$, we get the following result
\begin{equation}
\label{eq:inductionstep}
\xy
(0,0)*{\figs{0.25}{DeltaEEDeltaCompl}};
(12,0)*{(1^n)};
(-6,-4)*{\blue\scriptstyle{i}};
(0,-19)*{\scriptstyle \delta};
(1.9,4)*{\blue\bullet};
(-3.6,4)*{\blue\bullet};
(-5.5,5.5)*{\blue\scriptstyle{t}};
(0,5)*{\blue\scriptstyle{s}};
\endxy
\;\;=\;\;
\xy
(0,0)*{\figs{0.25}{DeltaEEDeltaCompl}};
(12,0)*{(1^n)};
(-6,-4)*{\blue\scriptstyle{i}};
(0,-19)*{\scriptstyle \delta};
(1.9,4)*{\blue\bullet};
(-3.6,4)*{\blue\bullet};
(-7,5.5)*{\blue\scriptstyle{t+1}};
(0,2)*{\blue\scriptstyle{s-1}};
\endxy
\;-\;
\xy
(0,0)*{\figs{0.25}{DeltaEEDeltaCompl}};
(-6,-4)*{\blue\scriptstyle{i}};
(0,-19)*{\scriptstyle \delta};
(1.9,4)*{\blue\bullet};
(-3.6,4)*{\blue\bullet};
(-5.5,5.5)*{\blue\scriptstyle{t}};
(0,2)*{\blue\scriptstyle{s-1}};
(-43,0)*{\left({\blue \nccbub}+{\blue \nccbub}+\cdots+{\blue \nccbub} \ \ \right)};
(-59,4)*{\blue\scriptstyle{i-1}};
(-42,4)*{\blue\scriptstyle{i-2}};
(-16.5,4)*{\blue\scriptstyle{m}};
(10,0)*{(1^n)};
\endxy
\end{equation}
Induction then proves the result for $m\ne i+1$.
\vskip0.5cm
For $m=i+1$ we have to adapt our reasoning above, because 
the region between the $i+2$-th and the $i+1$-th strands has label 
$\lambda=(1^i,2,0,1^{n-(i+2)})$. In particular $\lambda_{i+1}=2$, so the left 
$i+1$-curl has degree four this time, which prevents us from using induction. 
Therefore, we use a slightly different argument involving a right curl. 

We still assume that $s>0$ holds. First note that, 
by the resolution of the curl and the degree zero bubble relations 
in~\cite{MTh}, we have 
\begin{equation}
\label{eq:rightcurlrel}
0
\;\;
=
\;\;
\xy
(0,0)*{\figs{0.25}{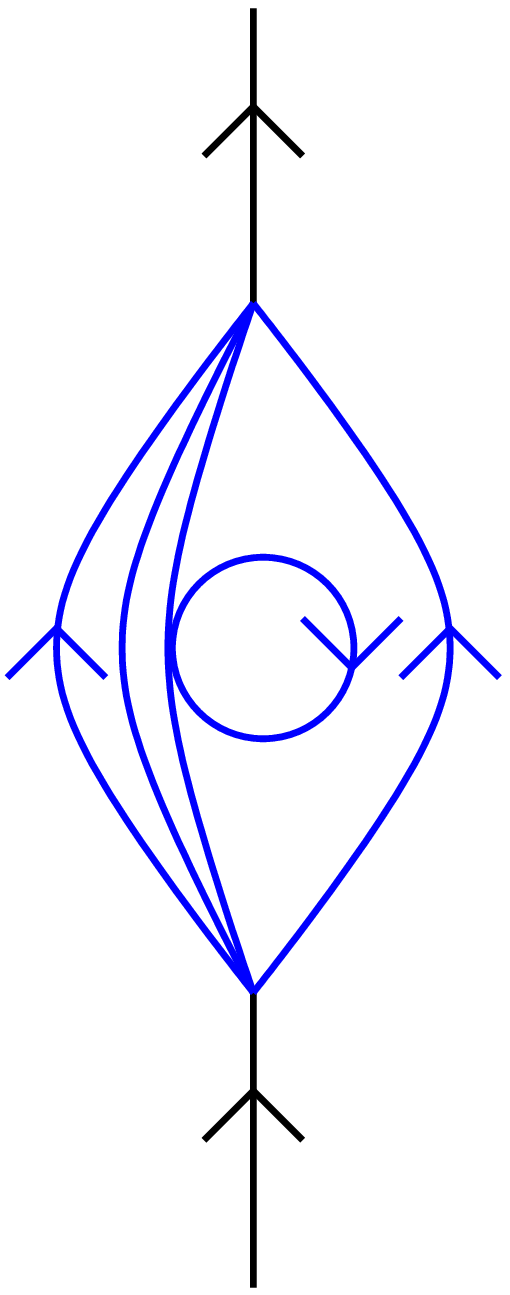}};
(12,0)*{(1^n)};
(-6,-4)*{\blue\scriptstyle{i}};
(1,-3.6)*{\blue\scriptstyle{i+2}};
(0,-19)*{\scriptstyle \delta};
(-3.6,4)*{\blue\bullet};
(-5.5,5.5)*{\blue\scriptstyle{t}};
(3.6,4)*{\blue\bullet};
(7,5.5)*{\blue\scriptstyle{s-1}};
\endxy
\;\;
=
\;\;
\xy
(0,0)*{\figs{0.25}{Polidigonopen}};
(12,0)*{(1^n)};
(-6,-4)*{\blue\scriptstyle{i}};
(0,-19)*{\scriptstyle \delta};
(2,4)*{\blue\bullet};
(0,-2)*{\blue\scriptstyle{i+2}};
(-3.6,4)*{\blue\bullet};
(-5.5,5.5)*{\blue\scriptstyle{t}};
(3.6,4)*{\blue\bullet};
(7,5.5)*{\blue\scriptstyle{s-1}};
\endxy
\;\;
-
\;\;
\xy
(0,0)*{\figs{0.25}{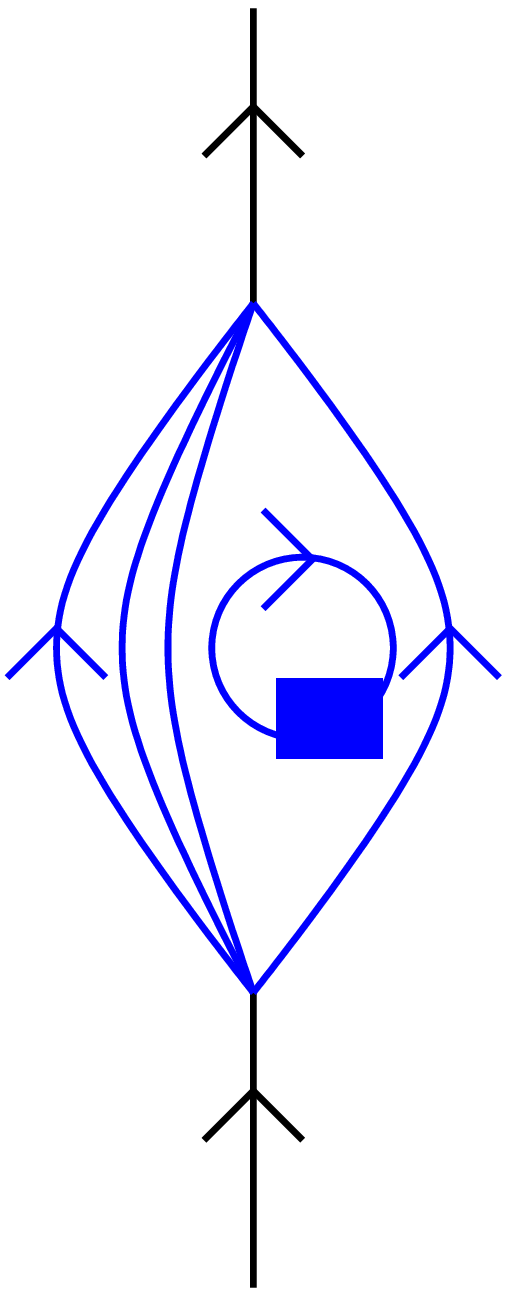}};
(12,0)*{(1^n)};
(-6,-4)*{\blue\scriptstyle{i}};
(0,-19)*{\scriptstyle \delta};
(1,-4)*{\blue\scriptstyle{-1}};
(-3.6,4)*{\blue\bullet};
(-5.5,5.5)*{\blue\scriptstyle{t}};
(3.6,4)*{\blue\bullet};
(7,5.5)*{\blue\scriptstyle{s-1}};
\endxy
\end{equation}
because the region between the $i+2$-th and the $i+1$-the strands 
is labeled $\lambda=(1^i,2,0,1^{n-(i+2)})$. In particular, 
we have $\lambda_{i+2}-\lambda_{i+3}=-1$ and $\lambda_{i+3}=1$, which explains 
the signs of the terms on the right-hand side of~\eqref{eq:rightcurlrel}. 

We now slide the $i+2$-bubble in the second term on the 
right-hand side of~\eqref{eq:rightcurlrel} to the right:
\begin{equation}
\label{eq:rightbubblemove}
\xy
(0,0)*{\figs{0.25}{Polidigonbubble3}};
(12,0)*{(1^n)};
(-6,-4)*{\blue\scriptstyle{i}};
(0,-19)*{\scriptstyle \delta};
(1,-4)*{\blue\scriptstyle{-1}};
(-3.6,4)*{\blue\bullet};
(-5.5,5.5)*{\blue\scriptstyle{t}};
(3.6,4)*{\blue\bullet};
(7,5.5)*{\blue\scriptstyle{s-1}};
\endxy  
\;\;
=
\;\;
\xy
(0,0)*{\figs{0.25}{Polidigonopen}};
(12,0)*{(1^n)};
(-6,-4)*{\blue\scriptstyle{i}};
(0,-19)*{\scriptstyle \delta};
(3.6,4)*{\blue\bullet};
(7,5.5)*{\blue\scriptstyle{s}};
(-3.6,4)*{\blue\bullet};
(-5.5,5.5)*{\blue\scriptstyle{t}};
\endxy
\;\;
+
\;\;
\xy
(0,0)*{\figs{0.25}{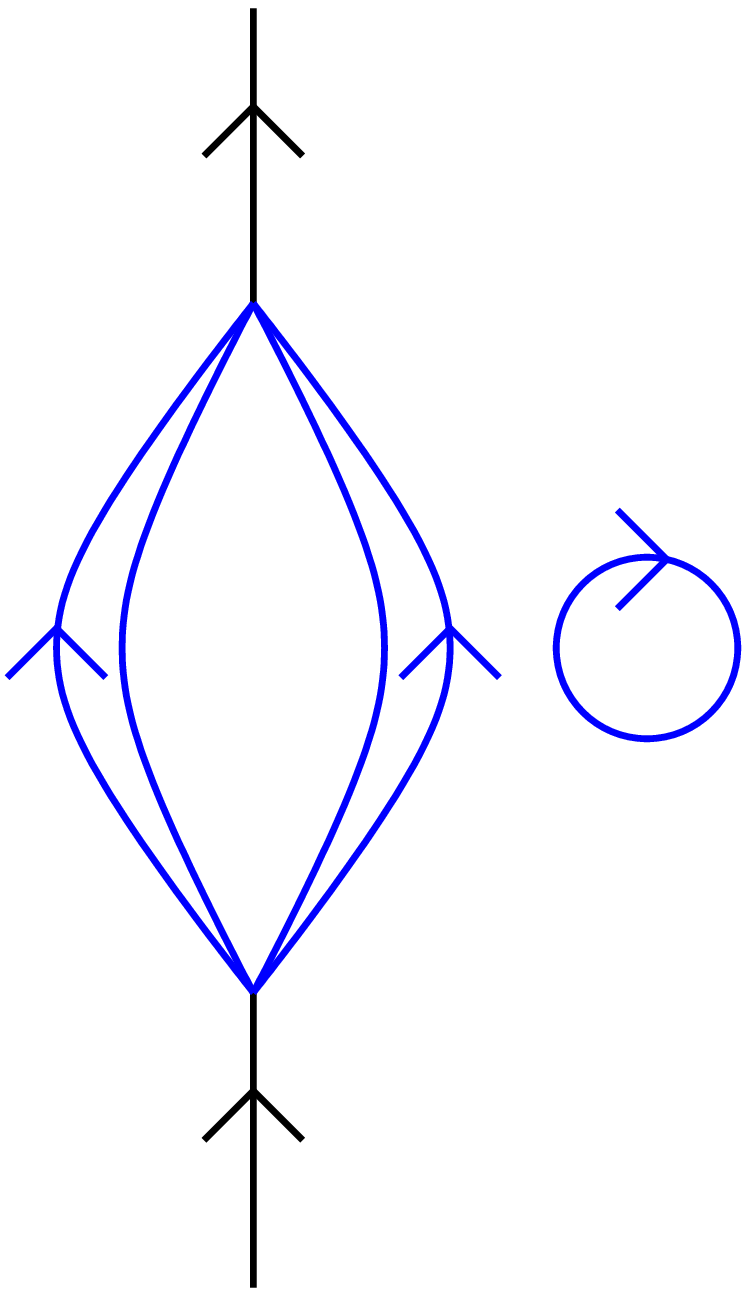}};
(16,0)*{(1^n)};
(-9,-4)*{\blue\scriptstyle{i}};
(0,-19)*{\scriptstyle \delta};
(10,-4)*{\blue\scriptstyle{i+2}};
(-6.5,4)*{\blue\bullet};
(-8.7,5.5)*{\blue\scriptstyle{t}};
(0.8,4)*{\blue\bullet};
(4.8,5.5)*{\blue\scriptstyle{s-1}};
\endxy  
\end{equation}
The sign of the first term on the right-hand side 
of~\eqref{eq:rightbubblemove} follows from the degree zero bubble relations 
in~\cite{MTh}.

Putting~\eqref{eq:rightcurlrel} and~\eqref{eq:rightbubblemove} together, we 
get
\begin{equation}
\label{eq:dotrighttoleft}
\xy
(0,0)*{\figs{0.25}{Polidigonopen}};
(12,0)*{(1^n)};
(-6,-4)*{\blue\scriptstyle{i}};
(0,-19)*{\scriptstyle \delta};
(-3.6,4)*{\blue\bullet};
(-5.5,5.5)*{\blue\scriptstyle{t}};
(3.6,4)*{\blue\bullet};
(7,5.5)*{\blue\scriptstyle{s}};
\endxy
\;\;
=
\;\;
\xy
(0,0)*{\figs{0.25}{Polidigonopen}};
(12,0)*{(1^n)};
(-6,-4)*{\blue\scriptstyle{i}};
(0,-19)*{\scriptstyle \delta};
(2,4)*{\blue\bullet};
(0,-2)*{\blue\scriptstyle{i+2}};
(-3.6,4)*{\blue\bullet};
(-5.5,5.5)*{\blue\scriptstyle{t}};
(3.6,4)*{\blue\bullet};
(8.5,5.5)*{\blue\scriptstyle{s-1}};
\endxy
\;\;
-
\;\;
\xy
(0,0)*{\figs{0.25}{Polidigonbubble4}};
(16,0)*{(1^n)};
(-9,-4)*{\blue\scriptstyle{i}};
(0,-19)*{\scriptstyle \delta};
(10,-4)*{\blue\scriptstyle{i+2}};
(-6.5,4)*{\blue\bullet};
(-8.7,5.5)*{\blue\scriptstyle{t}};
(0.8,4)*{\blue\bullet};
(4.8,5.5)*{\blue\scriptstyle{s-1}};
\endxy.  
\end{equation}
We can exchange the $i+2$-bubble on the right-hand side 
for an $i+1$-bubble on the left-hand side by 
Lemma~\ref{lem:polidigonreswithoutdots}, and invert its orientation by 
the infinite Grassmannian relation.    

By the same reasoning as above, we get  
\begin{equation}
\label{eq:inductionstepright2}
\xy
(0,0)*{\figs{0.25}{Polidigonopen}};
(12,0)*{(1^n)};
(-6,-4)*{\blue\scriptstyle{i}};
(0,-19)*{\scriptstyle \delta};
(2,4)*{\blue\bullet};
(0,-2)*{\blue\scriptstyle{i+2}};
(-3.6,4)*{\blue\bullet};
(-5.5,5.5)*{\blue\scriptstyle{t}};
(3.6,4)*{\blue\bullet};
(8.5,5.5)*{\blue\scriptstyle{s-1}};
\endxy
\;\;
=
\;\;
\xy
(0,0)*{\figs{0.25}{Polidigonopen}};
(12,0)*{(1^n)};
(-6,-4)*{\blue\scriptstyle{i}};
(0,-19)*{\scriptstyle \delta};
(-3.6,4)*{\blue\bullet};
(-5.5,5.5)*{\blue\scriptstyle{t+1}};
(3.6,4)*{\blue\bullet};
(7,5.5)*{\blue\scriptstyle{s-1}};
\endxy
\;\;
-
\;\;
\xy
(5,0)*{\figs{0.25}{DeltaEEDelta}};
(5,-17)*{\scriptstyle \delta};
(3.3,4)*{\blue\bullet};
(1,5.5)*{\blue\scriptstyle{t}};
(-35,0)*{\left({\blue \nccbub}+{\blue \nccbub}+\cdots+{\blue \nccbub} \ \quad \right)};
(-52,4)*{\blue\scriptstyle{i-1}};
(-35,4)*{\blue\scriptstyle{i-2}};
(-9,4)*{\blue\scriptstyle{i+2}};
(1.5,-4)*{\blue\scriptstyle{i}};
(6.7,4)*{\blue\bullet};
(10,5.5)*{\blue\scriptstyle{s-1}};
\endxy.
\end{equation}

Putting~\eqref{eq:dotrighttoleft} and~\eqref{eq:inductionstepright2} 
together, we obtain
\begin{equation}\label{eq:finalreduction}
\xy
(0,0)*{\figs{0.25}{Polidigonopen}};
(12,0)*{(1^n)};
(-6,-4)*{\blue\scriptstyle{i}};
(0,-19)*{\scriptstyle \delta};
(-3.6,4)*{\blue\bullet};
(-5.5,5.5)*{\blue\scriptstyle{t}};
(3.6,4)*{\blue\bullet};
(7,5.5)*{\blue\scriptstyle{s}};
\endxy
\;
=\;
\xy
(0,0)*{\figs{0.25}{Polidigonopen}};
(12,0)*{(1^n)};
(-6,-4)*{\blue\scriptstyle{i}};
(0,-19)*{\scriptstyle \delta};
(-3.6,4)*{\blue\bullet};
(-5.5,5.5)*{\blue\scriptstyle{t+1}};
(3.6,4)*{\blue\bullet};
(7,5.5)*{\blue\scriptstyle{s-1}};
\endxy
\;
-
\;
\xy
(5,0)*{\figs{0.25}{DeltaEEDelta}};
(5,-17)*{\scriptstyle \delta};
(3.3,4)*{\blue\bullet};
(1,5.5)*{\blue\scriptstyle{t}};
(-35,0)*{\left({\blue \nccbub}+{\blue \nccbub}+\cdots+{\blue \nccbub} \ \quad \right)};
(-52,4)*{\blue\scriptstyle{i-1}};
(-35,4)*{\blue\scriptstyle{i-2}};
(-9,4)*{\blue\scriptstyle{i+1}};
(1.5,-4)*{\blue\scriptstyle{i}};
(6.7,4)*{\blue\bullet};
(10,5.5)*{\blue\scriptstyle{s-1}};
\endxy.  
\end{equation}
As before, the result follows by induction. 
\end{proof}

\begin{prop}
\label{prop:polidigonreswithdots}

\begin{gather*}
\xy
(0,0)*{\figs{0.25}{DeltaEEDeltaCompl}};
(12,0)*{(1^n)};
(-6,-4)*{\blue\scriptstyle{i}};
(0,-19)*{\scriptstyle \delta};
(1.9,4)*{\blue\bullet};
(-3.6,4)*{\blue\bullet};
(3.6,4)*{\blue\bullet};
(-1.9,4)*{\blue\bullet};
(-6.5,5.5)*{\blue\scriptstyle{s_{i}}};
(7,5.5)*{\blue\scriptstyle{s_{i+1}}};
(-5,9)*{\blue\scriptstyle{s_{i-1}}};
(5,9)*{\blue\scriptstyle{s_{i+2}}};
\endxy
\;\;=\\
\sum_{j_{i-1}=0}^{s_{i-1}}\sum_{j_{i-2}=0}^{s_{i-2}}\cdots
\sum_{j_{i+1}=0}^{s_{i+1}}
\binom{s_{i-1}}{j_{i-1}}\cdots\binom{s_{i+1}}{j_{i+1}}
\;\;\xy
(-20,0)*{\blue\bbox{z_{i-1}^{j_{i-1}}\cdots 
z_{i+1}^{j_{i+1}}}};
(23,0)*{\black\xybox{(0,14);(0,-14); **\dir{-} ?(.5)*\dir{>}+
(2.3,0)*{\scriptstyle{}};}};
(23,-17)*{\scriptstyle \delta};
(0,0)*{\blue\cbub{s_{i}+\cdots+s_{i+1}-j_{i-1}-\cdots-j_{i+1}}{i}};
(28,0)*{(1^n)};
\endxy
\;=\;\\
\sum_{j_{i}=0}^{s_{i}}\sum_{j_{i-1}=0}^{s_{i-1}}\cdots
\sum_{j_{i+2}=0}^{s_{i+2}}
\binom{s_{i}}{j_{i}}\cdots\binom{s_{i+2}}{j_{i+2}}
\;\;\xy
(-20,0)*{\blue\bbox{y_{i}^{j_{i}}\cdots 
y_{i+2}^{j_{i+2}}}};
(-35,0)*{\black\xybox{(0,14);(0,-14); **\dir{-} ?(.5)*\dir{>}+
(2.3,0)*{\scriptstyle{}};}};
(-35,-17)*{\scriptstyle \delta};
(0,0)*{\blue\ccbub{s_{i}+\cdots+s_{i+1}-j_{i}-\cdots-j_{i+2}}{i+1\ \ \ \ }};
(28,0)*{(1^n)};
\endxy
\end{gather*}
\end{prop}
\begin{proof}
We only prove the first equation. The second can be proved by similar 
arguments. 

We use induction w.r.t. the reverse lexicographical ordering of 
the dot sequences 
$(s_i,\ldots,s_{i+1})$. The base of the induction, $s_i=\dots=s_{i+1}=0$, has 
been dealt with in Lemma~\ref{lem:polidigonreswithoutdots}.  

The case $s_{i-1}=\cdots=s_{i+1}=0$ has been dealt with in 
Lemma~\ref{lem:polidigonreswithdots}. Suppose there exists a 
$j\in\{i-1,\ldots,i+1\}$ with $s_j > 0$. The argument below works for 
arbitrary $j$, but let us assume that $j=i-1$ for simplicity. 

By the same arguments as used in the proof of Lemma~\ref{lem:polidigonreswithdots}, we get 
\begin{equation}
\label{eq:inductiongeneralcase}
\xy
(0,0)*{\figs{0.25}{DeltaEEDeltaCompl}};
(12,0)*{(1^n)};
(-6,-4)*{\blue\scriptstyle{i}};
(0,-19)*{\scriptstyle \delta};
(1.9,4)*{\blue\bullet};
(-3.6,4)*{\blue\bullet};
(3.6,4)*{\blue\bullet};
(-1.9,4)*{\blue\bullet};
(-6.5,5.5)*{\blue\scriptstyle{s_{i}}};
(7,5.5)*{\blue\scriptstyle{s_{i+1}}};
(-5,9)*{\blue\scriptstyle{s_{i-1}}};
(5,9)*{\blue\scriptstyle{s_{i+2}}};
\endxy
=
\xy
(0,0)*{\figs{0.25}{DeltaEEDeltaCompl}};
(12,0)*{(1^n)};
(-6,-4)*{\blue\scriptstyle{i}};
(0,-19)*{\scriptstyle \delta};
(1.9,4)*{\blue\bullet};
(-3.6,4)*{\blue\bullet};
(3.6,4)*{\blue\bullet};
(-1.9,4)*{\blue\bullet};
(-6.5,5.5)*{\blue\scriptstyle{s_{i}+1}};
(7,5.5)*{\blue\scriptstyle{s_{i+1}}};
(-5,9)*{\blue\scriptstyle{s_{i-1}-1}};
(5,9)*{\blue\scriptstyle{s_{i+2}}};
\endxy
\;-\;
\xy
(0,0)*{\figs{0.25}{DeltaEEDeltaCompl}};
(12,0)*{(1^n)};
(-6,-4)*{\blue\scriptstyle{i}};
(0,-19)*{\scriptstyle \delta};
(1.9,4)*{\blue\bullet};
(-3.6,4)*{\blue\bullet};
(3.6,4)*{\blue\bullet};
(-1.9,4)*{\blue\bullet};
(-6.5,5.5)*{\blue\scriptstyle{s_{i}}};
(7,5.5)*{\blue\scriptstyle{s_{i+1}}};
(-5,9)*{\blue\scriptstyle{s_{i-1}-1}};
(5,9)*{\blue\scriptstyle{s_{i+2}}};
(-13,0)*{\blue \nccbub};
(-18,5)*{\blue\scriptstyle{i-1}};
\endxy
\end{equation}
Induction on both terms on the right-hand side 
of~\eqref{eq:inductiongeneralcase} proves the proposition.  
\end{proof}

Proposition~\ref{prop:polidigonreswithdots} also allows us to derive 
two bubble slide formulas. The other two, for bubbles with the 
opposite orientation, can be obtained using the infinite Grassmannian 
relation and induction. Since we do not need them in this paper, 
we omit them.  
\begin{cor}
\label{cor:bubbleslides}
We have 
\begin{equation}
\label{eq:bubbleslides1}
\sum_{j=0}^{s}\binom{s}{j}
\xy
(0,0)*{\black\xybox{(0,14);(0,-14); **\dir{-} ?(.5)*\dir{>}+
(2.3,0)*{\scriptstyle{}};}};
(-1,-17)*{\scriptstyle \delta};
(-20,0)*{{\blue\cbub{s-j}{i}}\;{\blue\bbox{z_{i+1}^j}}};
(5,0)*{(1^n)};
\endxy
\;\;
=
\;\;
\xy
(0,0)*{\black\xybox{(0,14);(0,-14); **\dir{-} ?(.5)*\dir{>}+
(2.3,0)*{\scriptstyle{}};}};
(-1,-17)*{\scriptstyle \delta};
(10,0)*{\blue\ccbub{s}{i+1\ \ \ \ }};
(20,0)*{(1^n)};
\endxy
\end{equation}
and 
\begin{equation}
\label{eq:bubbleslides2}
\xy
(0,0)*{\black\xybox{(0,14);(0,-14); **\dir{-} ?(.5)*\dir{>}+
(2.3,0)*{\scriptstyle{}};}};
(-1,-17)*{\scriptstyle \delta};
(-7,0)*{(1^n)};
\endxy
\sum_{j=0}^{s}\binom{s}{j}
\xy
(10,0)*{{\blue\ccbub{s-j}{i+1\ \ \ \ }}\;{\blue\bbox{y_{i}^j}}};
\endxy
\;\;
=
\;\;
\xy
(13,0)*{\black\xybox{(0,14);(0,-14); **\dir{-} ?(.5)*\dir{>}+
(2.3,0)*{\scriptstyle{}};}};
(12,-17)*{\scriptstyle \delta};
(0,0)*{\blue\cbub{s}{i}};
(-12,0)*{(1^n)};
\endxy
\end{equation}
\end{cor}
\begin{proof}
These two bubble slide relations follow immediately from 
Lemma~\ref{lem:polidigonreswithdots}. For~\eqref{eq:bubbleslides1}, 
apply~\eqref{eq:polidigonreswithdots} and~\eqref{eq:polidigonreswithdots2} 
with $t=0, m=i+1$. For~\eqref{eq:bubbleslides1}, 
apply~\eqref{eq:polidigonreswithdots} and~\eqref{eq:polidigonreswithdots2} 
with $s=0, m=i$. 
\end{proof}
\section{Two useful $2$-functors}

\begin{defn}
\label{defn:Psi}
Let the $2$-functor $\Psi_{n,n}\colon \Ucataff^*\to \Scat(n,n)^*$ be 
defined just as $\Psi_{n,r}$ in Section 3.5.3 
in~\cite{MTh}, i.e. on objects and $1$-morphisms it is determined by 
\begin{eqnarray*}
\mu&\mapsto&\phi_{n,n}(\mu)=:\lambda\\
\mathcal{E}_{\ii}{\mathbf 1}_{\mu}&\mapsto&\mathcal{E}_{\ii}{\mathbf 1}_{\lambda}.
\end{eqnarray*}
By convention, we put $1_{*}:=0$. On $2$-morphisms it is determined 
by sending any diagram in $\Ucataff$ which is not a left cap or cup 
to the same diagram in $\Scat(n,n)$ and applying $\phi_{n,n}$ to 
the labels of the regions in the diagram. 
The images of the left caps and cups also have to 
be multiplied by certain signs. To be more precise, define 
\begin{equation}
\label{eq:signs}
\Ucapli_{i,\mu}\mapsto (-1)^{\lambda_{i+1}+1}\,\,\Ucapli_{i,\lambda}
\quad\mbox{and}\quad \Ucupli_{i,\mu}\mapsto (-1)^{\lambda_{i+1}}\,\,
\Ucupli_{i,\lambda}.
\end{equation}  
We define any diagram in $\Scat(n,n)$ to be equal to zero 
if it contains regions labeled $*$.
\end{defn}

Note that, unlike $\Psi_{n,r}$ for $n>r$, $\Psi_{n,n}$ is not 
essentially surjective. However, it still has the following useful 
property. 
\begin{lem}
\label{lem:full}
The $2$--functor $\Psi_{n,n}$ is full. 
\end{lem}
\begin{proof}
The proof follows from the following two observations, which show how 
to remove $\delta$-strands from diagrams in $\mathrm{HOM}_{\Scat(n,n)}
(\mathcal{E}_{\ii}1_{\lambda},\mathcal{E}_{\jj}1_{\lambda})$, for any signed sequences $\ii$ and 
$\jj$: 
\begin{itemize}
\item Closed $\delta$-diagrams always 
consist of disjoint $\delta$-circles. By Corollary~\ref{cor:bubbleslides} 
we can move any closed 
$i$-diagram, which is always equivalent to a linear combination of 
disjoint $i$-circles, from the interior to the exterior of a $\delta$-circle. 
By~\eqref{deltabubbles}, we can then remove 
the $\delta$-circles with empty interior. 
\item Any $\delta$-strand which is not part of a $\delta$-circle 
has to be part of a diagram obtained by glueing $\DeltaEE_{\delta,{\blue j}}$ 
on top of 
$\EEDelta^{\delta,{\blue i}}$ or $\DeltaFF_{\delta,{\blue i}}$ on top of 
$\FFDelta^{\delta,{\blue j}}$, for certain $1\leq i,j\leq n$. In both cases we can 
remove the $\delta$-strand by applying~\eqref{EEDeltaEE} 
or~\eqref{EEDeltaEEPerm}.  
\end{itemize}
\end{proof}

\begin{defn}
\label{defn:In}
We define the $2$-functor $\mathcal{I}_n\colon \Scat(n,n)\to \Scat(n+1,n)$ 
as follows: 
\begin{itemize}
\item on objects and $1$-morphisms use the map in 
Proposition~\ref{prop:iota}; 
\item on $2$-morphisms take the identity on all 
$i$-strands, for $1\leq i\leq n-1$, map 
all $n$-strands to two parallel strands labeled $n$ and $n+1$, e.g. 
$$
\begin{array}{ccc}
\xy
(-5,0)*{\blue\xybox{(0,10);(0,-10); **\dir{-} ?(.5)*\dir{>}+
(2.3,0)*{\scriptstyle{}};}};
(-6,-12)*{\scriptstyle\blue n};
(0,0)*{(\lambda)};
\endxy
\mapsto 
\xy
(-5,0)*{\blue\xybox{(0,10);(0,-10); **\dir{-} ?(.5)*\dir{>}+
(2.3,0)*{\scriptstyle{}};}};
(0,0)*{\blue\xybox{(0,10);(0,-10); **\dir{-} ?(.5)*\dir{>}+
(2.3,0)*{\scriptstyle{}};}};
(-6,-12)*{\scriptstyle\blue n};
(0,-12)*{\scriptstyle\blue n+1};
(5,0)*{(\lambda,0)};
\endxy
\qquad
&
\qquad
\xy
(-5,0)*{\blue\xybox{(0,10);(0,-10); **\dir{-} ?(.5)*\dir{<}+
(2.3,0)*{\scriptstyle{}};}};
(-6,-12)*{\scriptstyle\blue n};
(0,0)*{(\lambda)};
\endxy
\mapsto 
\xy
(-5,0)*{\blue\xybox{(0,10);(0,-10); **\dir{-} ?(.5)*\dir{<}+
(2.3,0)*{\scriptstyle{}};}};
(0,0)*{\blue\xybox{(0,10);(0,-10); **\dir{-} ?(.5)*\dir{<}+
(2.3,0)*{\scriptstyle{}};}};
(-6,-12)*{\scriptstyle\blue n+1};
(0,-12)*{\scriptstyle\blue n};
(5,0)*{(\lambda,0)};
\endxy, 
\end{array}
$$
map dots on $n$-strands to dots on the corresponding 
pairs of parallel strands as follows
$$
\begin{array}{ccc}
\xy
(-5,0)*{\blue\xybox{(0,10);(0,-10); **\dir{-} ?(.5)*\dir{>}+
(2.3,0)*{\scriptstyle{}};}};
(-6,-12)*{\scriptstyle\blue n};
(0,0)*{(\lambda)};
(-6,5)*{\blue \bullet}; 
\endxy
\mapsto 
\xy
(-5,0)*{\blue\xybox{(0,10);(0,-10); **\dir{-} ?(.5)*\dir{>}+
(2.3,0)*{\scriptstyle{}};}};
(0,0)*{\blue\xybox{(0,10);(0,-10); **\dir{-} ?(.5)*\dir{>}+
(2.3,0)*{\scriptstyle{}};}};
(-6,-12)*{\scriptstyle\blue n};
(0,-12)*{\scriptstyle\blue n+1};
(5,0)*{(\lambda,0)};
(-6,5)*{\blue \bullet}; 
\endxy
=
\xy
(-5,0)*{\blue\xybox{(0,10);(0,-10); **\dir{-} ?(.5)*\dir{>}+
(2.3,0)*{\scriptstyle{}};}};
(0,0)*{\blue\xybox{(0,10);(0,-10); **\dir{-} ?(.5)*\dir{>}+
(2.3,0)*{\scriptstyle{}};}};
(-6,-12)*{\scriptstyle\blue n};
(0,-12)*{\scriptstyle\blue n+1};
(5,0)*{(\lambda,0)};
(-1,5)*{\blue\bullet}; 
\endxy
\qquad
&
\qquad
\xy
(-5,0)*{\blue\xybox{(0,10);(0,-10); **\dir{-} ?(.5)*\dir{<}+
(2.3,0)*{\scriptstyle{}};}};
(-6,-12)*{\scriptstyle\blue n};
(0,0)*{(\lambda)};
(-6,5)*{\blue\bullet};
\endxy
\mapsto 
\xy
(-5,0)*{\blue\xybox{(0,10);(0,-10); **\dir{-} ?(.5)*\dir{<}+
(2.3,0)*{\scriptstyle{}};}};
(0,0)*{\blue\xybox{(0,10);(0,-10); **\dir{-} ?(.5)*\dir{<}+
(2.3,0)*{\scriptstyle{}};}};
(-6,-12)*{\scriptstyle\blue n+1};
(0,-12)*{\scriptstyle\blue n};
(5,0)*{(\lambda,0)};
(-6,5)*{\blue\bullet};
\endxy 
=
\xy
(-5,0)*{\blue\xybox{(0,10);(0,-10); **\dir{-} ?(.5)*\dir{<}+
(2.3,0)*{\scriptstyle{}};}};
(0,0)*{\blue\xybox{(0,10);(0,-10); **\dir{-} ?(.5)*\dir{<}+
(2.3,0)*{\scriptstyle{}};}};
(-6,-12)*{\scriptstyle\blue n+1};
(0,-12)*{\scriptstyle\blue n};
(5,0)*{(\lambda,0)};
(-1,5)*{\blue\bullet};
\endxy, 
\end{array}
$$ 
and send the generators involving $\delta$-strands to   
$$
\begin{array}{cc}
\xy
(-5,0)*{\xybox{(0,10);(0,-10); **\dir{-} ?(.5)*\dir{>}+
(2.3,0)*{\scriptstyle{}};}};
(-6,-12)*{\scriptstyle \delta};
(0,0)*{(1^n)};
\endxy
\mapsto 
\xy
(-5,0)*{\blue\xybox{(0,10);(0,-10); **\dir{-} ?(.5)*\dir{>}+
(2.3,0)*{\scriptstyle{}};}};
(0,0)*{\blue\xybox{(0,10);(0,-10); **\dir{-} ?(.5)*\dir{>}+
(2.3,0)*{\scriptstyle{}};}};
(5,0)*{\blue\cdots};
(10,0)*{\blue\xybox{(0,10);(0,-10); **\dir{-} ?(.5)*\dir{>}+
(2.3,0)*{\scriptstyle{}};}};
(15,0)*{\blue\xybox{(0,10);(0,-10); **\dir{-} ?(.5)*\dir{>}+
(2.3,0)*{\scriptstyle{}};}};
(-6,-12)*{\scriptstyle\blue n};
(-1,-12)*{\scriptstyle\blue n-1};
(9,-12)*{\scriptstyle\blue 1};
(14,-12)*{\scriptstyle\blue n+1};
(20,0)*{(1^n)};
\endxy
\qquad
&
\qquad
\xy
(-5,0)*{\xybox{(0,10);(0,-10); **\dir{-} ?(.5)*\dir{<}+
(2.3,0)*{\scriptstyle{}};}};
(-6,-12)*{\scriptstyle \delta};
(0,0)*{(1^n)};
\endxy
\mapsto 
\xy
(-5,0)*{\blue\xybox{(0,10);(0,-10); **\dir{-} ?(.5)*\dir{<}+
(2.3,0)*{\scriptstyle{}};}};
(0,0)*{\blue\xybox{(0,10);(0,-10); **\dir{-} ?(.5)*\dir{<}+
(2.3,0)*{\scriptstyle{}};}};
(5,0)*{\blue\cdots};
(10,0)*{\blue\xybox{(0,10);(0,-10); **\dir{-} ?(.5)*\dir{<}+
(2.3,0)*{\scriptstyle{}};}};
(15,0)*{\blue\xybox{(0,10);(0,-10); **\dir{-} ?(.5)*\dir{<}+
(2.3,0)*{\scriptstyle{}};}};
(-6,-12)*{\scriptstyle\blue n+1};
(-1,-12)*{\scriptstyle\blue 1};
(9,-12)*{\scriptstyle\blue n-1};
(14,-12)*{\scriptstyle\blue n};
(20,0)*{(1^n)};
\endxy
\end{array}
$$

\begin{gather*}
\xy
(-10,17)*{};
(-12,14)*{\scriptstyle\blue i};
(-8,14)*{\scriptstyle\blue i-1};
(-2,14)*{\scriptstyle\blue 1};
(2,14)*{\scriptstyle\blue n};
(7,14)*{\scriptstyle\blue i+2};
(13,14)*{\scriptstyle\blue i+1};
(0,-15)*{\scriptstyle\delta};
(7,0)*{(1^n)};
(0,0)*{\figs{0.27}{DeltaEE}};
\endxy
\mapsto 
\xy
(-5,0)*{\figs{0.27}{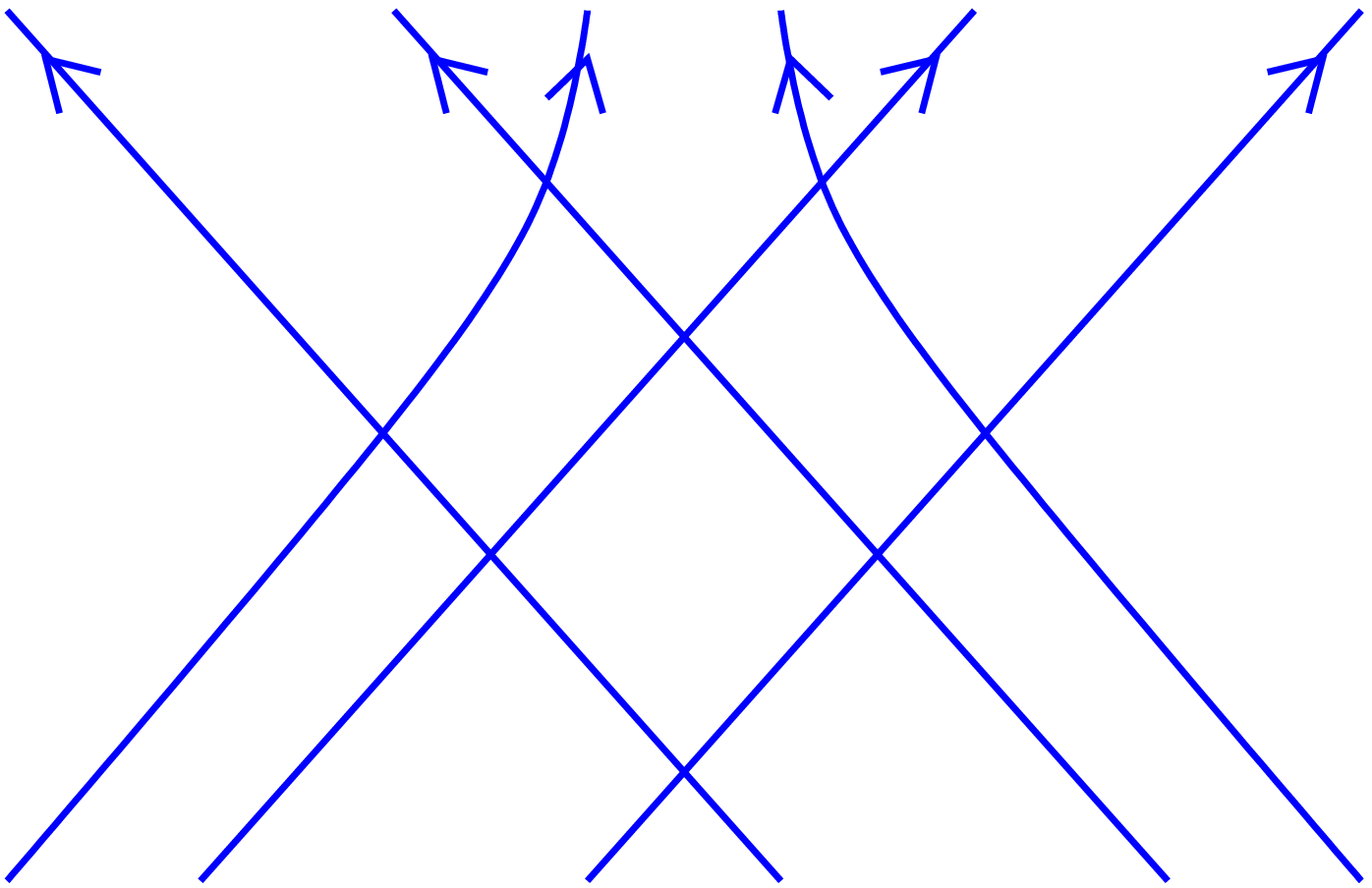}};
(-25,-15)*{\scriptstyle\blue n};
(-19,-15)*{\scriptstyle\blue n-1};
(-8,-15)*{\scriptstyle\blue i+1};
(-2,-15)*{\scriptstyle\blue i};
(9,-15)*{\scriptstyle\blue 1};
(16,-15)*{\scriptstyle\blue n+1};
(-25,15)*{\scriptstyle\blue i};
(-14,15)*{\scriptstyle\blue 1};
(-8,15)*{\scriptstyle\blue n};
(-2.5,15)*{\scriptstyle\blue n+1};
(4.5,15)*{\scriptstyle\blue n-1};
(16,15)*{\scriptstyle\blue i+1};
(18,0)*{(1^n)};
\endxy
\\ 
\xy
(-10,17)*{};
(-12,14)*{\scriptstyle\blue i};
(-8,14)*{\scriptstyle\blue i+1};
(-2,14)*{\scriptstyle\blue n};
(2,14)*{\scriptstyle\blue 1};
(7,14)*{\scriptstyle\blue i-2};
(13,14)*{\scriptstyle\blue i-1};
(0,-15)*{\scriptstyle\delta};
(7,0)*{(1^n)};
(0,0)*{\figs{0.27}{DeltaFF}};
\endxy
\mapsto 
\xy
(-5,0)*{\figs{0.27}{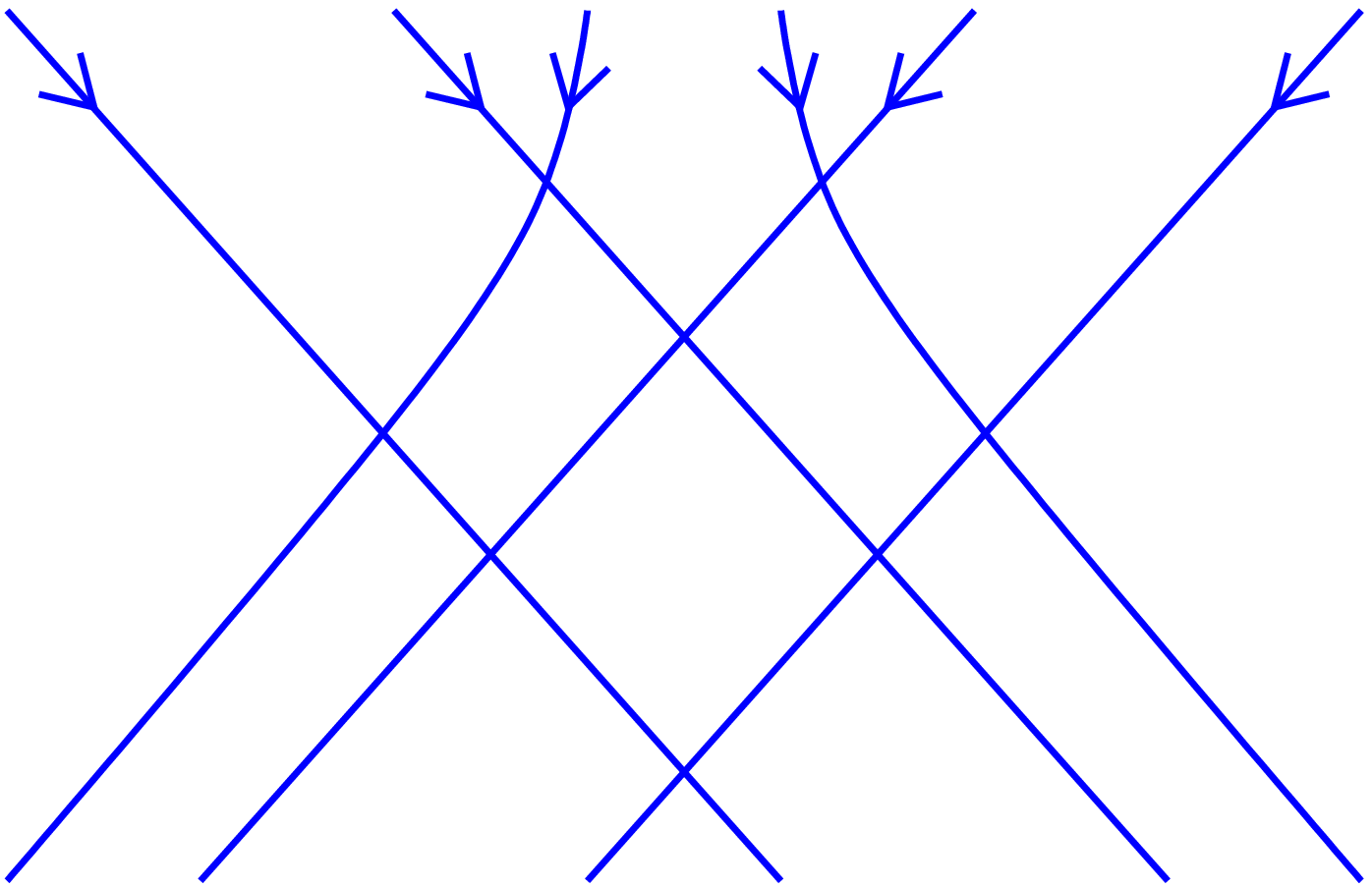}};
(-25,-15)*{\scriptstyle\blue n+1};
(-19,-15)*{\scriptstyle\blue 1};
(-8,-15)*{\scriptstyle\blue i-1};
(-2,-15)*{\scriptstyle\blue i};
(9,-15)*{\scriptstyle\blue n-1};
(16,-15)*{\scriptstyle\blue n};
(-25,15)*{\scriptstyle\blue i};
(-14,15)*{\scriptstyle\blue n-1};
(-8,15)*{\scriptstyle\blue n+1};
(-2.5,15)*{\scriptstyle\blue n};
(4.5,15)*{\scriptstyle\blue 1};
(16,15)*{\scriptstyle\blue i-1};
(18,0)*{(1^n)};
\endxy
\end{gather*}
with the image of the other two $\delta$-splitters being defined likewise 
using cyclicity. 
\end{itemize}
\end{defn}
\noindent Note that the two images of the dotted $n$-strands which are 
shown, are indeed equal in $\hat{\mathcal S}(n+1,n)$. 
This follows from the relevant Reidemeister 2 relations, because the diagrams 
with the crossings in those relations are equal to zero 
(the last entry of the labels of their middle regions is equal to $-1$). 

\begin{lem}
\label{lem:Iwell-defined}
For any $n\geq 3$, $\mathcal{I}_n$ is well-defined.
\end{lem}
\begin{proof}
We only have to prove that $\mathcal{I}_n$ preserves the relations 
involving $n$ and $\delta$-strands, because $\mathcal{I}_n$ clearly preserves 
all other relations. 

First consider the nilHecke relations which only 
involve $n$-strands. By cyclicity, we can assume that all strands are 
oriented upward. We give the proof of well-definedness w.r.t. one 
nilHecke relation in detail. The image of the left-hand side of 
\begin{equation}
\label{eq:nilHecke1}
\xy
(0,0)*{\blue\xybox{(0,0);(10,10); **\dir{-} ?(0)*\dir{<}+
(2.3,0)*{\scriptstyle{}};}};
(0,0)*{\blue\xybox{(0,10);(10,0); **\dir{-} ?(1)*\dir{>}+
(2.3,0)*{\scriptstyle{}};}};
(-3,3)*{\blue\bullet};
(-7,-7)*{\scriptstyle\blue n};
(5,-7)*{\scriptstyle\blue n};
(9,0)*{\lambda};
\endxy
\;
-
\;
\xy
(0,0)*{\blue\xybox{(0,0);(10,10); **\dir{-} ?(0)*\dir{<}+
(2.3,0)*{\scriptstyle{}};}};
(0,0)*{\blue\xybox{(0,10);(10,0); **\dir{-} ?(1)*\dir{>}+
(2.3,0)*{\scriptstyle{}};}};
(-7,-7)*{\scriptstyle\blue n};
(5,-7)*{\scriptstyle\blue n};
(3,-3)*{\blue\bullet};
(9,0)*{\lambda};
\endxy
\;
=
\;
\xy
(0,0)*{\blue\xybox{(0,0);(0,10); **\dir{-} ?(.5)*\dir{<}+
(2.3,0)*{\scriptstyle{}};}};
(10,0)*{\blue\xybox{(0,0);(0,10); **\dir{-} ?(.5)*\dir{<}+
(2.3,0)*{\scriptstyle{}};}};
(-1,-7)*{\scriptstyle\blue n};
(9,-7)*{\scriptstyle\blue n};
(14,0)*{\lambda};
\endxy 
\end{equation}
is given by 
$$
\xy
(0,0)*{\blue\xybox{(0,0);(10,10); **\dir{-} ?(0)*\dir{<}+
(2.3,0)*{\scriptstyle{}};}};
(0,0)*{\blue\xybox{(0,10);(10,0); **\dir{-} ?(1)*\dir{>}+
(2.3,0)*{\scriptstyle{}};}};
(-3,3)*{\blue\bullet};
(2.5,0)*{\blue\xybox{(0,0);(10,10); **\dir{-} ?(0)*\dir{<}+
(2.3,0)*{\scriptstyle{}};}};
(2.5,0)*{\blue\xybox{(0,10);(10,0); **\dir{-} ?(1)*\dir{>}+
(2.3,0)*{\scriptstyle{}};}};
(-7.5,-7)*{\scriptstyle\blue n};
(5,-7)*{\scriptstyle\blue n};
(-3,-7)*{\scriptstyle\blue n+1};
(9,-7)*{\scriptstyle\blue n+1};
(11.5,0)*{(\lambda,0)};
\endxy
\;
-
\;
\xy
(0,0)*{\blue\xybox{(0,0);(10,10); **\dir{-} ?(0)*\dir{<}+
(2.3,0)*{\scriptstyle{}};}};
(0,0)*{\blue\xybox{(0,10);(10,0); **\dir{-} ?(1)*\dir{>}+
(2.3,0)*{\scriptstyle{}};}};
(2.5,0)*{\blue\xybox{(0,0);(10,10); **\dir{-} ?(0)*\dir{<}+
(2.3,0)*{\scriptstyle{}};}};
(2.5,0)*{\blue\xybox{(0,10);(10,0); **\dir{-} ?(1)*\dir{>}+
(2.3,0)*{\scriptstyle{}};}};
(-7.5,-7)*{\scriptstyle\blue n};
(5,-7)*{\scriptstyle\blue n};
(-3,-7)*{\scriptstyle\blue n+1};
(9,-7)*{\scriptstyle\blue n+1};
(3,-3)*{\blue\bullet};
(11.5,0)*{(\lambda,0)};
\endxy.
$$ 
By the nilHecke relation for the $n$-strands, this is equal to 
$$
\xy
(0,0)*{\figs{0.15}{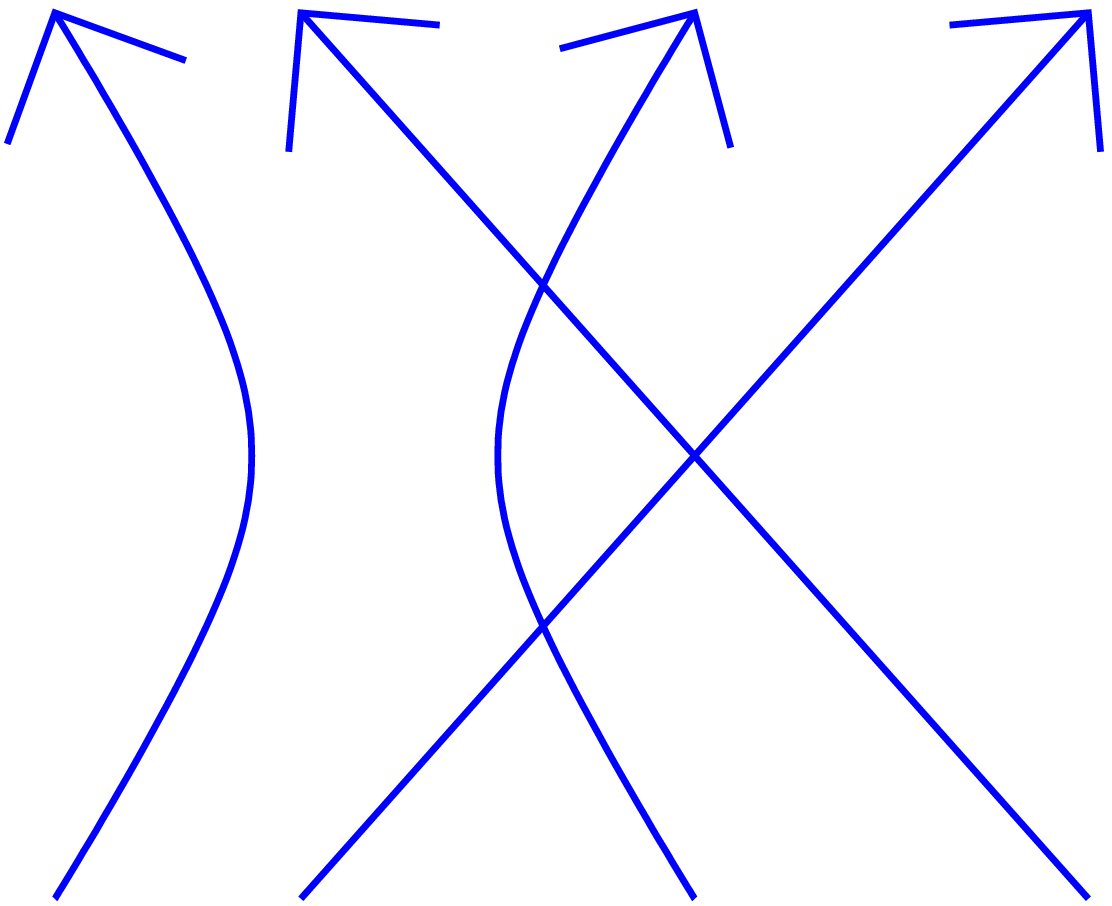}};
(-8.5,-9)*{\scriptstyle\blue n};
(2,-9)*{\scriptstyle\blue n};
(-4,-9)*{\scriptstyle\blue n+1};
(9,-9)*{\scriptstyle\blue n+1};
(11.5,0)*{(\lambda,0)};
\endxy
\;
=
\;
\xy
(0,0)*{\figs{0.15}{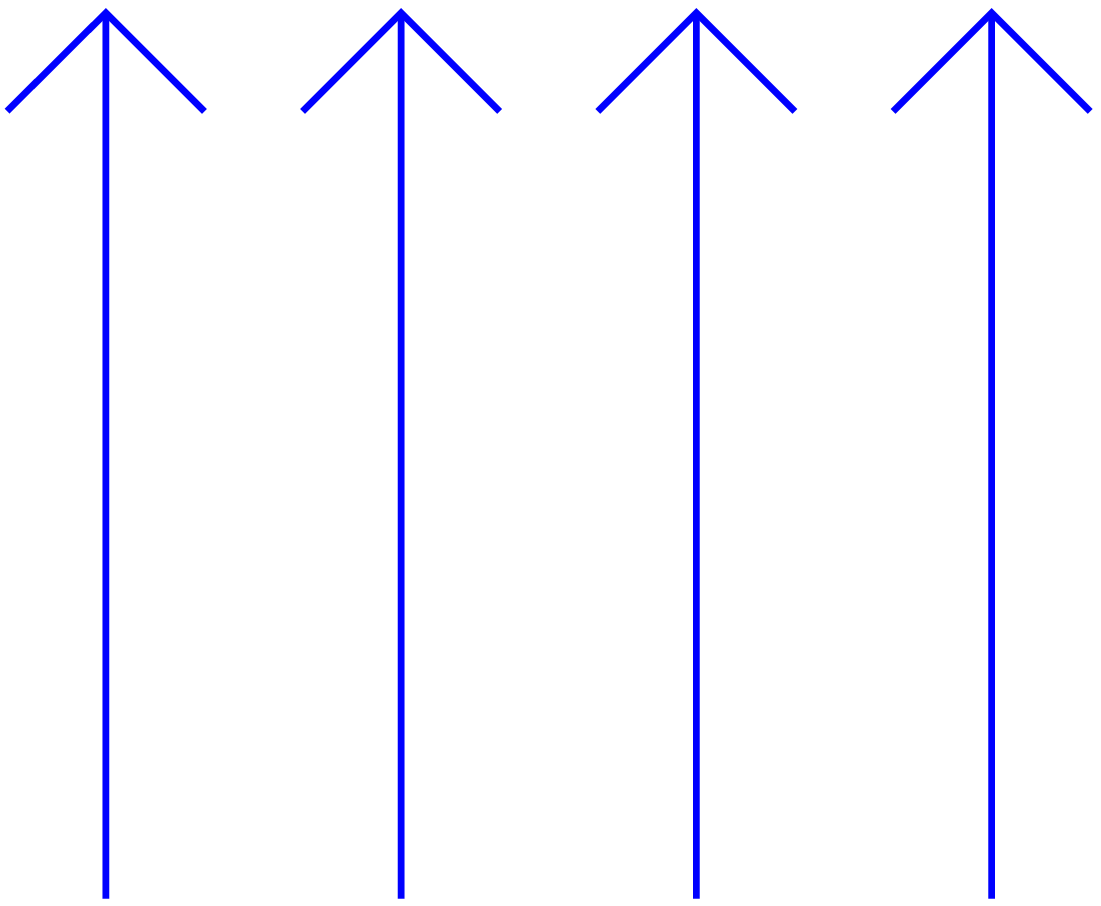}};
(-7,-9)*{\scriptstyle\blue n};
(2.5,-9)*{\scriptstyle\blue n};
(-2.5,-9)*{\scriptstyle\blue n+1};
(7,-9)*{\scriptstyle\blue n+1};
(13,0)*{(\lambda,0)};
\endxy,
$$
which is equal to the image of the right-hand side of~\eqref{eq:nilHecke1}. 
Note that in the last equality we have omitted one term, which is equal to 
zero because it contains a region whose label has a negative entry. 

Well-definedness w.r.t. the other two nilHecke relations for $n$-strands 
can be proved by similar arguments.

As for the other relations involving only $n$-strands, the first one we should 
have a look at is the infinite Grassmannian relation. The image of the 
$n$-bubbles is given by 
\begin{gather*}
\xy
(0,0)*{\figs{0.15}{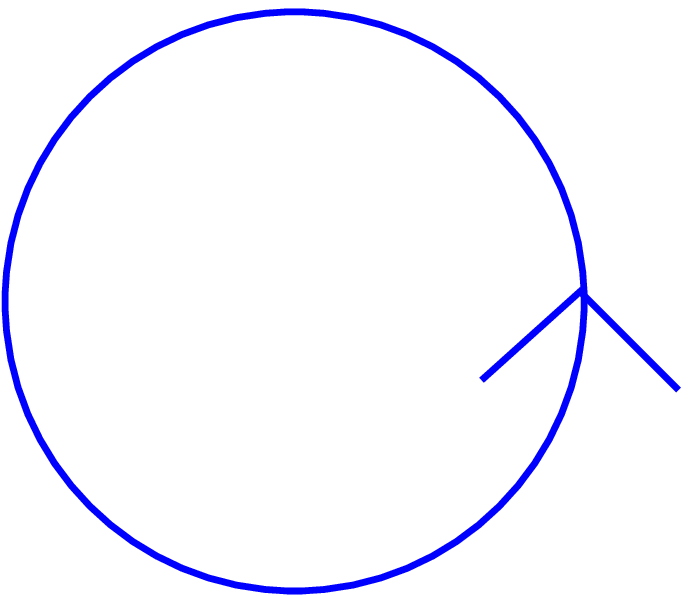}};
(-5,5.5)*{\scriptstyle\blue n};
(7.5,0)*{\lambda};
(2.5,-3)*{\blue \bullet};
(3,-6)*{\scriptstyle\blue \spadesuit + a};
\endxy
\;\mapsto\;
\xy
(0,0)*{\figs{0.15}{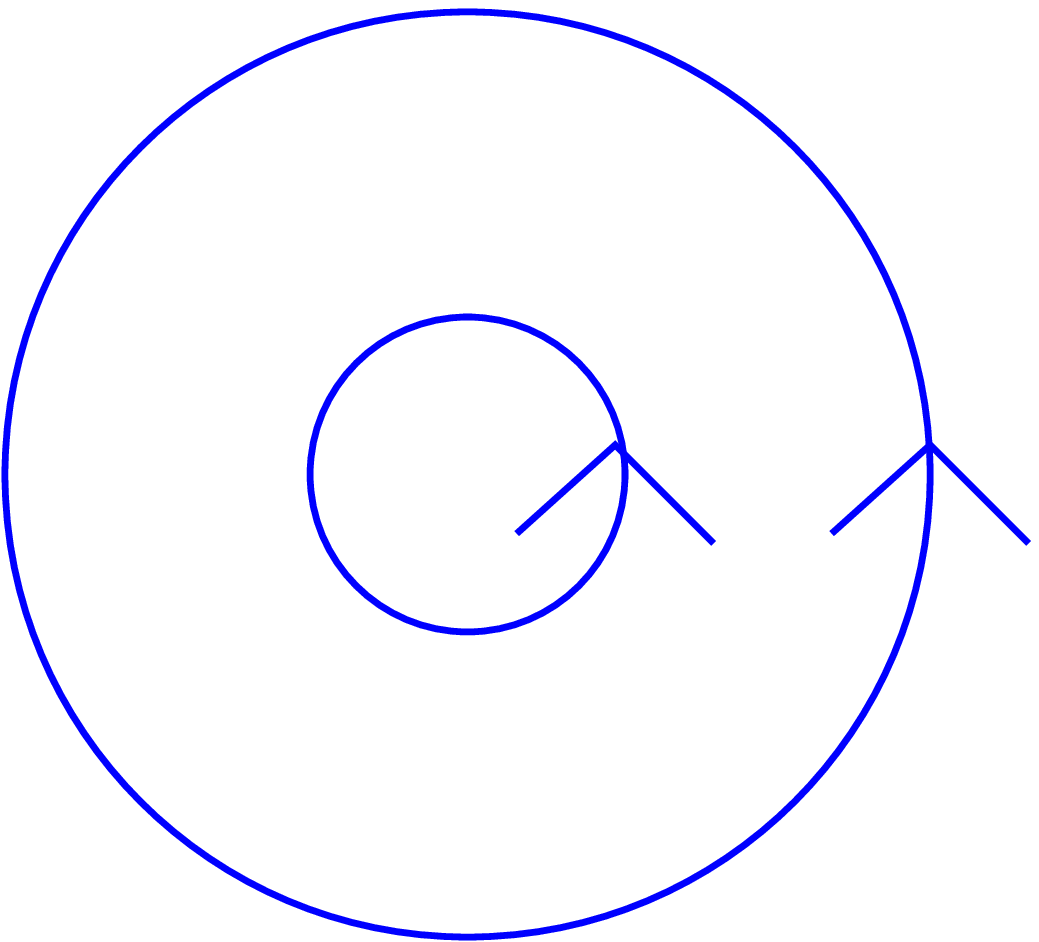}};
(-7,8)*{\scriptstyle\blue n+1};
(-3,4)*{\scriptstyle\blue n};
(13,0)*{(\lambda,0)};
(1,-2)*{\blue\bullet};
(6,-2)*{\blue\bullet};
(1.5,-4)*{\scriptstyle\blue \spadesuit + a};
(8,-4)*{\scriptstyle\blue \spadesuit};
\endxy
\\
\xy
(0,0)*{\figs{0.15}{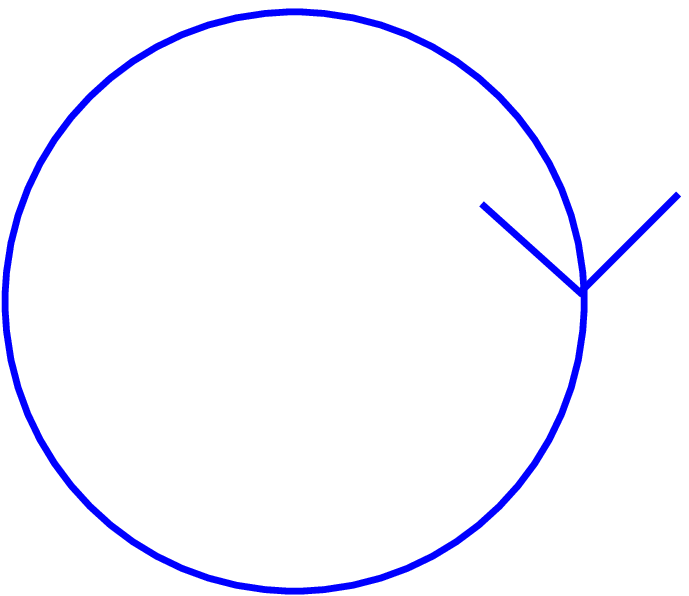}};
(-5,5.5)*{\scriptstyle\blue n};
(7.5,0)*{\lambda};
(2.5,-3)*{\blue \bullet};
(3,-6)*{\scriptstyle\blue \spadesuit + a};
\endxy
\;\mapsto\;
\xy
(0,0)*{\figs{0.15}{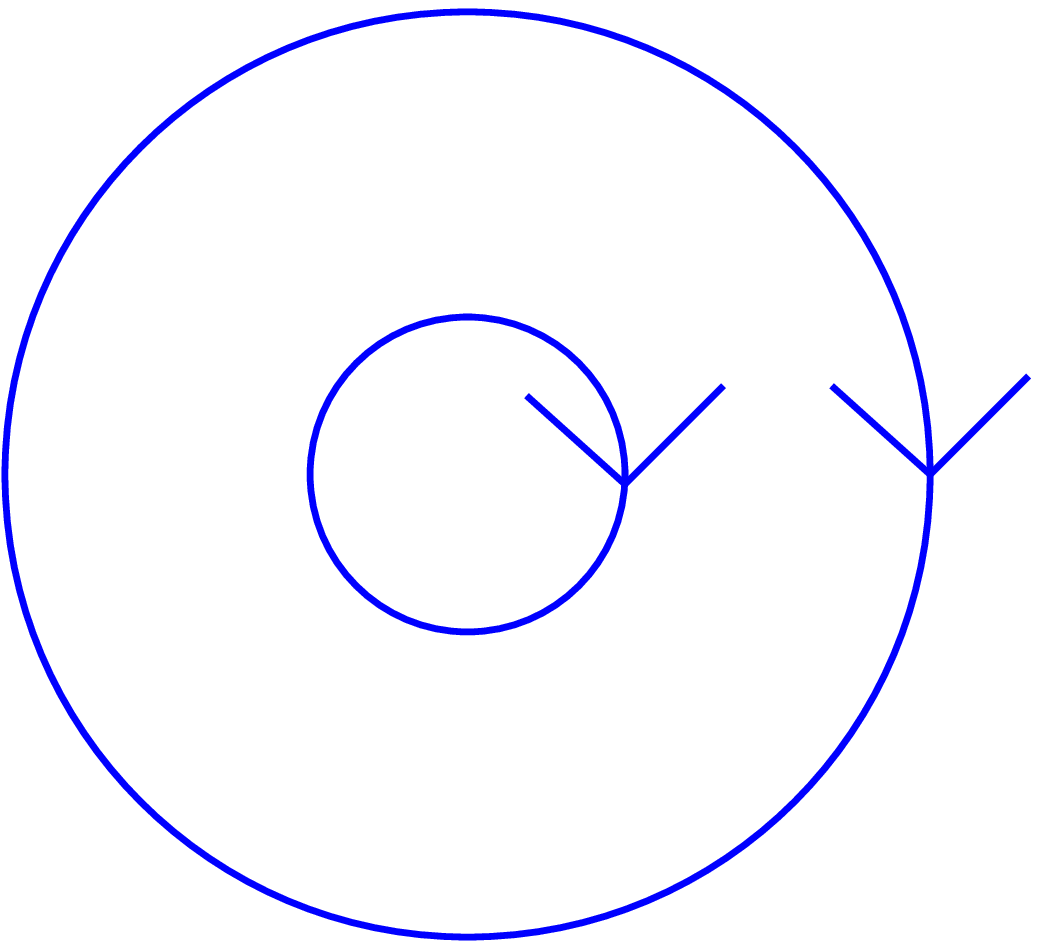}};
(-7,8)*{\scriptstyle\blue n};
(-3,4)*{\scriptstyle\blue n+1};
(13,0)*{(\lambda,0)};
(1,-2)*{\blue\bullet};
(6,-2)*{\blue\bullet};
(1.5,-4)*{\scriptstyle\blue \spadesuit + a};
(8,-4)*{\scriptstyle\blue \spadesuit};
\endxy
\end{gather*}
for any $a\in\mathbb{N}$ and $\lambda\in\Lambda(n,n)$. 
The notation $\spadesuit$ is defined by 
$$
\xy
(0,0)*{\figs{0.15}{Circlecc}};
(-5,5.5)*{\scriptstyle\blue i};
(7.5,0)*{\lambda};
(2.5,-3)*{\blue \bullet};
(3,-6)*{\scriptstyle\blue \spadesuit + b};
\endxy
\;:=\;
\xy
(0,0)*{\figs{0.15}{Circlecc}};
(-5,5.5)*{\scriptstyle\blue i};
(7.5,0)*{\lambda};
(2.5,-3)*{\blue \bullet};
(3,-6)*{\scriptstyle\blue -(\lambda_i-\lambda_{i+1})-1 + b};
\endxy
\qquad
\xy
(0,0)*{\figs{0.15}{Circlec}};
(-5,5.5)*{\scriptstyle\blue i};
(7.5,0)*{\lambda};
(2.5,-3)*{\blue \bullet};
(3,-6)*{\scriptstyle\blue \spadesuit + b};
\endxy
\;:=\;
\xy
(0,0)*{\figs{0.15}{Circlec}};
(-5,5.5)*{\scriptstyle\blue i};
(7.5,0)*{\lambda};
(2.5,-3)*{\blue \bullet};
(3,-6)*{\scriptstyle\blue \lambda_i-\lambda_{i+1}-1 + b};
\endxy,
$$     
for any $b\in\mathbb{N}$. 

For $\spadesuit+a<0$, the image of the fake $n$-bubbles above is a 
definition. For $\spadesuit +a\geq 0$, we have to prove that the 
image of the $n$-bubbles above is equal to the image assigned to them 
by $\mathcal{I}_n$. This is immediate if the two nested bubbles 
in the image are real (since the numbers of dots match), but one of them 
could be fake, in which case a proof is required. 
Let us give this proof for the counter-clockwise $n$-bubbles. Note that
\begin{equation}
\label{eq:bubbleimagesequality}
\xy
(0,0)*{\figs{0.15}{Circlecc}};
(-5,5.5)*{\scriptstyle\blue n};
(7.5,0)*{\lambda};
(2.5,-3)*{\blue \bullet};
(3,-6)*{\scriptstyle\blue \spadesuit + a};
\endxy
\;=\;
\xy
(0,0)*{\figs{0.15}{Circlecc}};
(-5,5.5)*{\scriptstyle\blue n};
(7.5,0)*{\lambda};
(2.5,-3)*{\blue \bullet};
(3,-6)*{\scriptstyle\blue -(\lambda_n-\lambda_1)-1 + a};
\endxy.
\end{equation}
By the definition above, the image 
of the l.h.s. of~\eqref{eq:bubbleimagesequality} is given 
by 
$$
\xy
(0,0)*{\figs{0.15}{Doublecirclecc}};
(-7,8)*{\scriptstyle\blue n+1};
(-3,4)*{\scriptstyle\blue n};
(13,0)*{(\lambda,0)};
(1,-2)*{\blue\bullet};
(6,-2)*{\blue\bullet};
(1.5,-4)*{\scriptstyle\blue \spadesuit + a};
(8,-4)*{\scriptstyle\blue \spadesuit};
\endxy
\;=\;
-\sum_{b+c=a} 
\xy
(0,0)*{\figs{0.15}{Circlecc}};
(-5,5.5)*{\scriptstyle\blue n+1};
(2.5,-3)*{\blue \bullet};
(3,-6)*{\scriptstyle\blue \spadesuit + b};
\endxy
\xy
(0,0)*{\figs{0.15}{Circlecc}};
(-5,5.5)*{\scriptstyle\blue n};
(11,0)*{(\lambda,0)};
(2.5,-3)*{\blue \bullet};
(3,-6)*{\scriptstyle\blue \spadesuit + c};
\endxy.
$$
The equality is obtained by applying a bubble-slide relation. 
By the definition of $\mathcal{I}_n$, the image of the r.h.s. 
of~\eqref{eq:bubbleimagesequality} is given by 
\begin{eqnarray*}
\xy
(0,0)*{\figs{0.15}{Doublecirclecc}};
(-7,8)*{\scriptstyle\blue n+1};
(-3,4)*{\scriptstyle\blue n};
(13,0)*{(\lambda,0)};
(1,-2)*{\blue\bullet};
(1.5,-4)*{\scriptstyle\blue a'};
\endxy
&=&
-\sum_{b'+c=a'+\lambda_n} 
\xy
(0,0)*{\figs{0.15}{Circlecc}};
(-5,5.5)*{\scriptstyle\blue n+1};
(2.5,-3)*{\blue \bullet};
(3,-6)*{\scriptstyle\blue b'};
\endxy
\xy
(0,0)*{\figs{0.15}{Circlecc}};
(-5,5.5)*{\scriptstyle\blue n};
(11,0)*{(\lambda,0)};
(2.5,-3)*{\blue \bullet};
(3,-6)*{\scriptstyle\blue \spadesuit + c};
\endxy
\\
&=&
-\sum_{b'+c=a'+\lambda_n} 
\xy
(0,0)*{\figs{0.15}{Circlecc}};
(-5,5.5)*{\scriptstyle\blue n+1};
(2.5,-3)*{\blue \bullet};
(3,-6)*{\scriptstyle\blue \spadesuit + b'-\lambda_1+1};
\endxy
\xy
(0,0)*{\figs{0.15}{Circlecc}};
(-5,5.5)*{\scriptstyle\blue n};
(11,0)*{(\lambda,0)};
(2.5,-3)*{\blue \bullet};
(3,-6)*{\scriptstyle\blue \spadesuit + c};
\endxy\\
&=&
-\sum_{b+c=a} 
\xy
(0,0)*{\figs{0.15}{Circlecc}};
(-5,5.5)*{\scriptstyle\blue n+1};
(2.5,-3)*{\blue \bullet};
(3,-6)*{\scriptstyle\blue \spadesuit + b};
\endxy
\xy
(0,0)*{\figs{0.15}{Circlecc}};
(-5,5.5)*{\scriptstyle\blue n};
(11,0)*{(\lambda,0)};
(2.5,-3)*{\blue \bullet};
(3,-6)*{\scriptstyle\blue \spadesuit + c};
\endxy
\end{eqnarray*}
with $a'=-(\lambda_n-\lambda_1)-1+a$. The first equality is obtained by 
applying a bubble-slide relation, the other equalities are obtained 
by reindexing. This finishes the proof that both definitions of the 
image of the counter-clockwise non-fake $n$-bubbles are equal. 
The proof for the clockwise $n$-bubbles is similar and is left to the reader.  

We now show that with the definitions above, the images of the bubbles satisfy 
the infinite Grassmannian relation. To be more precise we have to show that 
the relation 
\begin{equation}
\label{eq:IGpreserved}
\sum_{a=0}^b 
\xy
(0,0)*{\figs{0.15}{Circlecc}};
(-5,5.5)*{\scriptstyle\blue n};
(2.5,-3)*{\blue \bullet};
(3,-6)*{\scriptstyle\blue \spadesuit + b-a};
\endxy
\xy
(0,0)*{\figs{0.15}{Circlec}};
(-5,5.5)*{\scriptstyle\blue n};
(2.5,-3)*{\blue \bullet};
(11,0)*{\lambda};
(3,-6)*{\scriptstyle\blue \spadesuit + a};
\endxy
\;=\;
-\delta_{b,0}
\end{equation}
is preserved, for any $b\in\mathbb{N}$. For $b=0$, the image 
of~\eqref{eq:IGpreserved} is given by 
$$
\xy
(0,0)*{\figs{0.15}{Doublecirclecc}};
(-7,8)*{\scriptstyle\blue n+1};
(-3,4)*{\scriptstyle\blue n};
(1,-2)*{\blue\bullet};
(6,-2)*{\blue\bullet};
(1.5,-4)*{\scriptstyle\blue \spadesuit};
(8,-4)*{\scriptstyle\blue \spadesuit};
\endxy
\;
\xy
(0,0)*{\figs{0.15}{Doublecirclec}};
(-7,8)*{\scriptstyle\blue n};
(-3,4)*{\scriptstyle\blue n+1};
(13,0)*{(\lambda,0)};
(1,-2)*{\blue\bullet};
(6,-2)*{\blue\bullet};
(1.5,-4)*{\scriptstyle\blue \spadesuit};
(8,-4)*{\scriptstyle\blue \spadesuit};
\endxy
\;=\;
-1.
$$
The equality follows immediately from the degree-zero bubble relations. 
For $b>0$, the image of~\eqref{eq:IGpreserved} is given by 
\begin{eqnarray*}
\sum_{a=0}^b
\xy
(0,0)*{\figs{0.15}{Doublecirclecc}};
(-7,8)*{\scriptstyle\blue n+1};
(-3,4)*{\scriptstyle\blue n};
(1,-2)*{\blue\bullet};
(6,-2)*{\blue\bullet};
(1.5,-4)*{\scriptstyle\blue \spadesuit+b-a};
(8,-4)*{\scriptstyle\blue \spadesuit};
\endxy
\;
\xy
(0,0)*{\figs{0.15}{Doublecirclec}};
(-7,8)*{\scriptstyle\blue n};
(-3,4)*{\scriptstyle\blue n+1};
(13,0)*{(\lambda,0)};
(1,-2)*{\blue\bullet};
(6,-2)*{\blue\bullet};
(1.5,-4)*{\scriptstyle\blue \spadesuit + a};
(8,-4)*{\scriptstyle\blue \spadesuit};
\endxy
&=&
\sum_{a=0}^b \sum_{k=0}^a
\xy
(0,0)*{\figs{0.15}{Doublecirclecc}};
(-7,8)*{\scriptstyle\blue n+1};
(-3,4)*{\scriptstyle\blue n};
(1,-2)*{\blue\bullet};
(6,-2)*{\blue\bullet};
(1.5,-4)*{\scriptstyle\blue \spadesuit+b-a};
(8,-4)*{\scriptstyle\blue \spadesuit};
\endxy 
\xy
(0,0)*{\figs{0.15}{Circlec}};
(-5,5.5)*{\scriptstyle\blue n+1};
(2.5,-3)*{\blue \bullet};
(3,-6)*{\scriptstyle\blue \spadesuit + k};
\endxy
\xy
(0,0)*{\figs{0.15}{Circlec}};
(-5,5.5)*{\scriptstyle\blue n};
(11,0)*{(\lambda,0)};
(2.5,-3)*{\blue \bullet};
(3,-6)*{\scriptstyle\blue \spadesuit + a-k};
\endxy
\\ \\
&=&
-\sum_{a=0}^b\sum_{k=0}^a\sum_{\ell=0}^{a-k}
\xy
(0,0)*{\figs{0.15}{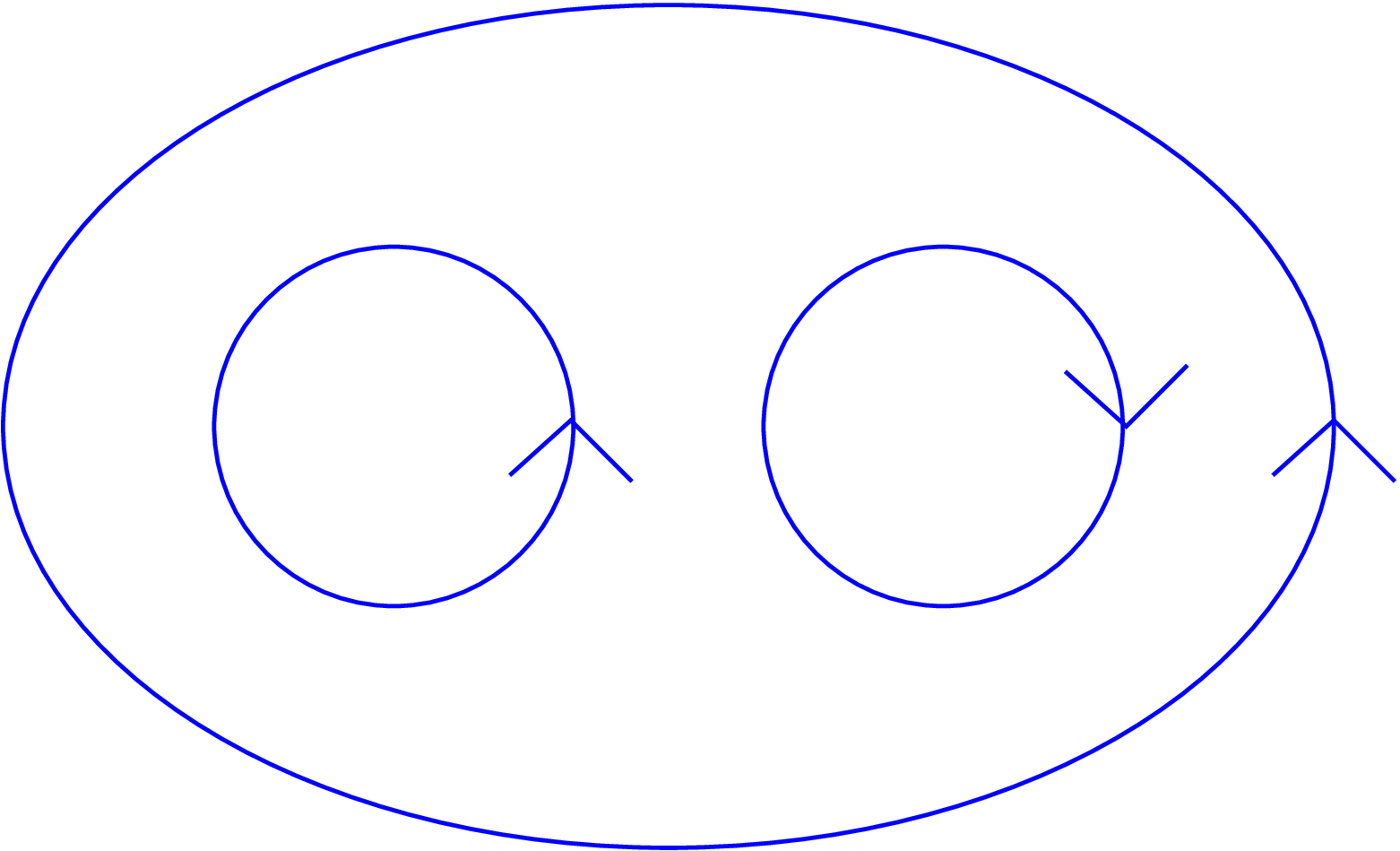}};
(-11.5,5)*{\scriptstyle\blue n};
(2,5)*{\scriptstyle\blue n};
(-16,9)*{\scriptstyle\blue n+1};
(-4,-3)*{\blue \bullet};
(9.5,-3)*{\blue \bullet};
(10,-8)*{\blue \bullet};
(-2,-5)*{\scriptstyle\blue \spadesuit + b-a};
(11.5,-5)*{\scriptstyle\blue \spadesuit + a-k-\ell};
(12.5,-10)*{\scriptstyle\blue \spadesuit + \ell};
\endxy
\xy
(0,0)*{\figs{0.15}{Circlec}};
(-5,5.5)*{\scriptstyle\blue n+1};
(11,0)*{(\lambda,0)};
(2.5,-3)*{\blue \bullet};
(3,-6)*{\scriptstyle\blue \spadesuit + k};
\endxy
\\ \\
&=&
-\sum_{a=0}^b\sum_{c=0}^a\sum_{k=0}^{c}
\xy
(0,0)*{\figs{0.15}{CCE}};
(-11.5,5)*{\scriptstyle\blue n};
(2,5)*{\scriptstyle\blue n};
(-16,9)*{\scriptstyle\blue n+1};
(-4,-3)*{\blue \bullet};
(9.5,-3)*{\blue \bullet};
(10,-8)*{\blue \bullet};
(-2,-5)*{\scriptstyle\blue \spadesuit + b-a};
(11.5,-5)*{\scriptstyle\blue \spadesuit + a-c};
(12.5,-10)*{\scriptstyle\blue \spadesuit + c-k};
\endxy
\xy
(0,0)*{\figs{0.15}{Circlec}};
(-5,5.5)*{\scriptstyle\blue n+1};
(11,0)*{(\lambda,0)};
(2.5,-3)*{\blue \bullet};
(3,-6)*{\scriptstyle\blue \spadesuit + k};
\endxy
\\ \\
&=&
-\sum_{c=0}^b\sum_{k=0}^c\sum_{m=0}^{b-c}
\xy
(0,0)*{\figs{0.15}{CCE}};
(-11.5,5)*{\scriptstyle\blue n};
(2,5)*{\scriptstyle\blue n};
(-16,9)*{\scriptstyle\blue n+1};
(-4,-3)*{\blue \bullet};
(9.5,-3)*{\blue \bullet};
(10,-8)*{\blue \bullet};
(-2,-5)*{\scriptstyle\blue \spadesuit + b-c-m};
(11.5,-5)*{\scriptstyle\blue \spadesuit + m};
(12.5,-10)*{\scriptstyle\blue \spadesuit + c-k};
\endxy
\xy
(0,0)*{\figs{0.15}{Circlec}};
(-5,5.5)*{\scriptstyle\blue n+1};
(11,0)*{(\lambda,0)};
(2.5,-3)*{\blue \bullet};
(3,-6)*{\scriptstyle\blue \spadesuit + k};
\endxy
\\ \\
&=&
0.
\end{eqnarray*}
The first two equalities follow from bubble-slide relations. The next two 
equalities follow from reindexing, as indicated. The last equality follows 
from the infinite Grassmannian relation: for the $n$-bubbles if $b>c$ (with 
$c$ fixed), and for the $n+1$-bubbles if $b=c$.

Knowing the images of the fake bubbles allows us to prove the other 
relations involving only $n$-strands very easily. Let us do just one example, 
the other relations can be proved in a similar fashion. We show that 
$\mathcal{I}_n$ preserves the relation
\begin{equation}
\label{eq:RtoBpreserved}
\xy
(0,0)*{\figs{0.15}{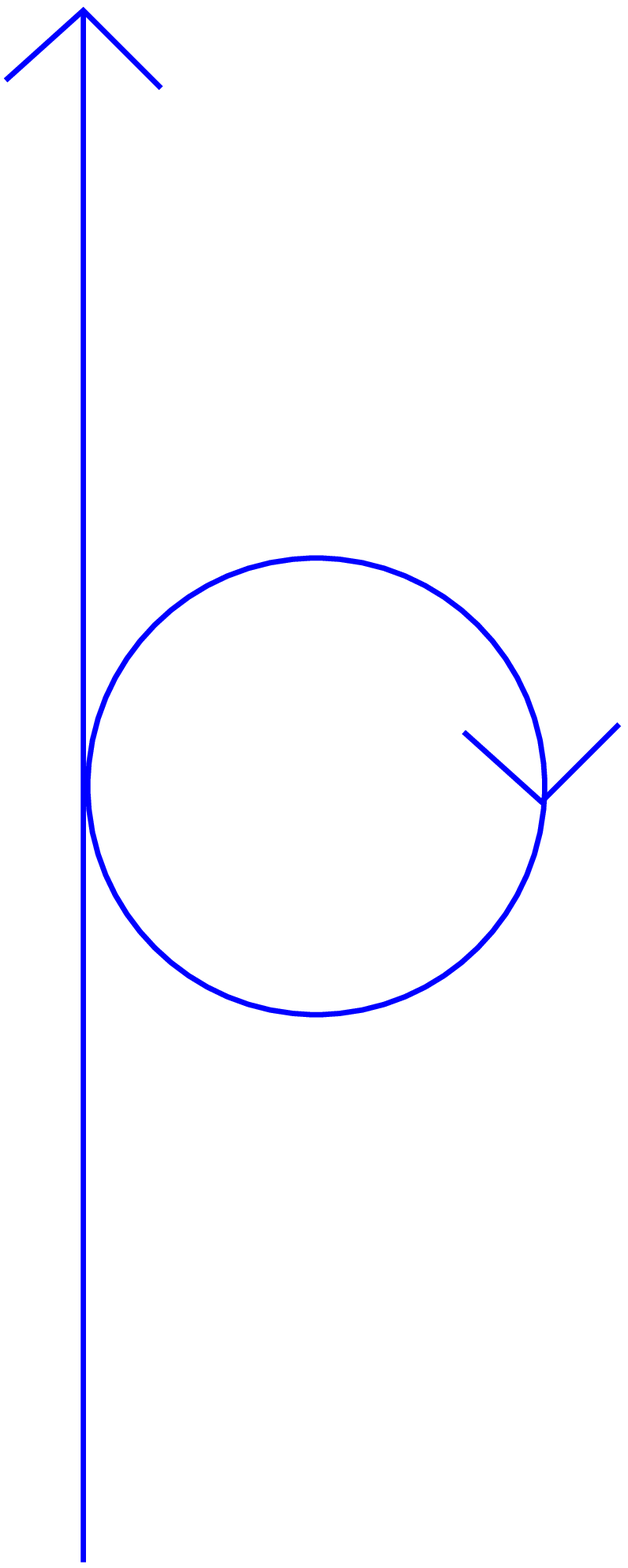}};
(-4.5,-17)*{\scriptstyle\blue n};
(9,0)*{\lambda};
\endxy
\;=\;
-\sum_{f=0}^{\lambda_1-\lambda_n}
\xy
(0,0)*{\figs{0.15}{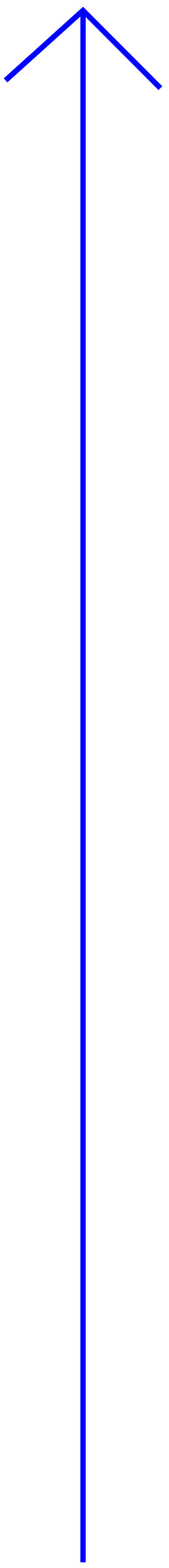}};
(0,0)*{\blue\bullet};
(8,0)*{\scriptstyle\blue \lambda_1-\lambda_n-f};
(0,-17)*{\scriptstyle\blue n};
\endxy
\xy
(0,0)*{\figs{0.15}{Circlec}};
(-5,5.5)*{\scriptstyle\blue n};
(9,0)*{\lambda};
(2.5,-3)*{\blue \bullet};
(3,-6)*{\scriptstyle\blue \spadesuit + f};
\endxy.
\end{equation} 
The image of the l.h.s. of~\eqref{eq:RtoBpreserved} is given by 
$$
\xy
(-0.2,0)*{\figs{0.15}{IdentityE}};
(4.6,0)*{\figs{0.15}{IdentityE}};
(7.7,0)*{\figs{0.15}{Doublecirclec}};
(-0.2,-17)*{\scriptstyle\blue n};
(4.6,-17)*{\scriptstyle\blue n+1};
(21,0)*{(\lambda,0)};
\endxy,
$$
which is equal to 
$$
-\sum_{f=0}^{\lambda_1-1}
\xy
(0,0)*{\figs{0.15}{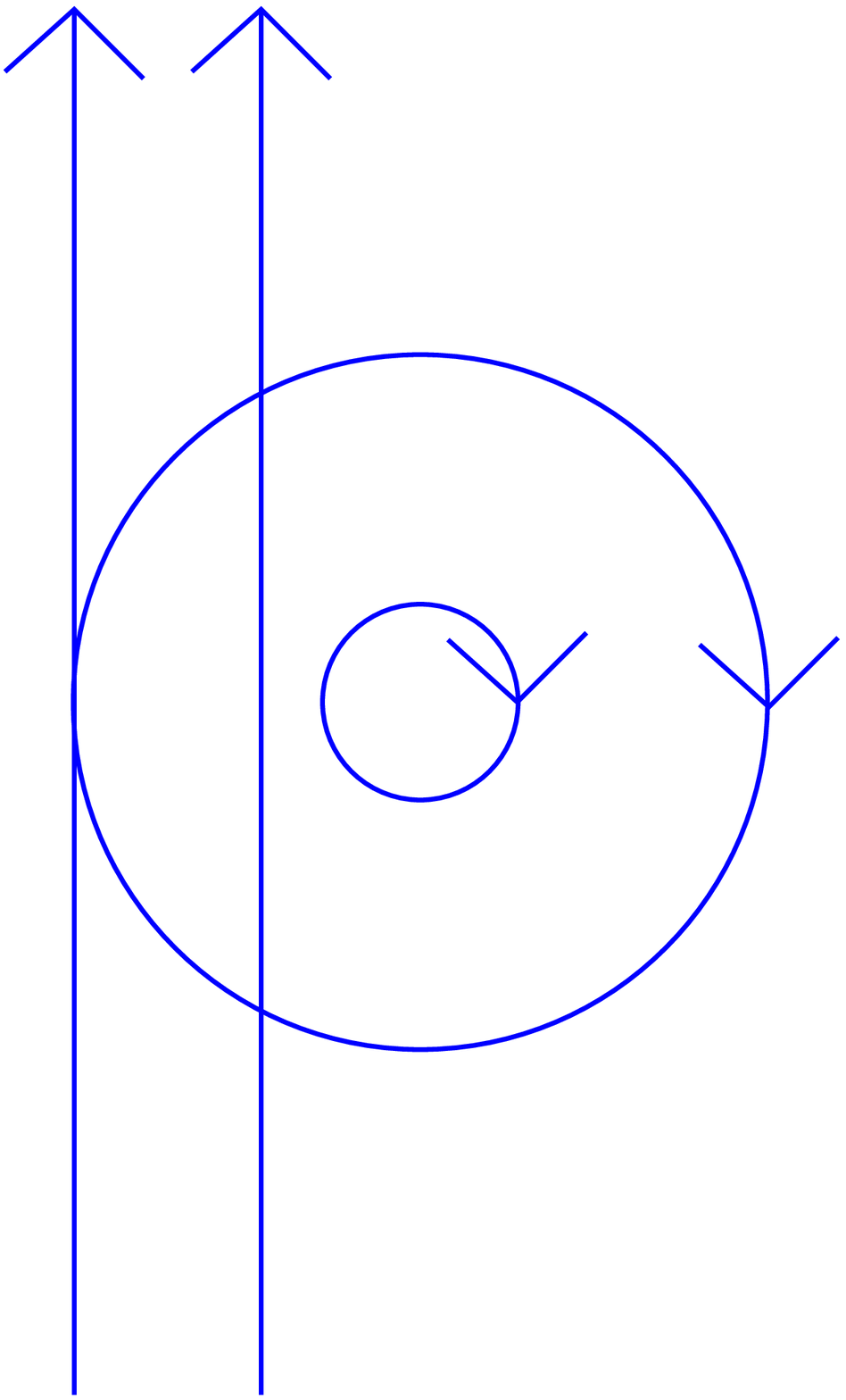}};
(2,-1)*{\blue\bullet};
(-3.5,-3)*{\blue\bullet};
(-7,-17)*{\scriptstyle\blue n};
(-2.5,-17)*{\scriptstyle\blue n+1};
(-2.5,3.5)*{\scriptstyle\blue n+1};
(-3,-5.5)*{\scriptstyle\blue \lambda_1-1 - f};
(6,-2.5)*{\scriptstyle\blue \spadesuit + f};
(16,0)*{(\lambda,0)};
\endxy
\;=\;
-\sum_{f=0}^{\lambda_1-1}
\xy
(0,0)*{\figs{0.15}{IdentityE}};
(5,0)*{\figs{0.15}{IdentityE}};
(15,0)*{\figs{0.15}{Doublecirclec}};
(5,-3)*{\blue\bullet};
(15,-2)*{\blue\bullet};
(0,-17)*{\scriptstyle\blue n};
(5,-17)*{\scriptstyle\blue n+1};
(8,7)*{\scriptstyle\blue n};
(12,3.5)*{\scriptstyle\blue n+1};
(7,-5)*{\scriptstyle\blue \lambda_1-1 - f};
(19,-3.5)*{\scriptstyle\blue \spadesuit + f};
(28,0)*{(\lambda,0)};
\endxy
\;=\;
-\sum_{f=0}^{\lambda_1-\lambda_n}
\xy
(0,0)*{\figs{0.15}{IdentityE}};
(5,0)*{\figs{0.15}{IdentityE}};
(15,0)*{\figs{0.15}{Doublecirclec}};
(5,-3)*{\blue\bullet};
(15,-2)*{\blue\bullet};
(18,-6)*{\blue\bullet};
(0,-17)*{\scriptstyle\blue n};
(5,-17)*{\scriptstyle\blue n+1};
(8,7)*{\scriptstyle\blue n};
(12,3.5)*{\scriptstyle\blue n+1};
(7,-5)*{\scriptstyle\blue \lambda_1-\lambda_n - f};
(19,-3.5)*{\scriptstyle\blue \spadesuit + f};
(20,-7)*{\scriptstyle\blue \spadesuit};
(28,0)*{(\lambda,0)};
\endxy.
$$
The first summation is obtained by resolving the $n+1$-curl. The second 
summation can then be obtained by applying a Reidemeister 3 relation to 
the strands colored $n$, $n+1$ and $n$. Note that only the terms 
which are shown survive, the other ones are zero 
because they are given by diagrams which contain a region 
whose label has a negative entry. The last summation is obtained by 
first reindexing. Then an argument similar to the one we 
used below~\eqref{eq:bubbleimagesequality} 
ensures that the nested bubbles, before and after the equality, match 
and that the first $\lambda_n-1$ terms of the reindexed summation 
vanish (indeed in those terms, bubbles of negative degree appear, and those 
are always zero). This last expression is equal to the image of the 
r.h.s. of~\eqref{eq:RtoBpreserved}, which 
finishes our proof that $\mathcal{I}_n$ preserves~\eqref{eq:RtoBpreserved}.

Next let us have a look at the relations involving $i$-strands of more than 
one color. We just do one example in detail, the other 
relations can be proved in a similar fashion. Consider the relation
\begin{equation}
\label{eq:R2}
\xy
(0,0)*{\figs{0.15}{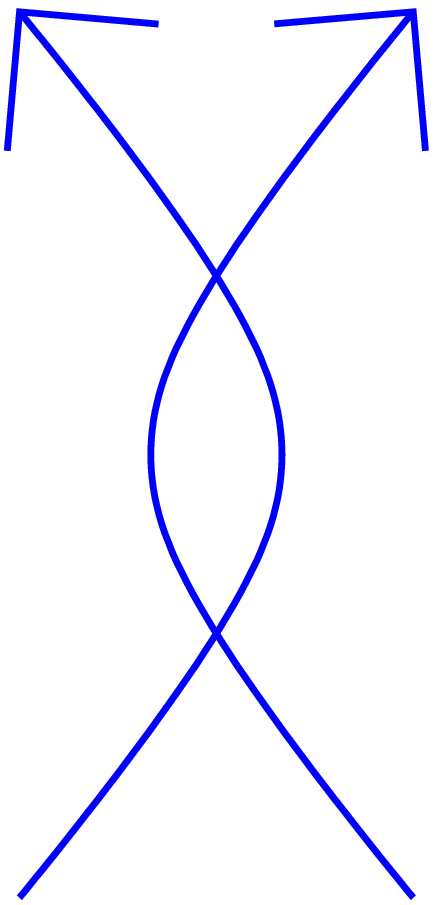}};
(-3,-9)*{\scriptstyle\blue n};
(3,-9)*{\scriptstyle\blue 1};
(5,0)*{\lambda};
\endxy
\;=\;
-
\xy
(0,0)*{\figs{0.15}{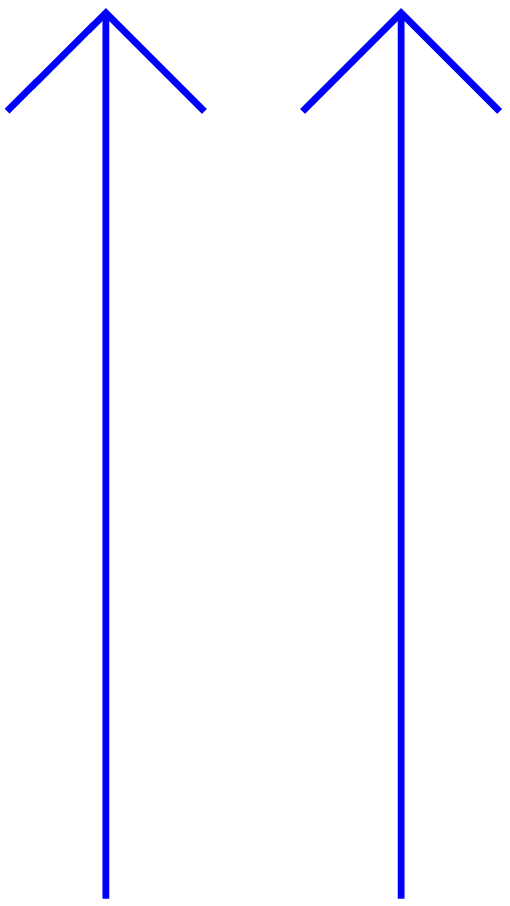}};
(-2.2,0)*{\blue\bullet};
(-2,-9)*{\scriptstyle\blue n};
(2,-9)*{\scriptstyle\blue 1};
(5,0)*{\lambda};
\endxy
\;+\;
\xy
(0,0)*{\figs{0.15}{IdentityR22}};
(2.4,0)*{\blue\bullet};
(-2,-9)*{\scriptstyle\blue n};
(2,-9)*{\scriptstyle\blue 1};
(5,0)*{\lambda};
\endxy
\end{equation}
in $\Scat(n,n)$. The image of the term on the l.h.s. is given by 
$$
\xy
(0,0)*{\figs{0.15}{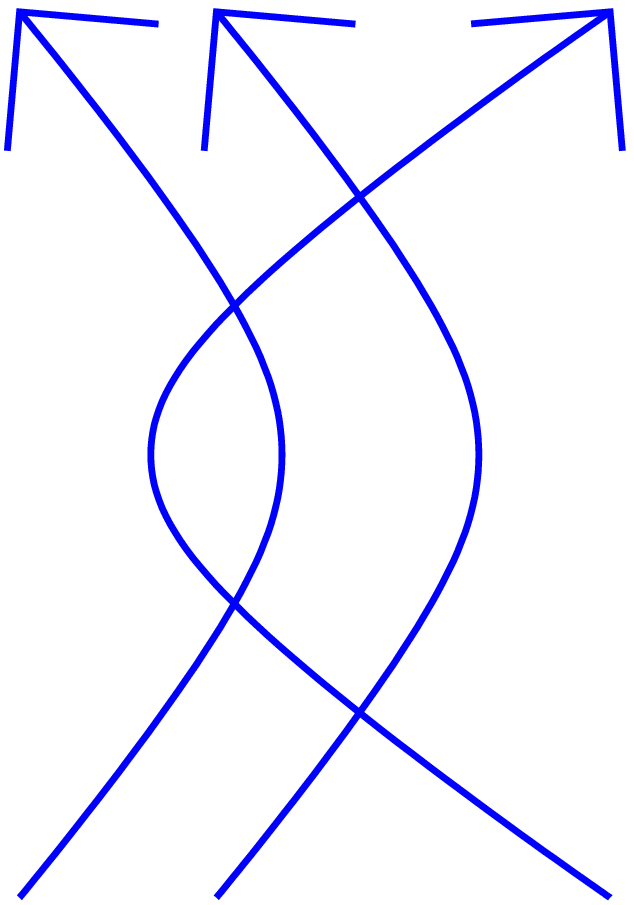}};
(-6,-9)*{\scriptstyle\blue n};
(-1,-9)*{\scriptstyle\blue n+1};
(5,-9)*{\scriptstyle\blue 1};
(8.5,0)*{(\lambda,0)};
\endxy
\;=\;
\xy
(0,0)*{\figs{0.15}{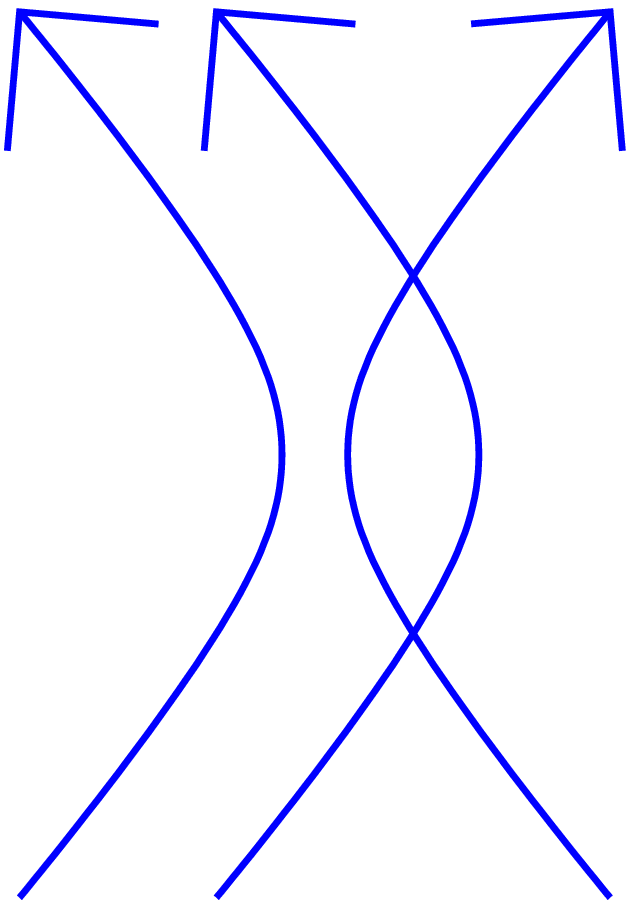}};
(-6,-9)*{\scriptstyle\blue n};
(-1,-9)*{\scriptstyle\blue n+1};
(5,-9)*{\scriptstyle\blue 1};
(8.5,0)*{(\lambda,0)};
\endxy
\;=\;
-
\xy
(0,0)*{\figs{0.15}{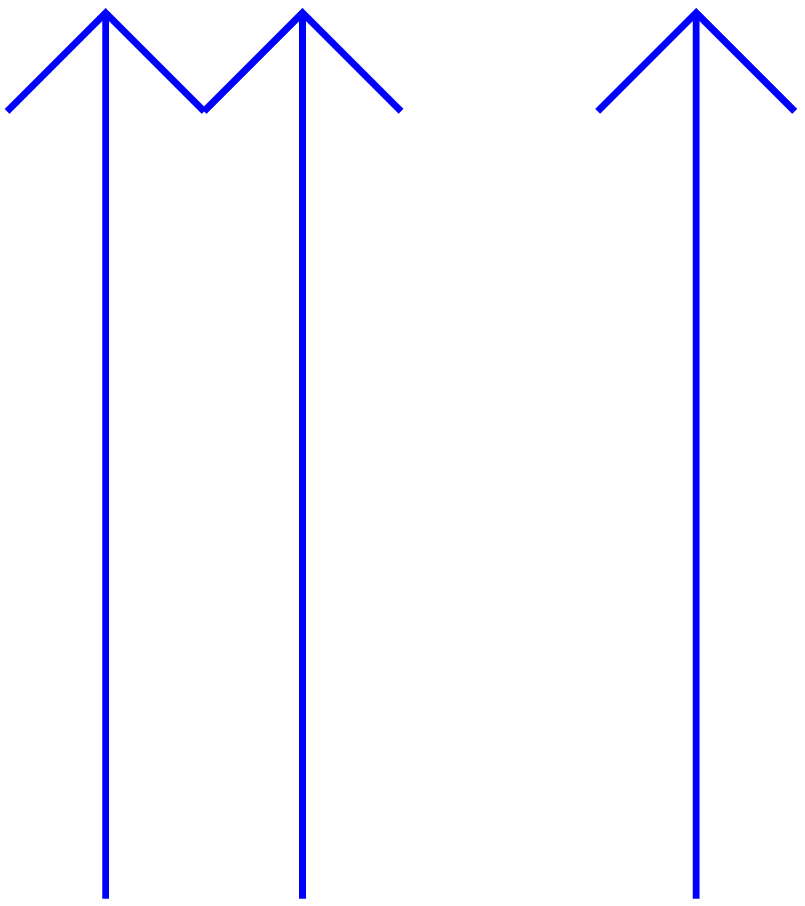}};
(-1.6,0)*{\blue\bullet};
(-5,-9)*{\scriptstyle\blue n};
(-1,-9)*{\scriptstyle\blue n+1};
(4.8,-9)*{\scriptstyle\blue 1};
(10.5,0)*{(\lambda,0)};
\endxy\;+\;
\xy
(0,0)*{\figs{0.15}{R2imageeq2}};
(4.5,0)*{\blue\bullet};
(-5,-9)*{\scriptstyle\blue n};
(-1,-9)*{\scriptstyle\blue n+1};
(4.8,-9)*{\scriptstyle\blue 1};
(10.5,0)*{(\lambda,0)};
\endxy.
$$
The first and the second equality follow from the Reidemeister 2 
relations in $\Scat(n+1,n)$. The linear combination at the end is 
exactly the image of the r.h.s. in~\eqref{eq:R2}, which proves 
that~\eqref{eq:R2} is preserved by $\mathcal{I}_n$.

Remains to be proved that $\mathcal{I}_n$ preserves 
the relations involving $\delta$-strands. 
For the relations~\eqref{eq_biadjoint1} and~\eqref{eq_biadjoint2} the proof 
follows immediately from the zig-zag relations for $i$-strands with 
$i=1,\ldots,n+1$. For the relations in~\eqref{deltabubbles} the proof 
follows immediately from the degree-zero $i$-bubble relations for 
$i=1,\ldots,n+1$. Let us explain the first relation in~\eqref{deltainvert2} 
in more detail, the second being similar. The image of 
$$
\xy   
    (0,0)*{\black\xybox{
    (0,-6.5);(0,6.5)**\dir{-} ?(.5)*\dir{<};}};
    (0,-9.5)*{\delta};
    (9,0)*{\black\xybox{
    (0,-6.5);(0,6.5)**\dir{-} ?(.5)*\dir{>};}};
    (15,0)*{(1^n)};
    (9,-9.5)*{\delta};
    \endxy
$$
is given by 
$$
\xy
(0,0)*{\figs{0.15}{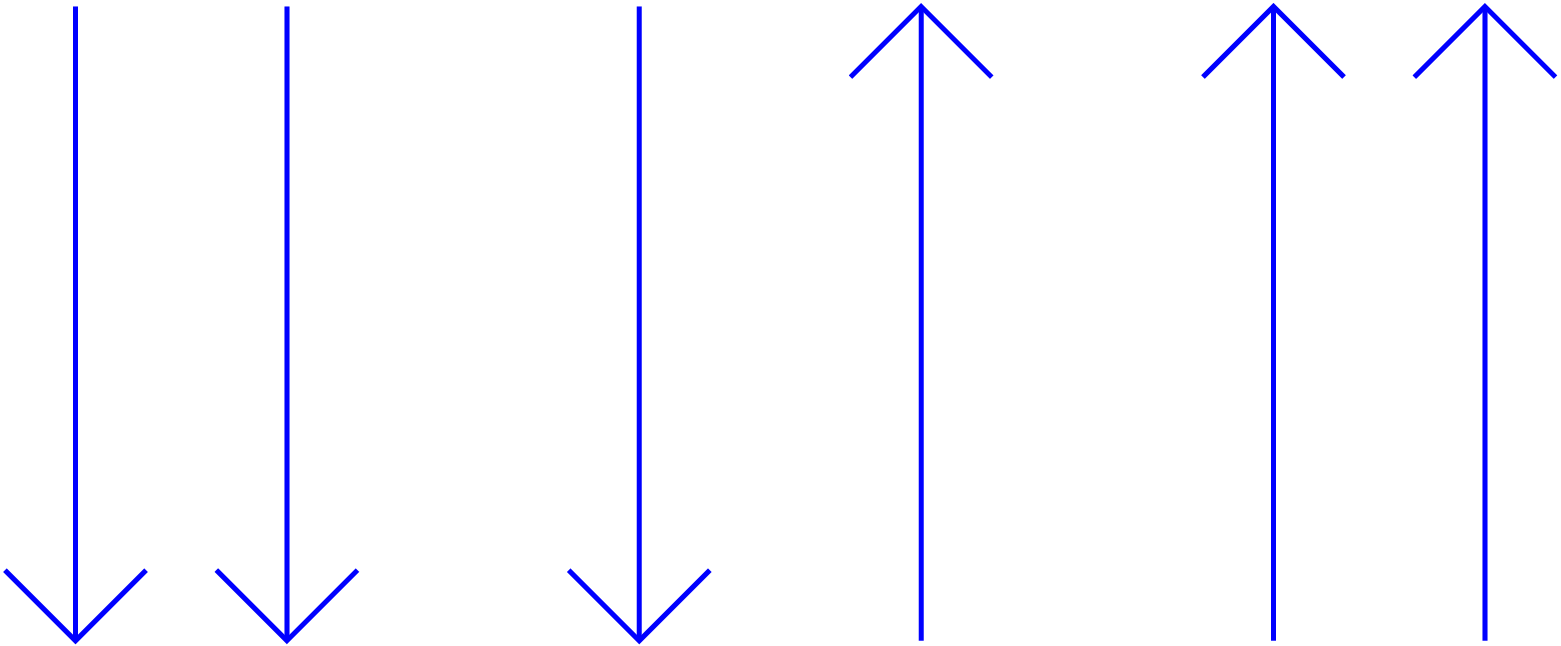}};
(-7,0)*{\blue\cdots};
(7,0)*{\blue\cdots};
(-15,-9)*{\scriptstyle\blue n+1};
(-10.5,-9)*{\scriptstyle\blue 1};
(-3,-9)*{\scriptstyle\blue n};
(3,-9)*{\scriptstyle\blue n};
(10.5,-9)*{\scriptstyle\blue 1};
(15,-9)*{\scriptstyle\blue n+1};
(-15,9)*{\scriptstyle\blue n+1};
(-10.5,9)*{\scriptstyle\blue 1};
(-3,9)*{\scriptstyle\blue n};
(3,9)*{\scriptstyle\blue n};
(10.5,9)*{\scriptstyle\blue 1};
(15,9)*{\scriptstyle\blue n+1};
(23,0)*{(1^n,0)};
\endxy
\;=\;
\xy
(0,0)*{\figs{0.15}{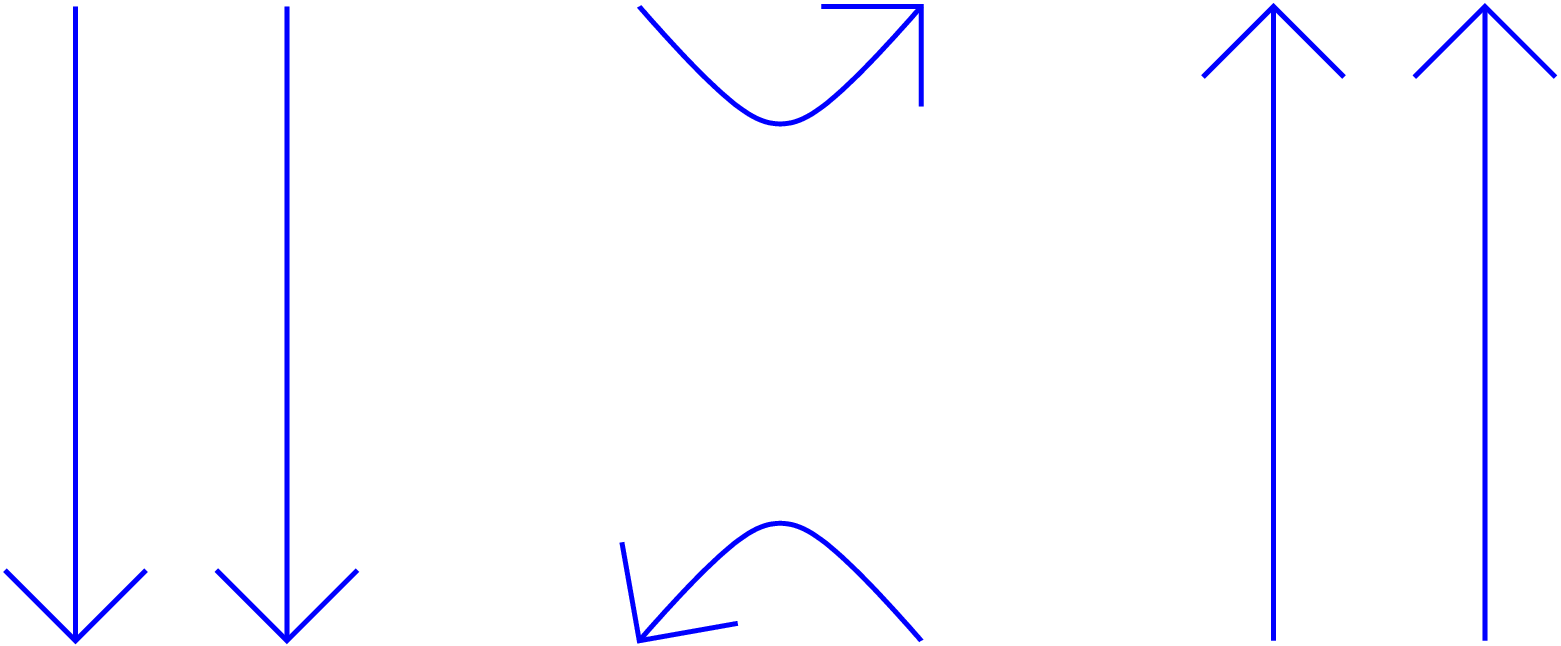}};
(-7,0)*{\blue\cdots};
(7,0)*{\blue\cdots};
(-15,-9)*{\scriptstyle\blue n+1};
(-10.5,-9)*{\scriptstyle\blue 1};
(-3,-9)*{\scriptstyle\blue n};
(3,-9)*{\scriptstyle\blue n};
(10.5,-9)*{\scriptstyle\blue 1};
(15,-9)*{\scriptstyle\blue n+1};
(-15,9)*{\scriptstyle\blue n+1};
(-10.5,9)*{\scriptstyle\blue 1};
(-3,9)*{\scriptstyle\blue n};
(3,9)*{\scriptstyle\blue n};
(10.5,9)*{\scriptstyle\blue 1};
(15,9)*{\scriptstyle\blue n+1};
(23,0)*{(1^n,0)};
\endxy
\;=\cdots=\;
\xy
(0,0)*{\figs{0.15}{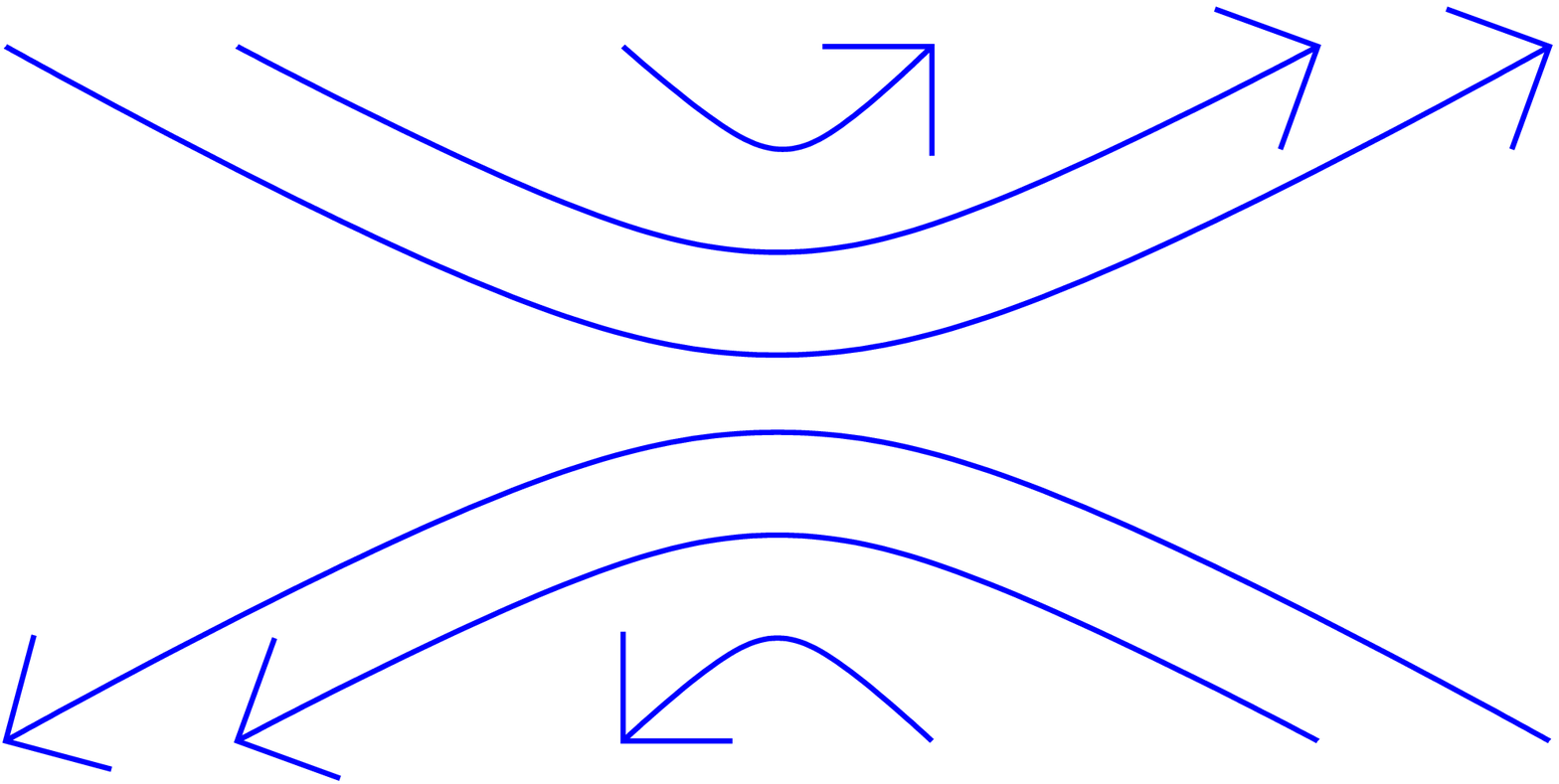}};
(-6,-7)*{\blue\cdots};
(6,-7)*{\blue\cdots};
(-6,7)*{\blue\cdots};
(6,7)*{\blue\cdots};
(-15,-9)*{\scriptstyle\blue n+1};
(-10.5,-9)*{\scriptstyle\blue 1};
(-3,-9)*{\scriptstyle\blue n};
(3,-9)*{\scriptstyle\blue n};
(10.5,-9)*{\scriptstyle\blue 1};
(15,-9)*{\scriptstyle\blue n+1};
(-15,9)*{\scriptstyle\blue n+1};
(-10.5,9)*{\scriptstyle\blue 1};
(-3,9)*{\scriptstyle\blue n};
(3,9)*{\scriptstyle\blue n};
(10.5,9)*{\scriptstyle\blue 1};
(15,9)*{\scriptstyle\blue n+1};
(20,0)*{(1^n,0)};
\endxy,
$$
which is indeed equal to the image of 
$$
\xy
    (0,3)*{\black\bbpef{\black \delta}};
    (8,0)*{(1^n)};
    (-12,0)*{};(12,0)*{};
    (0,-3)*{\black\bbcef{}};
    (-12,0)*{};(12,0)*{}; 
\endxy.
$$
The equalities above are obtained by repeatedly applying 
Reidemeister 2 relations on the pairs of $i$-strands with 
$\lambda_i-\lambda_{i+1}=-1$ 
for all $i=1,\ldots,n+1$. Note that the terms with two $i$-crossings are all 
equal to zero, because they contain a region whose label has one negative 
entry, and that all bubbles in the other terms are of degree zero 
and equal to $-1$. 
 
The fact that relations~\eqref{DeltaEEX},~\eqref{DeltaEERX},~\eqref{LXEEDelta} 
and~\eqref{RXEEDelta} are preserved follows easily 
from applying Reidemeister 2 and 3 relations to the images of the 
terms on their left-hand side. The dots appear after applying the 
Reidemeister 2 relation involving the $i$ and $i+1$-strands. 

We prove the left relation in~\eqref{DeltaFEEDeltaF} for $1\leq i < n$. 
The proof for $i=n$ and the proof of the right relation 
in~\eqref{DeltaFEEDeltaF} are similar and are left to the reader. 
The image on the l.h.s. of the first relation in~\eqref{DeltaFEEDeltaF} 
is given by 
\begin{equation}
\label{eq:DeltaFEEDeltaFimage}
\xy
(0,0)*{\figs{0.22}{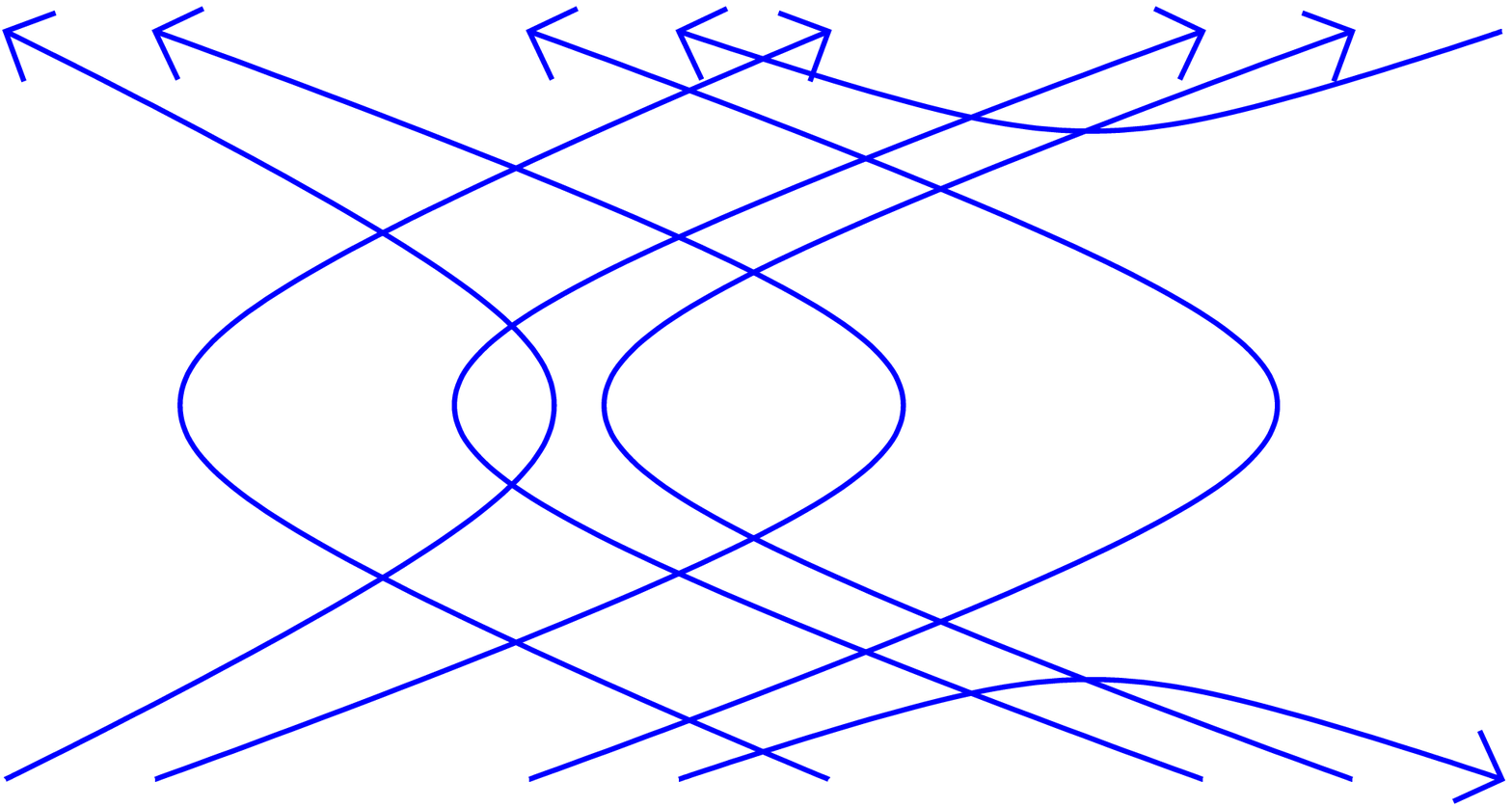}};
(-24,-13.5)*{\scriptstyle\blue n};
(-18,-13.5)*{\scriptstyle\blue n-1};
(-7,-13.5)*{\scriptstyle\blue i+1};
(-3,-13.5)*{\scriptstyle\blue i};
(3,-13.5)*{\scriptstyle\blue i-1};
(13,-13.5)*{\scriptstyle\blue 1};
(18,-13.5)*{\scriptstyle\blue n+1};
(23,-13.5)*{\scriptstyle\blue i};
(-24,13.5)*{\scriptstyle\blue n};
(-18,13.5)*{\scriptstyle\blue n-1};
(-7,13.5)*{\scriptstyle\blue i+1};
(-3,13.5)*{\scriptstyle\blue i};
(3,13.5)*{\scriptstyle\blue i-1};
(13,13.5)*{\scriptstyle\blue 1};
(18,13.5)*{\scriptstyle\blue n+1};
(23,13.5)*{\scriptstyle\blue i};
(-30,0)*{(1^n,0)};
(-11,11)*{\blue\cdots};
(-11,-11)*{\blue\cdots};
(7,11)*{\blue\cdots};
(7,-11)*{\blue\cdots};
\endxy.
\end{equation}
We claim that this is equal to 
$$
\xy
(0,0)*{\figs{0.22}{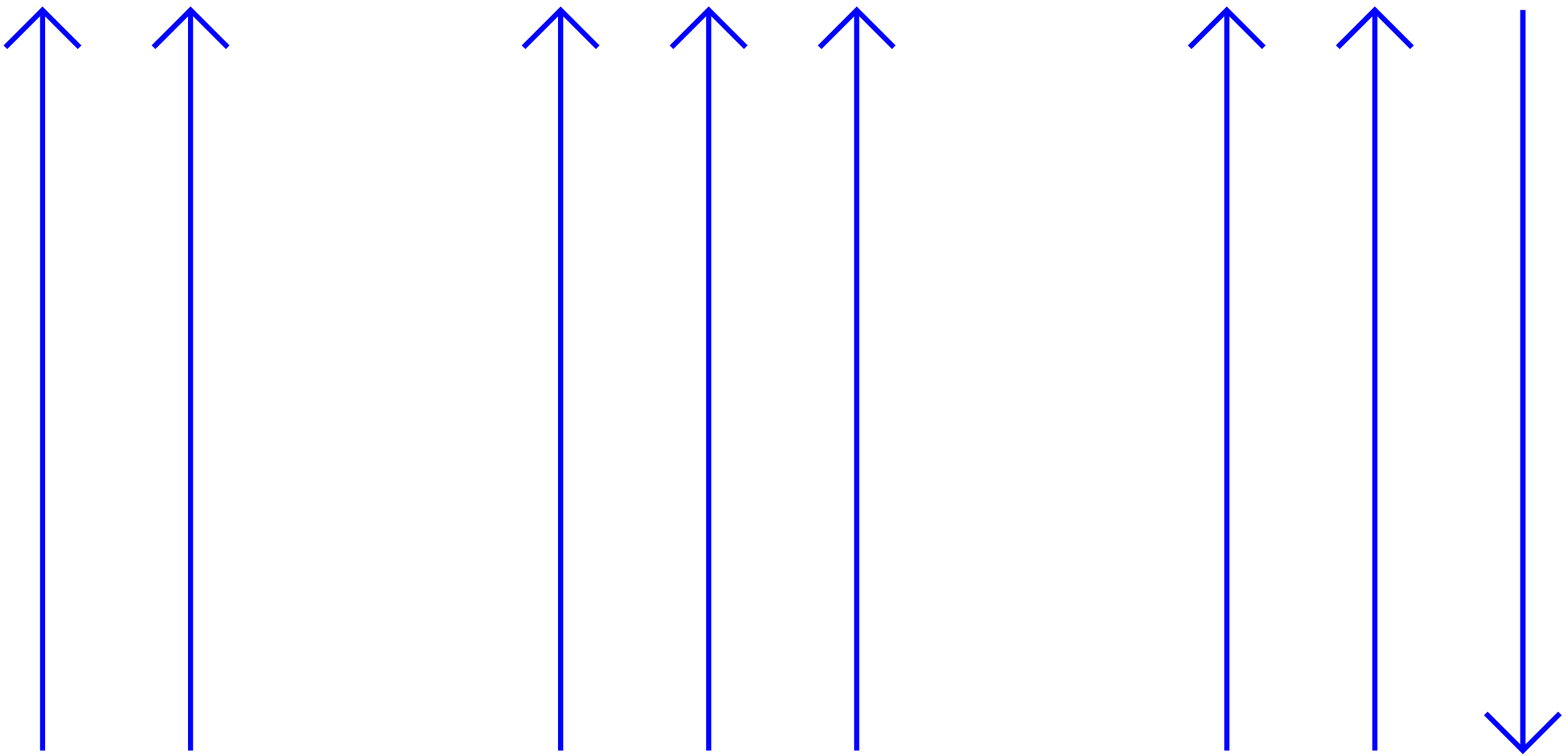}};
(-24,-13.5)*{\scriptstyle\blue n};
(-18,-13.5)*{\scriptstyle\blue n-1};
(-7,-13.5)*{\scriptstyle\blue i+1};
(-2,-13.5)*{\scriptstyle\blue i};
(3,-13.5)*{\scriptstyle\blue i-1};
(13,-13.5)*{\scriptstyle\blue 1};
(18,-13.5)*{\scriptstyle\blue n+1};
(23,-13.5)*{\scriptstyle\blue i};
(-24,13.5)*{\scriptstyle\blue n};
(-18,13.5)*{\scriptstyle\blue n-1};
(-7,13.5)*{\scriptstyle\blue i+1};
(-2,13.5)*{\scriptstyle\blue i};
(3,13.5)*{\scriptstyle\blue i-1};
(13,13.5)*{\scriptstyle\blue 1};
(18,13.5)*{\scriptstyle\blue n+1};
(23,13.5)*{\scriptstyle\blue i};
(-30,0)*{(1^n,0)};
(-12,10)*{\blue\cdots};
(-12,-10)*{\blue\cdots};
(8,10)*{\blue\cdots};
(8,-10)*{\blue\cdots};
\endxy,
$$
which is indeed the image of the r.h.s. of~\eqref{DeltaFEEDeltaF}. 
This follows from first applying Reidemeister 2 relations 
to~\eqref{eq:DeltaFEEDeltaFimage} in order to straighten all $j$-strands 
for $j\ne i$:  
$$
\xy
(0,0)*{\figs{0.22}{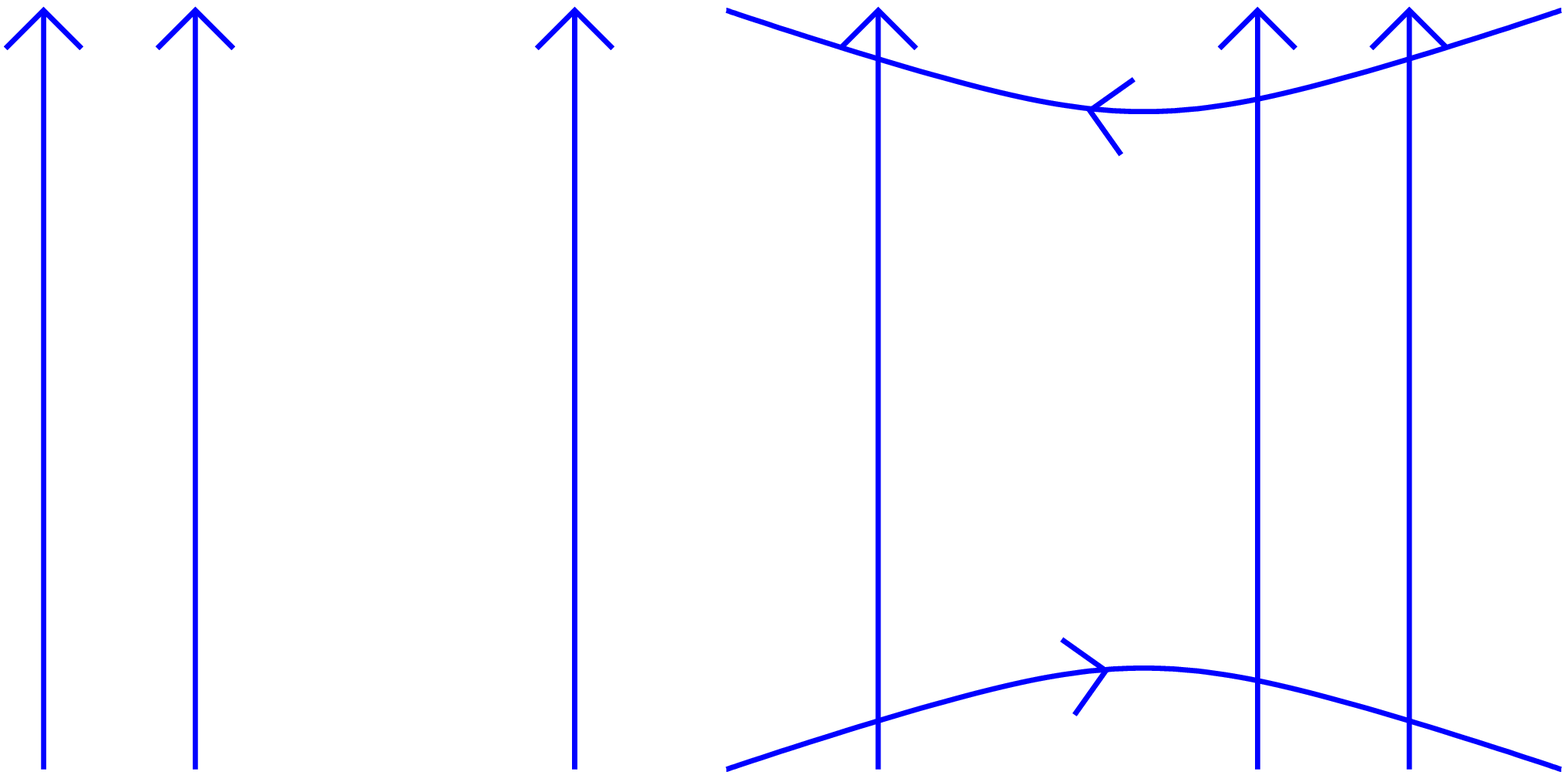}};
(-24,-13.5)*{\scriptstyle\blue n};
(-18,-13.5)*{\scriptstyle\blue n-1};
(-7,-13.5)*{\scriptstyle\blue i+1};
(-2,-13.5)*{\scriptstyle\blue i};
(3,-13.5)*{\scriptstyle\blue i-1};
(13,-13.5)*{\scriptstyle\blue 1};
(18,-13.5)*{\scriptstyle\blue n+1};
(23,-13.5)*{\scriptstyle\blue i};
(-24,13.5)*{\scriptstyle\blue n};
(-18,13.5)*{\scriptstyle\blue n-1};
(-7,13.5)*{\scriptstyle\blue i+1};
(-2,13.5)*{\scriptstyle\blue i};
(3,13.5)*{\scriptstyle\blue i-1};
(13,13.5)*{\scriptstyle\blue 1};
(18,13.5)*{\scriptstyle\blue n+1};
(23,13.5)*{\scriptstyle\blue i};
(-30,0)*{(1^n,0)};
(-12,10)*{\blue\cdots};
(-12,-10)*{\blue\cdots};
(8,10)*{\blue\cdots};
(8,-10)*{\blue\cdots};
\endxy;
$$
then a Reidemeister 2 relation to the 
$i$-strands in the middle 
(note that the region at the top and the bottom between 
the $i$ and the $i-1$-strand is labeled 
$(1,\ldots,1,0,1,\ldots,1)$ with $0$ on the $i$-th position):
$$
\xy
(0,0)*{\figs{0.22}{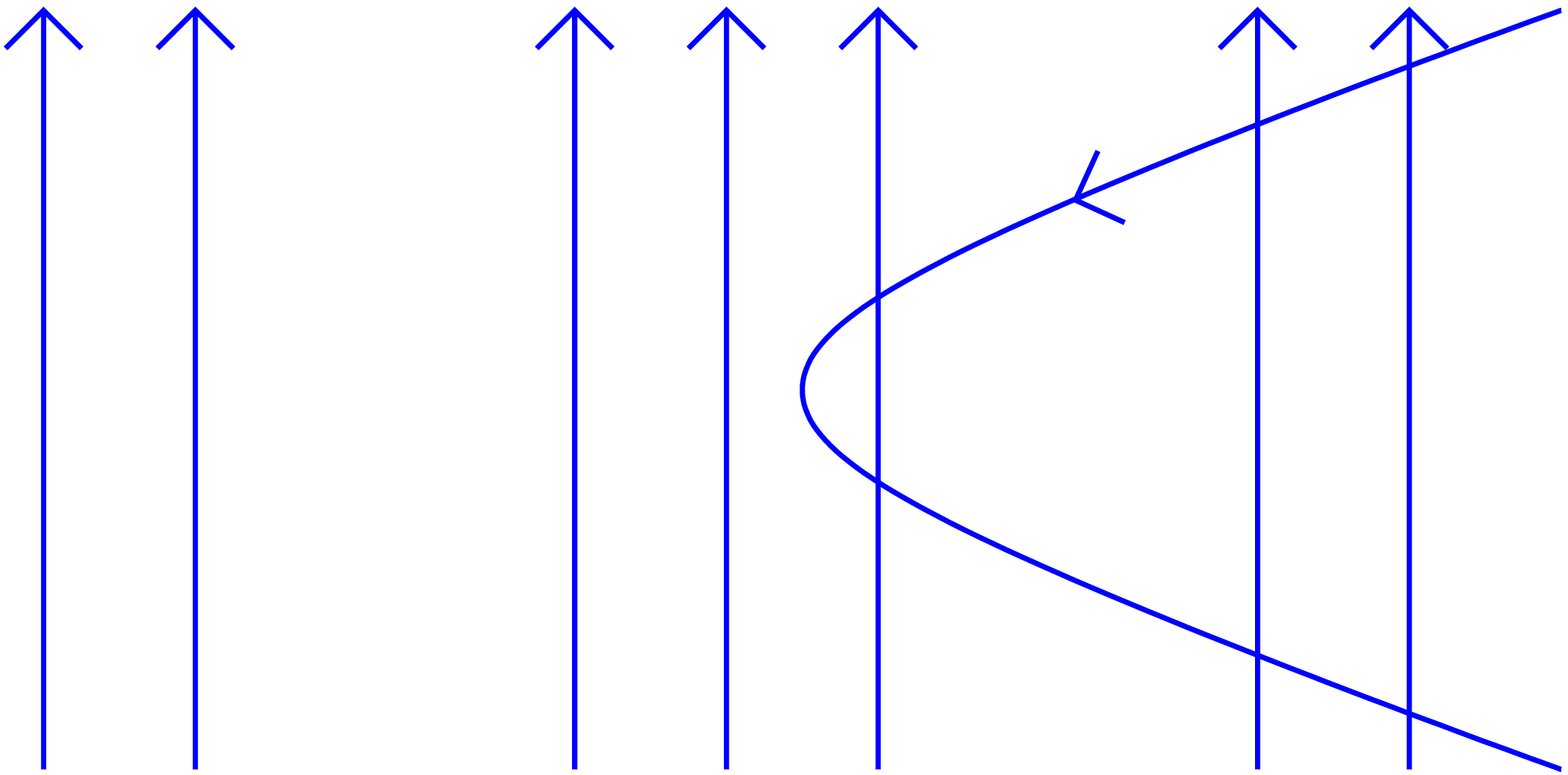}};
(-24,-13.5)*{\scriptstyle\blue n};
(-18,-13.5)*{\scriptstyle\blue n-1};
(-7,-13.5)*{\scriptstyle\blue i+1};
(-1,-13.5)*{\scriptstyle\blue i};
(3,-13.5)*{\scriptstyle\blue i-1};
(13,-13.5)*{\scriptstyle\blue 1};
(18,-13.5)*{\scriptstyle\blue n+1};
(23,-13.5)*{\scriptstyle\blue i};
(-24,13.5)*{\scriptstyle\blue n};
(-18,13.5)*{\scriptstyle\blue n-1};
(-7,13.5)*{\scriptstyle\blue i+1};
(-1,13.5)*{\scriptstyle\blue i};
(3,13.5)*{\scriptstyle\blue i-1};
(13,13.5)*{\scriptstyle\blue 1};
(18,13.5)*{\scriptstyle\blue n+1};
(23,13.5)*{\scriptstyle\blue i};
(-30,0)*{(1^n,0)};
(-12,10)*{\blue\cdots};
(-12,-10)*{\blue\cdots};
(8,10)*{\blue\cdots};
(8,-10)*{\blue\cdots};
\endxy;
$$
and finally Reidemeister 2 relations in order to straighten the 
downward $i$-strand.   

Finally, the fact that $\mathcal{I}_n$ preserves 
the relations~\eqref{EEDeltaEE} and~\eqref{EEDeltaEEPerm} can be easily 
proved by applying Reidemeister 2 and 3 relations to the images of the 
diagrams on the l.h.s. of those two relations. 
\end{proof}

\section{The Grothendieck group}
In this section we prove that $\Scat(n,n)$ categorifies $\hat{\SD}(n,n)$ 
(Theorem~\ref{thm:categorification}). All the hard work has been done already, 
we just have to put everything together.  
\begin{lem}
\label{lem:Extracatrels}
In $\Scat(n,n)$, we have  
\begin{enumerate}[i)]
\item\label{one}
$\mathcal{E}_{\pm\delta}{\mathbf 1}_{\lambda}\cong {\mathbf 1}_{\lambda}\mathcal{E}_{\pm\delta}\cong 0$ for all $\lambda\ne (1^n)$;
\item\label{two} $\mathcal{E}_{\pm\delta}{\mathbf 1}_n\cong {\mathbf 1}_n\mathcal{E}_{\pm\delta}$;
\item\label{three} $\mathcal{E}_{+\delta}\mathcal{E}_{-\delta}{\mathbf 1}_n\cong \mathcal{E}_{-\delta}\mathcal{E}_{+\delta}{\mathbf 1}_n\cong {\mathbf 1}_n$;
\item\label{four} $\mathcal{E}_i\mathcal{E}_{+\delta}{\mathbf 1}_n\cong \mathcal{E}_i^{(2)}\mathcal{E}_{i-1}\ldots \mathcal{E}_1\mathcal{E}_n\cdots \mathcal{E}_{i+1}{\mathbf 1}_n$;
\item\label{five} ${\mathbf 1}_n\mathcal{E}_{+\delta}\mathcal{E}_i\cong {\mathbf 1}_n\mathcal{E}_{i-1}\ldots \mathcal{E}_1\mathcal{E}_n\cdots \mathcal{E}_{i+1}\mathcal{E}_i^{(2)}$;
\item\label{six} $\mathcal{E}_{-i}\mathcal{E}_{+\delta}{\mathbf 1}_n\cong \mathcal{E}_{i-1}\cdots \mathcal{E}_1\mathcal{E}_n\cdots \mathcal{E}_{i+1}{\mathbf 1}_n$;
\item\label{seven} ${\mathbf 1}_n\mathcal{E}_{+\delta}\mathcal{E}_{-i}\cong {\mathbf 1}_n\mathcal{E}_{i-1}\cdots \mathcal{E}_1\mathcal{E}_n\cdots \mathcal{E}_{i+1}$;
\item\label{eight} $\mathcal{E}_{-i}\mathcal{E}_{-\delta}{\mathbf 1}_n\cong \mathcal{E}_{-i}^{(2)}\mathcal{E}_{-(i+1)}\cdots \mathcal{E}_{-n}\mathcal{E}_{-1}\cdots 
\mathcal{E}_{-(i-1)}{\mathbf 1}_n$;
\item\label{nine} ${\mathbf 1}_n\mathcal{E}_{-\delta}\mathcal{E}_{-i}\cong {\mathbf 1}_n\mathcal{E}_{-(i+1)}\cdots \mathcal{E}_{-n}\mathcal{E}_{-1}\cdots 
\mathcal{E}_{-(i-1)}\mathcal{E}_{-i}^{(2)}$;
\item\label{ten} $\mathcal{E}_i\mathcal{E}_{-\delta}{\mathbf 1}_n\cong \mathcal{E}_{-(i+1)}\cdots \mathcal{E}_{-n}\mathcal{E}_{-1}\cdots \mathcal{E}_{-(i-1)}{\mathbf 1}_n$;
\item\label{eleven} ${\mathbf 1}_n\mathcal{E}_{-\delta}\mathcal{E}_i\cong {\mathbf 1}_n\mathcal{E}_{-(i+1)}\cdots \mathcal{E}_{-n}\mathcal{E}_{-1}\cdots \mathcal{E}_{-(i-1)}$, 
\end{enumerate}
for any $i=1,\ldots,n$.
\end{lem}
\begin{proof}
The isomorphisms in~\eqref{one} and~\eqref{two} are immediate. 

For~\eqref{three}, consider the $2$-morphisms 
\begin{eqnarray*}
\txt{\xy
    (0,0)*{\black\bbpef{\black \delta}};
    (9,0)*{(1^n)};
    \endxy}
\colon& {\mathbf 1}_n\to \mathcal{E}_{-\delta}\mathcal{E}_{+\delta}{\mathbf 1}_n\\
\txt{\xy
    (0,0)*{\black\bbcef{\black \delta}};
    (9,0)*{(1^n)};
       \endxy} 
\colon &\mathcal{E}_{-\delta}\mathcal{E}_{+\delta}{\mathbf 1}_n\to {\mathbf 1}_n\\
\txt{\xy
    (0,0)*{\black\bbpfe{\black \delta}};
    (9,0)*{(1^n)};
        \endxy}
\colon&   {\mathbf 1}_n\to \mathcal{E}_{+\delta}\mathcal{E}_{-\delta}{\mathbf 1}_n\\
\txt{\xy    
(0,0)*{\black\bbcfe{\black \delta}};
    (9,0)*{(1^n)};
        \endxy}
\colon& \mathcal{E}_{+\delta}\mathcal{E}_{-\delta}{\mathbf 1}_n\to {\mathbf 1}_n.
\end{eqnarray*}
Relations~\eqref{deltabubbles} and~\eqref{deltainvert2} show that 
these $2$-morphisms are $2$-isomorphisms. 

Similarly, the isomorphisms in~\eqref{four} and~\eqref{five} follow from 
the relations in~\eqref{EDeltaEEEDelta} and~\eqref{DeltaEEEDeltaE}, and the 
isomorphisms in~\eqref{six} and~\eqref{seven} follow from 
the relations in~\eqref{DeltaFEEDeltaF} and~\eqref{EEFDeltaEE}.

The isomorphisms in~\eqref{eight}--\eqref{eleven} follow from the ones above 
by biadjointness. 
\end{proof}

Recall that $\mathrm{END}(X)$ denotes the ring generated by all homogeneous 
$2$-endomorphisms of a given $1$-morphism $X$, 
whereas $\mathrm{End}(X)\subset \mathrm{END}(X)$ only contains 
the ones of degree zero.  
\begin{lem}
\label{lem:END}
For any $t\in\mathbb{Z}$, 
$$
\mathrm{END}(\mathcal{E}_{+\delta}^t {\mathbf 1}_n)\cong 
1_{\mathcal{E}_{+\delta}^t}\mathrm{END}({\mathbf 1}_n)\cong \mathrm{END}({\mathbf 1}_n)1_{\mathcal{E}_{+\delta}^t}
$$
\end{lem}
\begin{proof}
Note that for $t=0$ there is nothing 
to prove. Let us now explain the proof for $t=1$. 
Given a diagram of the form 
$$
\figs{0.25}{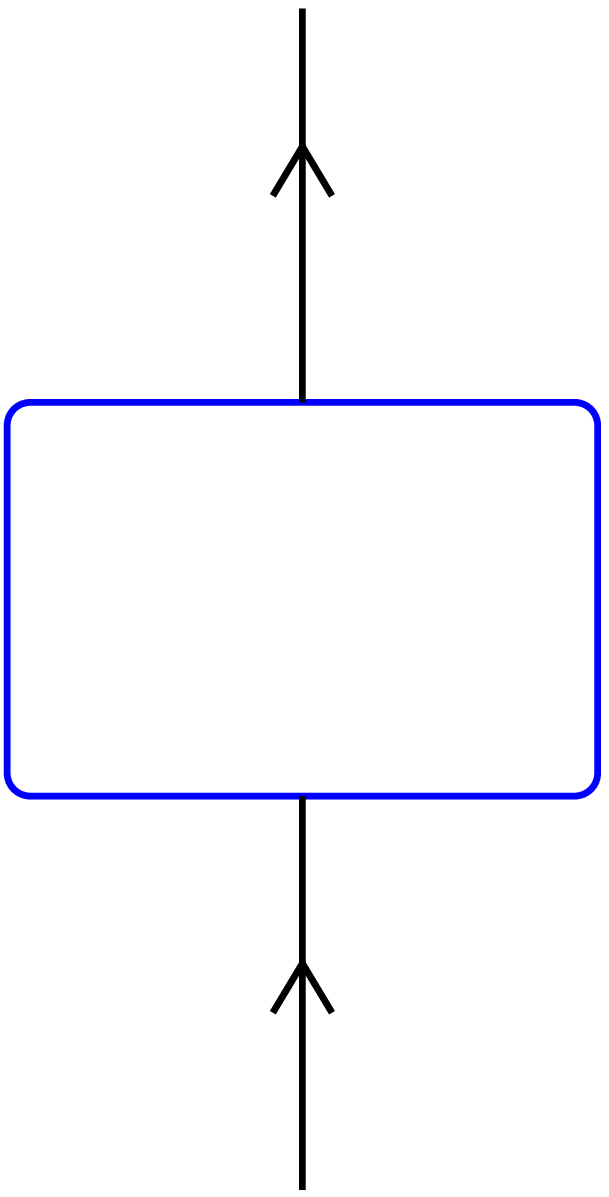},
$$      
we can create a $\delta$-bubble by~\eqref{deltabubbles} and 
apply~\eqref{deltainvert2} to obtain 
$$
\begin{array}{ccc}
\txt{\figs{0.25}{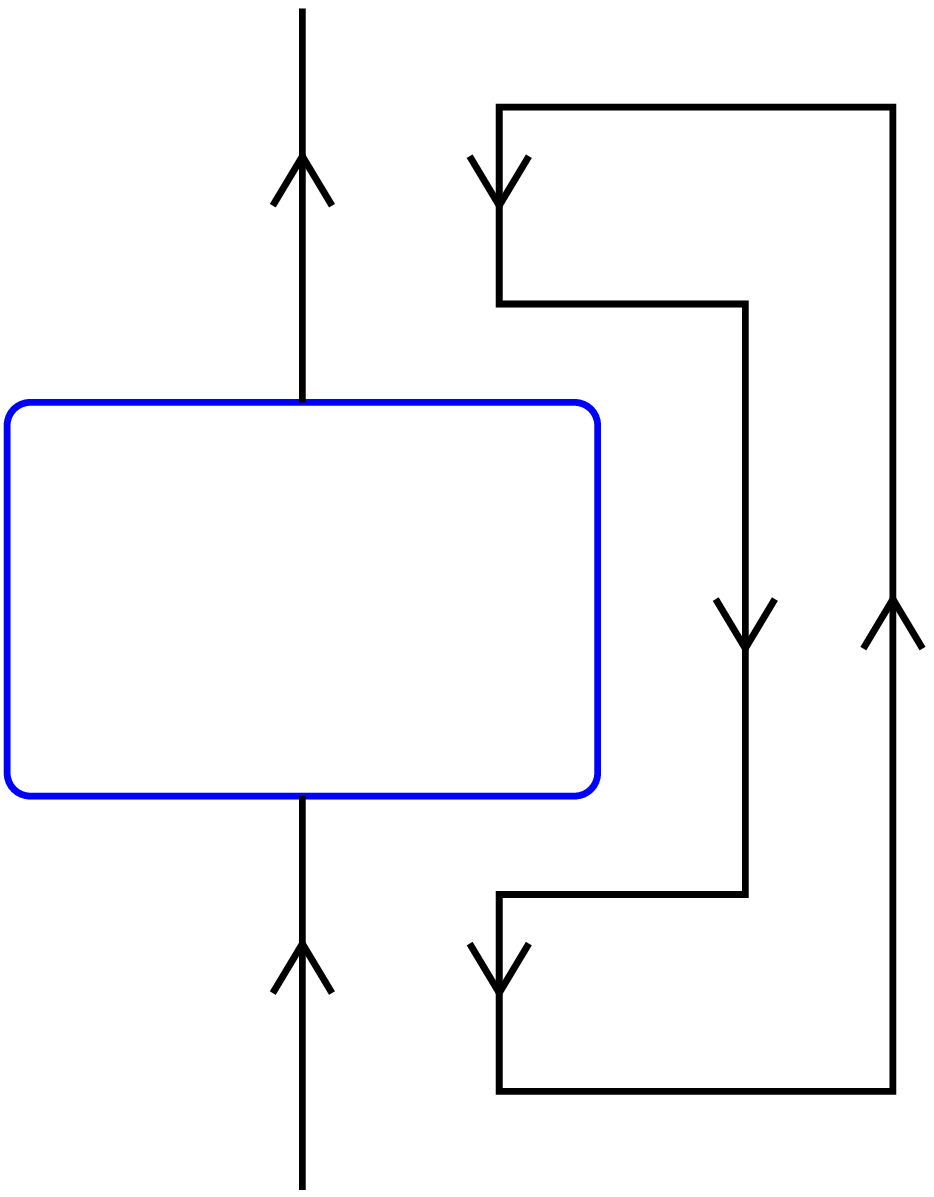}}&=&\txt{\figs{0.25}{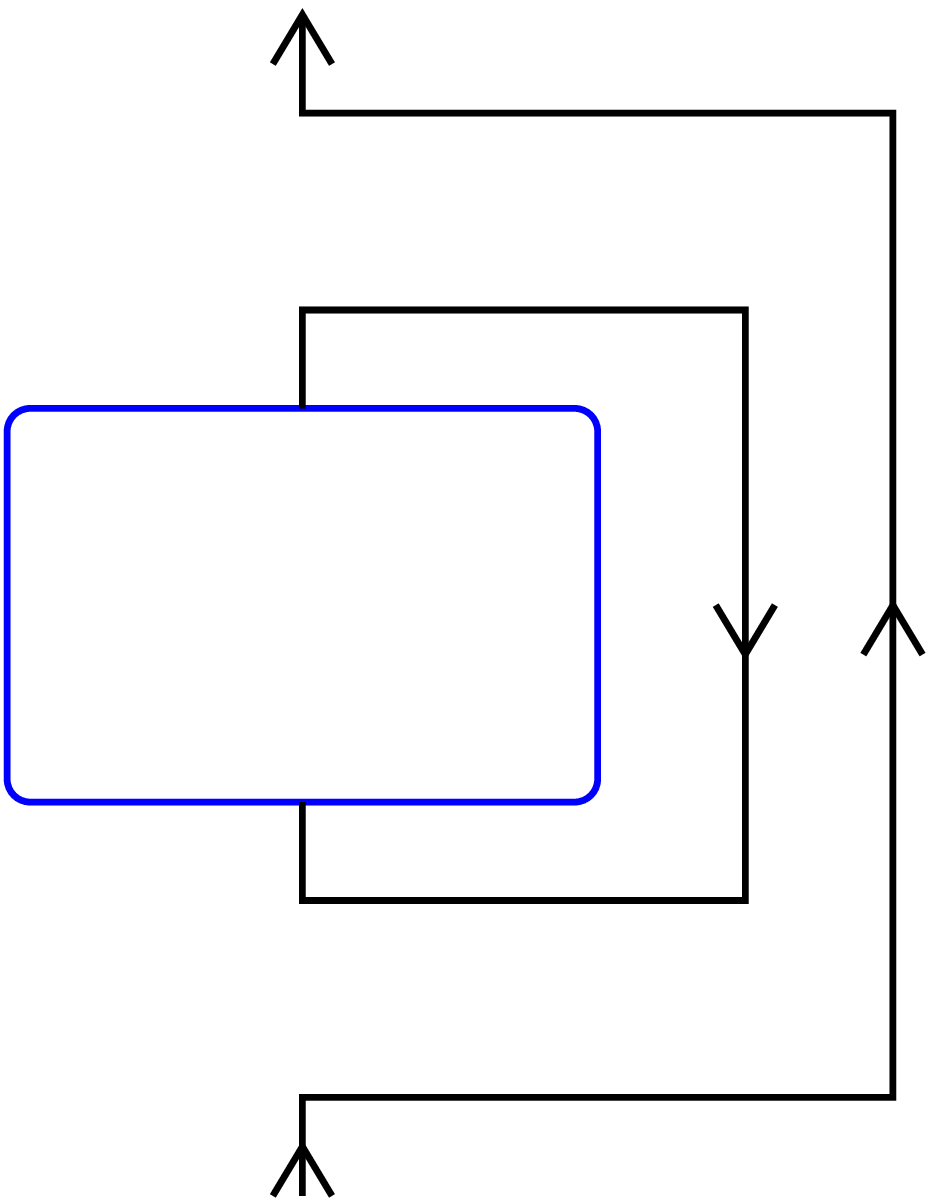}}.
\end{array}
$$  
This proves the lemma for $t=1$. For $t>1$ use the same trick repeatedly 
until you are left with a closed diagram and $t$ upward $\delta$-strands. 
For $t<0$, a similar trick can be applied using the opposite orientation 
on the $\delta$-strands.  
\end{proof}

Let $K_0(\mathrm{Kar}\Scat(n,n))$ be the split Grothendieck group of $\mathrm{Kar}\Scat(n,n)$. This 
is a $\mathbb{Z}[q,q^{-1}]$-module, where the action of $q$ is defined by 
$$q[X]:=[X\{1\}].$$
Furthermore, let 
$$K_0^{\mathbb{Q}(q)}(\mathrm{Kar}\Scat(n,n)):=K_0(\mathrm{Kar}\Scat(n,n))\otimes_{\mathbb{Z}[q,q^{-1}]}
\mathbb{Q}(q).
$$
\begin{defn}
Define the $\mathbb{Q}(q)$-linear 
algebra homomorphism $\gamma_n\colon \hat{\SD}(n,n)\to 
K_0^{\mathbb{Q}(q)}(\mathrm{Kar}\Scat(n,n))$ 
by 
$$\gamma_n(E_{\ii}1_{\lambda}):=[\mathcal{E}_{\ii}{\mathbf 1}_{\lambda}]\otimes 1\quad
\text{and}\quad \gamma_n(E_{+\delta}^t1_n):=[\mathcal{E}_{+\delta}^t {\mathbf 1}_n]\otimes 1$$
for any signed sequence $\ii$, $\lambda\in\Lambda(n,n)$ and $t\in \mathbb{Z}$.
\end{defn}

\begin{thm}
\label{thm:categorification}
The homomorphism $\gamma_n$ is well-defined and bijective. 
\end{thm}
\begin{proof}
Well-definedness follows from the corresponding statement for $\Ucataff$ 
by Khovanov and Lauda in~\cite{KL3} and from Lemma~\ref{lem:Extracatrels}.

Let us now show surjectivity. By Lemma~\ref{lem:Extracatrels}, 
any indecomposable object in $\mathrm{Kar}\Scat(n,n)$ is isomorphic to an object 
of the form $(X,e)$, where $X$ is either of the form $\mathcal{E}_{+\delta}^t$ 
for some $t\in\mathbb{Z}$ or of the form $\mathcal{E}_{\ii}$ for some signed 
sequence $\ii$, and $e$ is some idempotent in $\mathrm{End}(X)$. By 
Lemmas~\ref{lem:full} and~\ref{lem:END}, we see 
that $\mathrm{End}(\mathcal{E}_{+\delta}^t)\cong 
\mathbb{Q}1_{\mathcal{E}_{+\delta}^t}$. Therefore $\mathcal{E}_{+\delta}^t$ is 
indecomposable in $\mathrm{Kar}\Scat(n,n)$. Note that its Grothendieck class 
lies indeed in the image of $\gamma_n$. By Lemma~\ref{lem:full} we know that 
$\mathrm{End}_{\Scat(n,n)}(\mathcal{E}_{\ii})$ is the surjective image of the 
analogous endomorphism ring in $\Ucataff$, for 
any signed sequence $\ii$. By Khovanov and Lauda's Theorem 1.1~\cite{KL3} 
and some general arguments which were explained in detail in~\cite{MSVschur}, 
and also used in~\cite{MTh}, this implies that the Grothendieck classes of 
all direct summands of $\mathcal{E}_{\ii}$ in $\mathrm{Kar}\Scat(n,n)$ are contained in 
the image of $\gamma_n$. This concludes the proof that $\gamma_n$ is 
surjective. 

For injectivity, consider the following commutative diagram 
$$
\begin{CD}
\hat{\SD}(n,n) @>{\iota_n}>> \hat{\SD}(n+1,n)\\
@V{\gamma_n}VV  @VV{\gamma_{n+1}}V\\
K_0^{\mathbb{Q}(q)}(\mathrm{Kar}\Scat(n,n)) @>{K_0(\mathcal{I}_n)\otimes 1}>> 
K_0^{\mathbb{Q}(q)}(\mathrm{Kar}\Scat(n+1,n))
\end{CD},
$$
where $\gamma_{n+1}$ is the isomorphism from Theorem 6.4 in~\cite{MTh}. 
Since $\iota_n$ and $\gamma_{n+1}$ are both injective, 
their composite is also injective. The commutativity of the diagram above 
then implies that $\gamma_n$ is injective too. 
\end{proof}

\bibliographystyle{alpha}
\bibliography{biblioAFFnn}
\vspace{0.1in}
\noindent M.M.: { \sl \small Center for Mathematical Analysis, Geometry, and Dynamical Systems, Departamento de Matem\'{a}tica, Instituto Superior T\'{e}cnico, 
1049-001 Lisboa, Portugal; Departamento de Matem\'{a}tica, FCT, Universidade do Algarve, Campus de Gambelas, 8005-139 Faro, Portugal} 
\newline \noindent {\tt \small email: mmackaay@ualg.pt}

\noindent A.-L.T.: { \sl \small Matematiska Institutionen, 
Uppsala Universitet, 75106 Uppsala, Sverige} 
\newline \noindent {\tt \small email: anne-laure.thiel@math.uu.se}

\end{document}

%% file: KLdiags.tex
\newcommand{\lowrru}[1]{\xybox{%
  (-8,0)*{};
  (8,0)*{};
  (-6,-18)*{};(6,-9)*{} **\crv{(-6,-13) & (6,-15)} ?(1)*\dir{>};
  (6,-9)*{};(6,0)*{}  **\dir{-} ?(.3)*\dir{ }+(2,0)*{\scs {\bf j}};
}}

\newcommand{\lowllu}[1]{\xybox{%
  (-8,0)*{};
  (8,0)*{};
  (6,-18)*{};(-6,-9)*{} **\crv{(6,-13) & (-6,-15)} ?(1)*\dir{>};
  (-6,-9)*{};(-6,0)*{}  **\dir{-} ?(.3)*\dir{ }+(-2,0)*{\scs {\bf j}};
}}

\newcommand{\bbsid}{\xybox{%
  (-2,0)*{};
  (2,0)*{};
  (0,10);(0,4) **\dir{-};
}}
\newcommand{\bbpef}[1]{\xybox{%
  (-6,0)*{};
  (6,0)*{};
  (-4,0)*{}="t1";
  (4,0)*{}="t2";
  "t1";"t2" **\crv{(-4,-6) & (4,-6)}; ?(.15)*\dir{>} ?(.9)*\dir{>}
   ?(.5)*\dir{}+(0,-2)*{\scriptstyle{#1}};
}}
\newcommand{\bbpfe}[1]{\xybox{%
  (-6,0)*{};
  (6,0)*{};
  (-4,0)*{}="t1";
  (4,0)*{}="t2";
  "t2";"t1" **\crv{(4,-6) & (-4,-6)}; ?(.15)*\dir{>} ?(.9)*\dir{>}
  ?(.5)*\dir{}+(0,-2)*{\scriptstyle{#1}};
}}

\newcommand{\bbcfe}[1]{\xybox{%
  (-6,0)*{};
  (6,0)*{};
  (-4,0)*{}="t1";
  (4,0)*{}="t2";
  "t1";"t2" **\crv{(-4,6) & (4,6)}; ?(.15)*\dir{>} ?(.9)*\dir{>}
  ?(.5)*\dir{}+(0,2)*{\scriptstyle{#1}};
}}
\newcommand{\bbcef}[1]{\xybox{%
  (-6,0)*{};
  (6,0)*{};
  (-4,0)*{}="t1";
  (4,0)*{}="t2";
  "t2";"t1" **\crv{(4,6) & (-4,6)}; ?(.15)*\dir{>}
  ?(.9)*\dir{>} ?(.5)*\dir{}+(0,2)*{\scriptstyle{#1}};
}}

\newcommand{\ccbub}[2]{
\xybox{%
 (-6,0)*{};
  (6,0)*{};
  (-4,0)*{}="t1";
  (4,0)*{}="t2";
  "t2";"t1" **\crv{(4,6) & (-4,6)}; ?(.7)*\dir{}+(-2,0)*{\scs #2}
  ?(.05)*\dir{>} ?(1)*\dir{>};
  "t2";"t1" **\crv{(4,-6) & (-4,-6)};
   ?(.3)*\dir{}+(0,0)*{\bullet}+(0,-3)*{\scs {#1}};
}}
\newcommand{\cbub}[2]{
\xybox{%
 (-6,0)*{};
  (6,0)*{};
  (-4,0)*{}="t1";
  (4,0)*{}="t2";
  "t2";"t1" **\crv{(4,6) & (-4,6)};?(.7)*\dir{}+(-2,0)*{\scs #2};
   ?(0)*\dir{<} ?(.95)*\dir{<};
  "t2";"t1" **\crv{(4,-6) & (-4,-6)};
   ?(.3)*\dir{}+(0,0)*{\bullet}+(0,-3)*{\scs {#1}};
}}
\newcommand{\nccbub}{
\xybox{%
 (-6,0)*{};
  (6,0)*{};
  (-4,0)*{}="t1";
  (4,0)*{}="t2";
  "t2";"t1" **\crv{(4,6) & (-4,6)}; ?(.05)*\dir{>} ?(1)*\dir{>};
  "t2";"t1" **\crv{(4,-6) & (-4,-6)}; ?(.3)*\dir{};
}}
\newcommand{\ncbub}{
\xybox{%
 (-6,0)*{};
  (6,0)*{};
  (-4,0)*{}="t1";
  (4,0)*{}="t2";
  "t2";"t1" **\crv{(4,6) & (-4,6)}; ?(0)*\dir{<} ?(.95)*\dir{<};
  "t2";"t1" **\crv{(4,-6) & (-4,-6)}; ?(.3)*\dir{};
}}

\newcommand{\bbdl}[1]{\xybox{%
  (2,0);(0,-8) **\crv{(2,-2)&(0,-6)}; ?(.5)*\dir{>}
}}
\newcommand{\bbdlu}[1]{\xybox{%
  (2,0);(0,-8) **\crv{(2,-2)&(0,-6)}; ?(.5)*\dir{<}
}}
\newcommand{\bbdr}[1]{\xybox{%
  (-2,0);(0,-8) **\crv{(-2,-2)&(0,-6)}; ?(.5)*\dir{>}
}}
\newcommand{\bbdru}[1]{\xybox{%
  (-2,0);(0,-8) **\crv{(-2,-2)&(0,-6)}; ?(.5)*\dir{<}
}}

%% file: KLdefs.tex
\newcommand{\refequal}[1]{\xy {\ar@{=}^{#1}
(-1,0)*{};(1,0)*{}};
\endxy}

\newcommand{\SD}{{\bf S}}

\newcommand{\Uaff}{\dot{{\bf U}}(\hat{\mathfrak{sl}}_n)}
\newcommand{\Uglaff}{\dot{{\bf U}}(\hat{\mathfrak{gl}}_n)}
\newcommand{\Uglaffext}{\hat{\dot{{\bf U}}}(\hat{\mathfrak{gl}}_n)}

\newcommand{\Ucataff}{\cal{U}(\widehat{\mathfrak sl}_n)}
\newcommand{\Scat}{\hat{\mathcal{S}}}

\newcommand{\Ucupr}{\xy (-2,1)*{}; (2,1)*{} **\crv{(-2,-3) & (2,-3)} ?(1)*\dir{>}; \endxy}

\newcommand{\Ucupl}{\xy (2,1)*{}; (-2,1)*{} **\crv{(2,-3) & (-2,-3)}?(1)*\dir{>};\endxy}
\newcommand{\Ucupli}{\xy (0,-1)*{\dblue\xybox{(2,1)*{}; (-2,1)*{} **\crv{(2,-3) & (-2,-3)}?(1)*\dir{>};}}\endxy}
\newcommand{\Ucapr}{\xy (-2,-1)*{}; (2,-1)*{} **\crv{(-2,3) & (2,3)}?(1)*\dir{>};\endxy\;\;}

\newcommand{\Ucapl}{\xy (2,-1)*{}; (-2,-1)*{} **\crv{(2,3) &(-2,3) }?(1)*\dir{>};\endxy\;}
\newcommand{\Ucapli}{\xy (0,1)*{\dblue\xybox{(2,-1)*{}; (-2,-1)*{} **\crv{(2,3) &(-2,3) }?(1)*\dir{>};}}\endxy\;}

\newcommand{\EEDelta}{\xy (0,0)*{\figs{0.1}{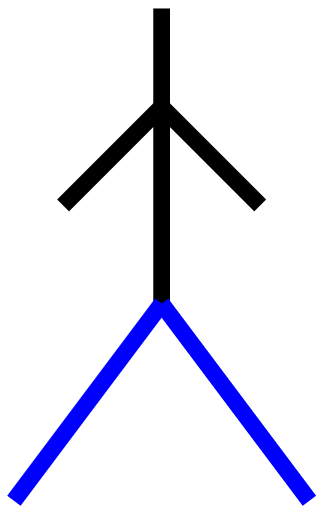}}\endxy}
\newcommand{\FFDelta}{\xy (0,0)*{\figs{0.1}{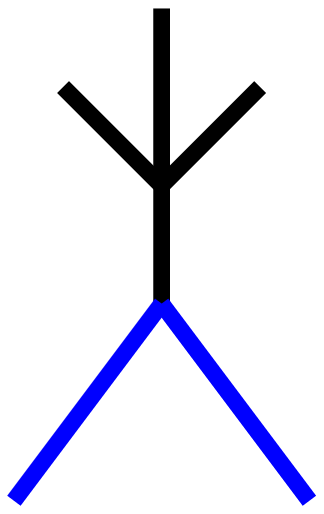}}\endxy}
\newcommand{\DeltaEE}{\xy (0,0)*{\figs{0.1}{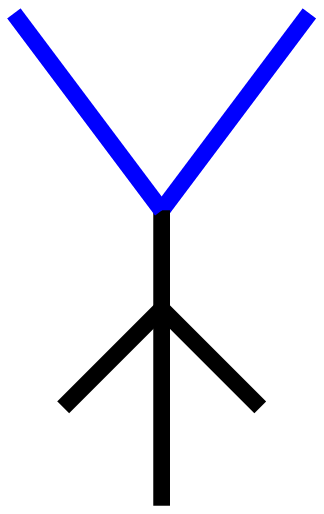}}\endxy}
\newcommand{\DeltaFF}{\xy (0,0)*{\figs{0.1}{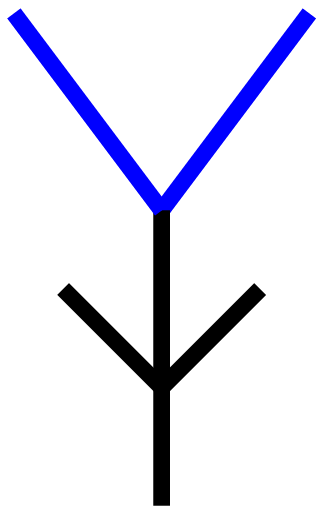}}\endxy}

\newcommand{\BOX}{\hbox {$\sqcap$ \kern -1em $\sqcup$}}

\renewcommand{\to}{\rightarrow}

\newcommand{\scs}{\scriptstyle}

\theoremstyle{definition}
\newtheorem{thm}{Theorem}[section]
\newtheorem{cor}[thm]{Corollary}

\newtheorem{lem}[thm]{Lemma}
\newtheorem{rem}[thm]{Remark}
\newtheorem{prop}[thm]{Proposition}
\newtheorem{defn}[thm]{Definition}
\newtheorem{notat}[thm]{Notation}

        \newcommand{\be}{\begin{equation}}
        \newcommand{\ee}{\end{equation}}
        \newcommand{\ba}{\begin{eqnarray}}
        \newcommand{\ea}{\end{eqnarray}}
        \newcommand{\ban}{\begin{eqnarray*}}
        \newcommand{\ean}{\end{eqnarray*}}
        \newcommand{\barr}{\begin{array}}
        \newcommand{\earr}{\end{array}}

\numberwithin{equation}{section}

\def\emph#1{{\sl #1\/}}

\let\hat=\widehat

\let\phi=\varphi
\let\epsilon=\varepsilon

\usepackage{bbm}

\def\Q{{\mathbbm Q}}

\def\cal#1{\mathcal{#1}}%
\def\1{\mathbbm{1}}%
\def\shuffle{\,\raise 1pt\hbox{$\scriptscriptstyle\cup{\mskip
               -4mu}\cup$}\,}

%% file: defs.tex
\newrgbcolor{dblue}{0 0 0.5}
\newrgbcolor{dred}{0.9 0 0}
\newrgbcolor{dgreen}{0 0.6 0}

\newcommand{\figs}[2] 
{{\includegraphics[scale =#1]{#2}}}

\newcommand{\bbox}[1]{\framebox{$\scs #1$}}

\newcommand{\ii}{\underline{\textbf{\textit{i}}}}
\newcommand{\jj}{\underline{\textbf{\textit{j}}}}

\newcommand{\bZ}{\mathbb{Z}}
\newcommand{\bQ}{\mathbb{Q}}